\documentclass[a4paper,reqno,twoside]{article}
\usepackage[T1]{fontenc}
\usepackage{authblk}
\pdfoutput=1
\usepackage{amsmath}
\usepackage{amsfonts,amsthm,amssymb}
\usepackage{mathabx}
\usepackage{times}
\usepackage{bm}
\usepackage{amscd}
\usepackage[latin2]{inputenc}
\usepackage{t1enc}
\usepackage[mathscr]{eucal}
\usepackage{indentfirst}
\usepackage{graphicx}
\usepackage{graphics}
\usepackage{pict2e}
\usepackage{epic}
\numberwithin{equation}{section}
\usepackage[margin=1in]{geometry}
\usepackage{epstopdf} 
\usepackage[colorlinks,linkcolor=blue]{hyperref}
\usepackage[capitalise,noabbrev]{cleveref}

\usepackage{fancyhdr}

\crefformat{equation}{(#2#1#3)}
\crefmultiformat{equation}{(#2#1#3)}{ and~(#2#1#3)}{, (#2#1#3)}{ and~(#2#1#3)}
\crefrangeformat{equation}{(#3#1#4) to~(#5#2#6)}

\usepackage[normalem]{ulem}

\allowdisplaybreaks

\theoremstyle{plain}
\newtheorem{Thm}{Theorem}[section]
\newtheorem*{Thm*}{Theorem}
\newtheorem{Lem}[Thm]{Lemma}

\newtheorem{Prop}[Thm]{Proposition}

\theoremstyle{definition}

\newtheorem{Rem}[Thm]{Remark}
\newtheorem{?}[Thm]{Problem}

\newcommand{\p}{\partial}
\newcommand{\R}{\mathbb{R}}
\newcommand{\e}{\varepsilon}

\newcommand{\E}{\mathbf{e}}

\newcommand{\fv}{\mathbf{f}}

\newcommand{\Torus}{\mathbb{T}}

\newcommand{\dv}{\text{div}}
\newcommand{\od}{\flat}
\newcommand{\md}{\sharp}
\newcommand{\lap}{\triangle}
\newcommand{\dnab}{\cdot \nabla}

\newcommand{\rhob}{\bar{\rho}}
\newcommand{\rhos}{\rho^s}
\newcommand{\rhot}{\tilde{\rho}}
\newcommand{\ub}{\bar{u}}
\newcommand{\mb}{\bar{m}}
\newcommand{\us}{u^s}
\newcommand{\ut}{\tilde{u}}
\newcommand{\uv}{\mathbf{u}}
\newcommand{\mv}{\mathbf{m}}
\newcommand{\wv}{\mathbf{w}}
\newcommand{\zv}{\mathbf{z}}
\newcommand{\uvb}{\bar{\mathbf{u}}}
\newcommand{\mvb}{\bar{\mathbf{m}}}
\newcommand{\uvt}{\tilde{\mathbf{u}}}
\newcommand{\mvt}{\tilde{\mathbf{m}}}
\newcommand{\mvs}{\mathbf{m}^s}

\newcommand{\mt}{\tilde{m}}
\newcommand{\ms}{m^s}
\newcommand{\mut}{\tilde{\mu}}
\newcommand{\xp}{x^{\perp}}

\newcommand{\Fv}{\mathbf{F}}
\newcommand{\gv}{\mathbf{g}}

\newcommand{\phib}{\bar{\phi}}
\newcommand{\psib}{\bar{\psi}}

\newcommand{\etat}{\tilde{\eta}}
\newcommand{\rv}{\mathbf{r}}

\newcommand{\abso}[1]{\left\lvert#1\right\rvert}

\newcommand{\norm}[1]{\left\lVert#1\right\rVert}
\newcommand{\jump}[1]{\left\ldbrack#1\right\rdbrack}

\date{}

\title{Time-asymptotic stability of planar Navier-Stokes shocks with spatial oscillations}

\author{Qian Yuan\thanks{The research is partially supported by the National Natural Science Foundation of China 12201614, Youth Innovation Promotion Association of CAS 2022003 and CAS Project for Young Scientists in Basic Research YSBR-031.
	}
}

\affil{\footnotesize Academy of Mathematics and Systems Science, Chinese Academy of Sciences, Beijing, China. \\
	qyuan@amss.ac.cn}

\begin{document}
	
\cfoot{} 

\clearpage\maketitle
\thispagestyle{empty}

\begin{abstract}
%	The stability of multi-dimensional shock waves is an important issue in the theory of gas dynamics. 

	This paper shows that for the three-dimensional compressible isentropic Navier-Stokes equations, the planar viscous shocks are time-asymptotically stable to suitably small initial perturbations with zero masses. 
	In particular, the perturbations consist of not only $ H^3 $-perturbations, but also periodic ones that oscillate at spatial infinity. In the former case, the final shock locations can be predicted in terms of the initial conditions, while in the latter the locations are subject to the dynamics of the oscillations.
	The stability analysis is based on the $L^2$-energy method. The key point is that the bad effect due to the compression of the shock waves can be removed by a combination of an anti-derivative technique and the use of Poincar\'{e} inequality in the normal and transversal directions, respectively.
\end{abstract}


\vspace{.2cm}


%\tableofcontents

%%%%%%%%%%%%%%%%%%%%%%%%%%%%%%%%%%%%%%%%%%%
%% Introduction
%%%%%%%%%%%%%%%%%%%%%%%%%%%%%%%%%%%%%%%%%%%

\section{Introduction}

The three-dimensional (3D) compressible isentropic Navier-Stokes (NS) equations read,
\begin{equation}\label{NS}
\begin{cases}
\p_t \rho + \dv \mv = 0, \\
\p_t \mv + \dv \big( \frac{\mv\otimes\mv}{\rho} \big) + \nabla p(\rho) = \mu \lap \uv + \big(\mu+\lambda \big) \nabla \dv \uv,
\end{cases} \quad x\in\R^3, t>0, 
\end{equation}
where $ \rho(x,t)>0 $ is the density, $ \mv(x,t) = (m_1,m_2,m_3)(x,t) = (\rho \uv)(x,t) \in\R^3 $ is the momentum with $ \uv(x,t) = \big( u_1,u_2,u_3 \big)(x,t) $ being the velocity, the pressure $ p(\rho) $ is the gamma law, satisfying $ p(\rho) = \rho^\gamma $ with $ \gamma> 1, $ and the viscous coefficients  $ \mu $ and $ \lambda $ satisfy
\begin{equation}\label{viscous}
	\mu > 0 \quad \text{and} \quad \mu + \lambda \geq 0.
\end{equation}
Note that \cref{viscous} is a relaxed version of the physical condition that $ \mu >0 $ and $ \frac{2}{3} \mu + \lambda \geq 0. $ 

\vspace{.2cm}

When $ \mu = \lambda =0, $  the 3D compressible isentropic Euler equations \cref{NS} admit rich wave phenomena such as shock waves, rarefaction waves and compressible vortex sheets.
A Lax shock wave, $ (\rhob^s,\mb_1^s) = (\rhob^s, \mb_1^s)(x_1,t), $ is a weak entropy solution of the 1D compressible isentropic Euler equations, satisfying
\begin{equation}\label{Euler}
\begin{cases}
\p_t \rhob^s + \p_1  \mb_1^s = 0, & \\
\p_t \mb_1^s + \p_1  \big( \frac{(\mb^s_1)^2}{\rhob^s} \big) + \p_1  p(\rhob^s) = 0, & 
\end{cases} \qquad x_1\in\R, t>0,
\end{equation}
and
\begin{equation*}
	\big( \rhob^s, \mb^s_1 \big)(x_1 ,0) = \begin{cases}
		\big( \rhob_-, \mb_{1-} \big), \quad x_1 <0, & \\
		\big( \rhob_+, \mb_{1+} \big), \quad x_1 >0, &
	\end{cases}
\end{equation*}
where $ (\rhob_\pm, \mb_{1\pm}) $ are two constant states with $ \rhob_\pm>0 $, satisfying the Rankine-Hugoniot conditions,
\begin{equation}\label{RH}
	\begin{cases}
		-s (\rhob_+-\rhob_-) + \mb_{1+} - \mb_{1-} = 0, \\
		-s (\mb_{1+} - \mb_{1-}) + \frac{\mb_{1+}^2}{\rhob_+} - \frac{\mb_{1-}^2}{\rhob_-} + p(\rhob_+) - p(\rhob_-) = 0.
	\end{cases}
\end{equation}
The system \cref{Euler} possesses two eigenvalues $ \lambda_\pm(\rho, u_1) = u_1 \pm \sqrt{p'(\rho)} $. We assume that the shock $ (\rhob^s, \mb_1^s) $ belongs to the second family (the other case is similar) and satisfies the Lax's entropy condition,
\begin{equation}\label{entropy}
	\lambda_+(\rhob_+, \ub_{1+}) < s < \lambda_+(\rhob_-, \ub_{1-}) \quad \text{and} \quad s > \lambda_-(\rhob_-, \ub_{1-}).
\end{equation}
Then it follows from \cref{RH,entropy} (see \cite{Smoller1994}) that
\begin{equation*}
	\rhob_->\rhob_+ \quad \text{and} \quad \ub_{1-} > \ub_{1+}.
\end{equation*}
With viscosity effect, the Lax shock wave is smoothed out to be a smooth viscous shock, $ (\rho^s, m_1^s) = (\rho^s, m_1^s)(x_1-st), $ which is a traveling wave solution of the 1D compressible isentropic NS equations, satisfying
\begin{equation}\label{ode-shock}
	\begin{cases}
		-s (\rhos)' + (\ms_1)' = 0, \\
		-s (\ms_1)' + (\rhos (\us_1)^2 + p(\rhos))' = \mut (\us_1)'', \\
		(\rhos, \ms_1)(x_1) \to (\rhob_\pm, \mb_{1\pm}) \qquad \text{as } \ x_1 \to \pm \infty.
	\end{cases}
\end{equation}
Here $ \us_1 := \frac{\ms_1}{\rhos} $ and $ \mut := 2\mu + \lambda >0. $
In three dimensions, the planar viscous shock
\begin{equation}\label{planar-shock}
	(\rhos,\mvs)(x_1-st) = (\rhos, \ms_1, 0, 0)(x_1-st) 
\end{equation}
is a traveling wave solution of \cref{NS}, propagating along the $ x_1 $-axis with the shock speed $ s. $ 
This paper considers 
%\begin{equation}\label{ic-p}
%	(\rho, \mv)(x,t=0) = (\rhos, \mvs)(x_1) + (v_0, \wv_0 )(x), \quad x \in \R^3,
%\end{equation}
\begin{equation}\label{ic}
	(\rho, \mv)(x,t=0) = \big(\rhos, \mvs\big)(x_1) + (v_0, \wv_0 )(x) + (\phib_0, \psib_0)(x), \quad x \in \R^3,
\end{equation}
where the perturbation $ (v_0, \wv_0)(x) =  (v_0, w_{01}, w_{02}, w_{03} )(x) \in \R \times \R^3 $ is periodic over the 3D torus $ \Torus^3 = [0,1]^3 $ and of zero average,
\begin{equation}\label{zero-ave}
	\int_{\Torus^3} (v_0, \wv_0)(x) dx = 0,
\end{equation}  
and $ (\phib_0, \psib_0) = (\phib_0, \psib_{01}, \psib_{02}, \psib_{03}) $ is periodic in the transverse variables $ \xp $ and belongs to the $ H^3(\R \times \Torus^2) $ space. That is, we consider not only the localized perturbations on the infinitely duct $ \R\times\Torus^2, $ but also periodic ones that oscillate along the wave propagation.

\vspace{.2cm}

The study of shock waves is an important issue in the theory of both inviscid and viscous conservation laws, such as the compressible Euler  and Navier-Stokes equations in gas dynamics. 
We first overview several remarkable works in one dimension. 
The shock stability for hyperbolic conservation laws has been done by many articles; see \cite{Ilin1960,Liu1977} for instance. 
In the viscous case, based on a comparison principle, the study was initiated by \cite{Ilin1960} for the scalar conservation laws. Then \cite{G1986,MN} extended the stability results of viscous shock profiles to the systems by using the $ L^2 $-energy method, provided that the initial perturbations are of zero masses. 
The case for generic perturbations was solved by introducing additional diffusion waves and coupled diffusion waves propagating along the transversal characteristics; see \cite{Liu1997,SX} for the uniformly parabolic systems.
Later, \cite{MZ2004,LZ2015} made more subtle analysis to study some physical models including the compressible NS equations. 
%We remark here that in the 1D case, the zero-mass stability (e.g. \cite{G1986,MN}) can indeed imply the spectral stability.

The shock theory in multiple dimensions are way more limited than the 1D case. For the multi-d hyperbolic conservation laws, the stability of planar inviscid Lax shocks was established by \cite{Majda1983s} in local time, where the perturbed shock waves possess piece-wise smooth structures. If the solution space of the compressible Euler equations is relaxed to the $ L^\infty $-space, then \cite{CDK2015, MK2018} showed the non-uniqueness of the admissible solutions, in which the initial data are some Riemann data containing at least one planar Lax shock. We also refer to \cite{Chio2014} for the non-uniqueness of spatially periodic admissible solutions for the compressible Euler equations. 
%Hence, the shock theory in multiple dimensions is essentially different from that in one dimension.
The viscous shock theory for general multi-d viscous conservation laws is also unsatisfactory. 
Nevertheless, the nonlinear stability of planar viscous shocks was proved by \cite{GM1999,HZ2000} in the scalar case, based on detailed point-wise estimates of Green's functions; and see \cite{KVW2019} for the $ L^2 $-contraction property of planar viscous shocks.
For the multi-d systems, a linear or spectral stability of planar viscous shocks was established through a spectral method; see \cite{Z,FS2010,HLZ2017} and the references therein for instance. 
%In particular, \cite{HLZ2017} investigated a so-called strong spectral stability condition, which is sufficient for the nonlinear stability of planar viscous shocks. 
It should be mentioned that in multiple dimensions, the zero-mass condition does not imply the spectral one.
We also refer to the recent paper \cite{WW2022} for a nonlinear stability of planar viscous shocks in an infinitely duct $ \R\times\Torus^2 $, which will be discussed in \cref{Rem-loc} for more details.

\vspace{.1cm}

The study of spatially periodic solutions is important and interesting in the theory of conservation laws. For instance, the periodic solutions of hyperbolic conservation laws lose bounded total variations and are not able to be constructed through the celebrated Glimm scheme \cite{Glimm1965}. In that respect, the genuinely nonlinear conservation laws of at most two equations are exceptional; see \cite{Lax1957,Glimm1970,Dafe1995,Dafe2013}. 
However, the periodic oscillations for the non-isentropic compressible Euler equations possess totally different nonlinear effect from the isentropic case. The recent celebrated paper \cite{TY2023} showed that the spatially periodic solutions do not decay to constant states as time goes to infinity in general.
We also refer to \cite{TY1996,QX2015} for a large time existence and to \cite{MR1984,HMR1986} for the study of a resonance phenomenon of periodic solutions, respectively. 
Besides the constant states, the periodic perturbations of general Riemann solutions are also studied by recent works \cite{XYYsiam,YY2019,XYYindi,HY1shock,YY2022,HXY2022,Yuan2021}. 
It is found that different from the localized perturbations (e.g. \cite{Ilin1960}), the periodic oscillations at spatial infinity have a new  effect on the asymptotic orbital stability of shock waves. 
In particular, for the viscous conservation laws, the final shock locations are subject to the dynamics of the periodic oscillations; see \cite{XYYindi} for more details. 
In this regard, the study of spatial oscillations is important and interesting in the shock theory. 

\vspace{.1cm}

This paper, extending the 1D shock stability results of \cite{MN,G1986}, establishes the time-asymptotic stability of planar Navier-Stokes shocks in the three dimensions, in which the perturbations are allowed to oscillate at spatial infinity. 
For convenience, we assume a zero-mass condition as in \cite{G1986,MN,HY1shock} (see \cref{zero-mass-con}), while the case for generic perturbations which generate additional diffusion waves along the transversal characteristics will be studied in future work. 
The stability analysis about the Cauchy problem \cref{NS}, \cref{ic} is based on the $ L^2 $-energy method. However, in comparison with the localized perturbations and the 1D shock theory, there are several new difficulties:
\begin{itemize}
	\item[i) ] For positive time the perturbation is no longer periodic and not $ L^2 $-integrable along wave propagation. Thus, we construct a multi-dimensional ansatz $ (\rhot, \mvt) $ with two carefully chosen time-dependent shift curves, such that the ansatz oscillates together with the solution $ (\rho,\mv) $ at spatial infinity, and satisfies the zero mass property,
	\begin{equation}\label{intr-zero-mass}
		\int_{\R\times\Torus^2} ((\rho, m_1) - (\rhot, \mt_1))(x,t) dx = 0 \qquad \forall t\geq 0.
	\end{equation}
	This is a necessary condition in the use of anti-derivative technique and can also help determine the final shock location. 
	
	\item[ii) ] The compression of shock waves has a bad effect in the energy estimates. In one dimension, this obstacle can be overcome by using the anti-derivative technique (\cite{Ilin1960,MN,G1986}). However, this method is so far not that effective in multiple dimensions. In this paper, we follow the idea of \cite{Yuan2021} to decompose the perturbation into a zero mode and non-zero mode, respectively. Then the zero-mode estimates can be established with the aid of the 1D anti-derivative technique. Besides, it is an important observation that the shock compression does not influence the non-zero mode estimates with the aid of the Poincar\'{e} inequality in the transversal directions $ \Torus^2. $
	
	\item[iii) ] In the study of Navier-Stokes shocks, the inherent structures of Lagrangian and Eulerian coordinate systems are very different from each other even in one dimension. 
	The former has been used thoroughly by the previous works to study the 1D shock stability.
	In this paper, we find out new cancellations to overcome the obstacles arising from the Eulerian coordinates, which is owing essentially to the smallness of the velocity fields (the smallness is without loss of generality for weak shocks). We also refer to \cite{WW2022} for a different approach to deal with the difficulties.
\end{itemize}

\vspace{0.3cm}

\textbf{Outline of the paper}. In the next section, we show the construction of the ansatz and present the main result of this paper, \cref{Thm}. In \cref{Sec-pre}, we give some useful lemmas. After the preliminaries, we show the proof of the a priori estimates and the local existence in Sections \ref{Sec-esti} and \ref{Sec-local}, respectively. Afterwards, the time-asymptotic behaviors of the solution, together with the exponential decay rates of the non-zero modes, are shown in \cref{Sec-stability}.

\vspace{.3cm}

\textbf{Notations}. Throughout this paper, we use the following notations. 

1) Let $ \Omega $ denote the infinitely long nozzle domain,
\begin{equation*}
	\Omega := \R \times \Torus^2.
\end{equation*}

2) Let $ \jump{\cdot} $ denote the difference of the constant states associated with the shock wave, $ (\rhob^s, \mb_1^s) $. For instance, $ \jump{\rho} = \rhob_+ - \rhob_- $,  $ \jump{p} = p(\rhob_+) - p(\rhob_-) $ and $ \jump{u_1 m_1} = \ub_{1+} \mb_{1+} - \ub_{1-} \mb_{1-}$, etc.

3) Let $ C \geq 1 $ be a generic constant. And we write $ A \lesssim B $, $ A \gtrsim B $, $ A \sim B $ and $ A = O(1) B $ when $ A \leq CB $, $ A \geq C^{-1} B $, $ C^{-1} A \leq B \leq CA $ and $ \abso{A} \leq C \abso{B} $ hold, respectively.

4) For any given 1D function $ f=f(x_1) $, we use the lower index $ f_d $ to denote a shift operator, namely,
\begin{equation*}
	f_{d}(x_1):= f(x_1-d(t)) \quad \text{for any curve } \ d = d(t).
\end{equation*}

5) 
Let $ \E_i \ (i=1,2,3) $  denote the $ i $-th column of the $ 3\times 3 $ identity matrix Id$ _{3\times3}, $ and $ \delta_{ij} \ (i,j=1,2,3) $ denote the Kronecker delta function.

%%%%%%%%%%%%%%%%%%%%%%%%%%%%%%%%%%%%%%%%%%%%%%%%%%%%%
%%%%%%%%%%%%%%%%%%%%%%%%%%%%%%%%%%%%%%%%%%%%%%%%%%%%%

\vspace{0.3cm}

\section{Ansatz and Main result}

First, let $ (\rho_\pm, \mv_\pm) = (\rho_\pm, m_{1\pm}, m_{2\pm},m_{3\pm})(x,t) $ be the periodic solutions of \cref{NS} with the periodic initial data 
\begin{equation*}
	(\rho_\pm, \mv_\pm)(x,0) = (\rho_{\pm,0}, \mv_{\pm,0})(x) = (\rhob_\pm, \mvb_\pm) + (v_0, \wv_0)(x), \qquad x\in\R^3,
\end{equation*}
where $ (\rhob_\pm, \mvb_\pm) $ are the constant states of the shock wave at the far field and $ (v_0, \wv_0) $ is the periodic perturbation in \cref{ic}, satisfying \cref{zero-ave}. The global existence and exponential decay rate of $ (\rho_\pm, \mv_\pm) $ can be found in \cite{HXY2022}.

\begin{Lem}[\cite{HXY2022}]\label{Lem-per}
	Consider the Cauchy problem \cref{NS} with the periodic initial data
	\begin{equation}\label{ic-periodic}
		(\rho, \mv)(x,0) = (\rhob,\mvb) + (v_0, \wv_0)(x), \quad x\in\R^3,
	\end{equation}
	where  $ \rhob > 0, \mvb \in \R^3 $ are constants, and $ (v_0, \wv_0)\in H^{k+2}(\Torus^3) ~(k\geq 1) $ is periodic on $ \Torus^3 = [0,1]^3, $ satisfying 
	$$ \int_{\Torus^3} (v_0, \wv_0 ) dx =0. $$ 
	Then there exists $ \e_0>0 $ such that if $ \e := \norm{v_0,\wv_0}_{H^{k+2}(\Torus^3)} \leq \e_0, $ then the Cauchy problem \cref{NS}, \cref{ic-periodic} admits a unique global periodic solution $ (\rho, \mv) \in C\big((0,+\infty); H^{k+2}(\Torus^3)\big), $ satisfying 
	\begin{equation*}
		\int_{\Torus^3} (\rho - \rhob, \mv - \mvb )(x,t) dx = 0,\qquad t\geq 0
	\end{equation*}
	and 
	\begin{equation}\label{exp-decay}
		\norm{(\rho, \mv) - (\rhob, \mvb) }_{W^{k,\infty}(\Torus^3)} \lesssim \e e^{- c t}, \qquad t>0.
	\end{equation}
	Here the constant $ c>0 $ is independent of $ \e $ and $ t. $ 
\end{Lem}
For convenience, we denote the periodic perturbations,
\begin{equation}\label{vw-pm}
	(v_\pm, \wv_\pm) := (\rho_\pm,\mv_\pm) - (\rhob_\pm, \mvb_\pm), \qquad x\in\R^3, t>0.
\end{equation}
Note that both $ (v_-, \wv_-) $ and $ (v_+, \wv_+) $ have zero averages over $ \Torus^3 $ for all $ t\geq 0 $ and decay exponentially fast in time. 

Denote $ \e := \norm{v_0,\wv_0}_{H^6(\Torus^3)} $. If $ \e>0 $ is small, then it follows from \cref{Lem-per} that
\begin{equation}\label{decay-vw}
	\norm{v_\pm, \wv_\pm}_{W^{4,\infty}(\R^3)} \lesssim \e e^{-ct} \qquad \forall t>0,
\end{equation}
and 
$$ \frac{1}{2} \rhob_\pm \leq \inf\limits_{x\in\R^3, t>0} \rho_\pm(x,t) \leq \sup\limits_{x\in\R^3,t>0} \rho_\pm(x,t) \leq 2 \rhob_\pm. $$ 
By \cref{RH,ode-shock}, one can prove that
\begin{equation}\label{sig-eta}
	\frac{\rhos(x_1)-\rhob_-}{\jump{\rho}} = \frac{\ms_1(x_1) - \mb_{1-}}{\jump{m_1}} := \eta(x_1) \qquad \forall x_1\in\R.
\end{equation}

\textbf{Ansatz.}
Motivated by \cite{XYYindi,HY1shock,Yuan2021}, we construct the ansatz for the Cauchy problem \cref{NS}, \cref{ic} as follows. For any $ x\in\R^3 $ and $ t\geq 0, $ set
\begin{equation}\label{ansatz}
	\begin{aligned}
		\rhot(x,t) & := \rho_-(x,t) \big(1-\eta(x_1-st-X(t)) \big) + \rho_+(x,t) \eta(x_1-st-X(t)), \\
		\mvt(x,t) & := \mv_-(x,t) \big(1-\eta(x_1-st-Y(t)) \big) + \mv_+(x,t) \eta(x_1-st-Y(t)),
	\end{aligned}
\end{equation}
where $ X(t) $ and $ Y(t) $ are two $ C^1 $ curves to be determined. Set
\begin{equation}\label{uvt}
	\uvt(x,t) := \frac{\mvt(x,t)}{\rhot(x,t)} \qquad \forall x\in\R^3, t>0.
\end{equation}
It follows from direct calculations that the ansatz \cref{ansatz} satisfies 
\begin{equation}\label{equ-ansatz}
	\begin{cases}
		\p_t \rhot + \dv \mvt = - \dv \Fv_1 - f_2, \\
		\p_t \mvt + \dv \big(\frac{\mvt\otimes\mvt}{\rhot} \big) + \nabla p(\rhot) - \mu \lap \uvt -(\mu+\lambda) \nabla \dv \uvt \\
		\qquad\qquad\quad = - \sum\limits_{i=1}^{3} \p_i \Fv_{3,i} - \fv_4,
	\end{cases} \qquad x \in \R^3, t>0,
\end{equation}
where the source terms, $ \Fv_1 = (F_{1,1}, F_{1,2}, F_{1,3}) \in \R^3, $  $ f_2 \in \R, $  $ \Fv_{3,i} = (F_{3,i1}, F_{3,i2}, F_{3,i3}) \in \R^3 $ for $ i=1,2,3, $ and $ \fv_4 = (f_{4,1},f_{4,2},f_{4,3}) \in \R^3 $, are given by
\begin{equation}\label{F}
	\begin{aligned}
	\Fv_1 & := (\mv_{+} - \mv_{-}) (\eta_{st+X}-\eta_{st+Y}),  \\
	f_2 & := \big[ (\rho_+-\rho_-) (s+X')  - (m_{1+}-m_{1-}) \big] \eta'_{st+X},\\
	\Fv_{3,i} & := u_{i-} \mv_- (1-\eta_{st+Y}) +  u_{i+} \mv_+ \eta_{st+Y} - \ut_i \mvt \\
	& \quad + \big[ p(\rho_-) (1-\eta_{st+Y}) + p(\rho_+) \eta_{st+Y} - p(\rhot)  \big] \E_i  \\
	& \quad  - \mu  \big[ \p_i \uv_- (1-\eta_{st+Y}) + \p_i \uv_+ \eta_{st+Y} - \p_i \uvt \big] \\
	& \quad - (\mu+\lambda) \big[ \dv \uv_- (1-\eta_{st+Y}) + \dv \uv_+ \eta_{st+Y} - \dv \uvt \big] \E_i, \\
%	& \quad +  \mu (\uv_+ - \uv_-) \eta'_{st+Y} \delta_{1i} + (\mu+\lambda) (u_{i+} - u_{i-}) \eta'_{st+Y} \E_1, \\
	\fv_4 & := \big[ (\mv_+-\mv_-) (s+Y') -  ( u_{1+} \mv_+ - u_{1-} \mv_- ) -  (p(\rho_+) - p(\rho_-) ) \E_1 \\
	& \qquad  + \mu \p_1 (\uv_+ - \uv_-) + (\mu+\lambda) \dv (\uv_+ - \uv_-) \E_1 \big] \eta'_{st+Y},
	\end{aligned}
\end{equation}
respectively. The initial data of the ansatz is given by
\begin{equation}\label{ic-ansatz}
	(\rhot, \mvt)(x,0) = \big(\rhos(x_1-X(0)), \mvs(x_1-Y(0))\big) + (v_0, \wv_0)(x), \qquad x\in\R^3.
\end{equation}

In this paper, we expect the perturbation associated with the density and the normal component of momentum to carry zero mass over $ \Omega $ for all $ t\geq 0, $ i.e.
\begin{equation}\label{zero-mass}
	\int_\Omega (\rho-\rhot)(x,t) dx = \int_\Omega (m_1 - \mt_1)(x,t) dx = 0 \qquad \forall t\geq 0.
\end{equation}
Here we do not need the zero mass of the perturbation associated with the transverse components of momentum, $ (m_2,m_3) - (\mt_2,\mt_3) $.
Note that for all $ t > 0, $ both $ \abso{\Fv_1} $ and $ \abso{\Fv_{3,i}} $ for $ i=1,2,3 $ vanish as $ \abso{x_1} \to +\infty. $ 
Thus, if we can verify that
\begin{equation}\label{zero-mass-f}
	\int_{\Omega} f_2(x,t) dx = \int_{\Omega} f_{4,1}(x,t) dx = 0 \qquad \forall t>0,
\end{equation}
and
\begin{equation}\label{zero-mass-ic}
	\int_\Omega (\rho-\rhot)(x,0) dx = \int_\Omega (m_1 - \mt_1)(x,0) dx = 0,
\end{equation}
then \cref{zero-mass} follows from the conservation laws immediately.

By direct calculations, \cref{zero-mass-f} is equivalent to that the curves $ X $ and $ Y $ satisfy the ODEs,
\begin{equation}\label{ode-XY}
	X'(t) = - s + \frac{\mathcal{L}_1(X,t)}{\mathcal{L}_2(X,t)} \quad \text{and} \quad Y'(t) = -s + \frac{\mathcal{L}_3(Y,t)}{\mathcal{L}_1(Y,t)}, \qquad t>0,
\end{equation}
respectively, where 
\begin{equation}\label{L}
	\begin{aligned}
	\mathcal{L}_1(d,t) & := \int_\Omega (m_{1+} - m_{1-}) \eta'_{st+d}  dx, \\
	\mathcal{L}_2(d,t) & := \int_\Omega (\rho_+ - \rho_-)  \eta'_{st+d} dx, \\
	\mathcal{L}_3(d,t) & := \int_\Omega \Big\{ \big[ \big( u_{1+} m_{1+} - u_{1-} m_{1-}\big) + \big(p(\rho_+) - p(\rho_-) \big) \big] \eta'_{st+d} \\
	& \qquad\quad\ +  \mut  (u_{1+} - u_{1-}) \eta''_{st+d} \Big\}  dx,
\end{aligned} \qquad d \in \R, t>0.
\end{equation}
And \cref{zero-mass-ic} is equivalent to that $ X $ and $ Y $ take the initial values,
\begin{equation}\label{XY-0}
		(X,Y)(0) = \Big(\frac{-1}{\jump{\rho}} \int_\Omega \phib _0(x) dx, \frac{-1}{\jump{m_1}} \int_\Omega \psib_{01} (x) dx\Big) := (X_0, Y_0).
\end{equation}
\begin{Lem}\label{Lem-shift}
	Assume that \cref{RH,entropy} hold. Let $ (X_0, Y_0) $ be the constant pair given by \cref{XY-0}. Then there exists  $ \gamma_0 >0 $ such that if the periodic perturbation $ (v_0,\wv_0) $ satisfies \cref{zero-ave} and
	\begin{equation*}
		\e := \norm{(v_0, \wv_0)}_{H^6(\Torus^3)} \leq \gamma_0 \delta \quad \text{with} \ \delta := \abso{\jump{\rho}}>0,
	\end{equation*}
	then the problem \cref{ode-XY} with the initial data, 
	\begin{equation}\label{ic-XY}
		 (X,Y)(0) = (X_0, Y_0),
	\end{equation}
	admits a unique solution $ (X,Y) \in C^3(0,+\infty), $  satisfying that
	\begin{equation}\label{decay-XYp}
		\abso{X'(t)} + \abso{Y'(t)} \lesssim \e e^{-ct} \qquad \forall t>0,
	\end{equation}
	and
	\begin{equation}\label{decay-XY}
			\abso{X(t)-X_\infty} + \abso{Y(t)-Y_\infty} \lesssim \e \delta^{-1} e^{-ct} \lesssim \gamma_0 e^{-ct} \qquad \forall t > 0,
	\end{equation}
	where the limits $ X_\infty $ and $ Y_\infty $ are constants, given by
	\begin{equation}\label{XY-infty}
		\begin{aligned}
			X_\infty = X_0 \quad \text{and} \quad
			Y_\infty = Y_0 + Y_{\infty,p}, \qquad \text{respectively, }
		\end{aligned}
	\end{equation}
	with
	\begin{equation}\label{Y-shift-p}
		\begin{aligned}
			Y_{\infty,p} & := \frac{1}{\jump{m_1}} \int_0^{+\infty} \int_{\Torus^3} \Big\{ \big(u_{1+} m_{1+} - \ub_{1+} \mb_{1+} \big) - \big(u_{1-} m_{1-} - \ub_{1-} \mb_{1-} \big) \\
			& \qquad\qquad\qquad\qquad\quad + \big( p(\rho_+) - p(\rhob_+) \big) - \big( p(\rho_-) - p(\rhob_-) \big) \Big\} dx dt.
		\end{aligned}
	\end{equation}
\end{Lem}
Note that the integral in \cref{Y-shift-p} is finite by \cref{Lem-per}. The proof of \cref{Lem-shift} is similar to that in \cite{XYYindi,HY1shock,Yuan2021} and we place it in \cref{App-shift}.

\vspace{0.3cm}

By Lemmas \ref{Lem-per} and \ref{Lem-shift}, we can observe that
\begin{equation*}
	\norm{(\rhot, \mvt)(x,t) - \left(\rhos(x_1-st-X_\infty), \mvs(x_1-st-Y_\infty)\right)}_{L^\infty(\R^3)} \to 0 \quad \text{as } t\to+\infty.
\end{equation*}
However, the pair $ \left(\rhos(x_1-st-X_\infty), \mvs(x_1-st-Y_\infty)\right) $ is a planar viscous shock of \cref{NS}, if and only if 
\begin{equation}\label{coincide-XY}
	X_\infty = Y_\infty.
\end{equation}
Thus, \cref{coincide-XY} is a necessary condition for the shock stability. 
From \cref{XY-infty,Y-shift-p}, the relation \cref{coincide-XY} is actually equivalent to that
\begin{equation}\label{zero-mass-con}
	\begin{aligned}
		& \int_\Omega \psib_{01}(x) dx - s \int_\Omega \phib_0(x) dx \\
		& \qquad = \int_0^{+\infty} \int_{\Torus^3} \Big\{ \big(u_{1+} m_{1+} - \ub_{1+} \mb_{1+} \big) - \big(u_{1-} m_{1-} - \ub_{1-} \mb_{1-} \big) \\
		& \qquad\qquad\qquad\qquad + \big( p(\rho_+) - p(\rhob_+) \big) - \big( p(\rho_-) - p(\rhob_-) \big) \Big\} dx dt.
	\end{aligned}
\end{equation}
This is exactly the zero-mass condition in \cite{MN} when the periodic perturbation in \cref{ic} vanishes, that is,
\begin{equation*}
	s \int_\Omega \phib_0(x) dx = \int_\Omega \psib_{01}(x) dx.
\end{equation*}
This is equivalent to that there exists a (unique) common number $ \sigma \in\R $ such that
\begin{equation}\label{zero-mass-loc}
	\int_\Omega \big(\rho_0(x) - \rhos(x_1-\sigma) \big) dx = \int_\Omega \big( m_{01}(x) - \ms_1(x_1-\sigma) \big) dx = 0.
\end{equation}

\vspace{.2cm}

With the constants $ X_0 $ and $ Y_0 $ given by \cref{XY-0}, we define
\begin{align}
	(\phi_0, \psi_0)(x) & = (\phi_0, \psi_{01}, \psi_{02}, \psi_{03})(x) \notag \\
	& := ( \rho-\rhot, \mv - \mvt )(x, t=0) \notag \\
	& = \big(\rhos(x_1)-\rhos(x_1-X_0), \ms_1(x_1) - \ms_1(x_1-Y_0), 0, 0 \big) + (\phib_0, \psib_0)(x), \label{phi0psi0}
\end{align}
and
\begin{align}
	(\Phi_0, \Psi_{01})(x_1) := \int_{-\infty}^{x_1} \int_{\Torus^2} (\phi_0, \psi_{01})(y_1, \xp) d\xp, \qquad x_1 \in \R. 
	\label{Phi0Psi0}
\end{align}

\vspace{.3cm}

Now we are ready to state the main result.

\begin{Thm}\label{Thm}
	Assume that \cref{RH,entropy} hold. Then there exist $ \delta_0>0, \gamma_0>0 $ and $ \nu_0>0 $ such that if $ \abso{\rhob^+ -\rhob^-} \leq \delta_0 $ and
	\begin{itemize}
		\item[1)] the perturbation $ (v_0, \wv_0) $ in \cref{ic} is periodic on torus $ \Torus^3 = [0,1]^3 $, satisfying \cref{zero-ave} and that 
		$$ \e:=\norm{ v_0,\wv_0}_{H^6(\Torus^3)} \leq \gamma_0 \abso{\rhob^+ -\rhob^-}; $$
		
		\item[2)] the perturbation $ (\phib_0,\psib_0) $ in \cref{ic} is periodic in the transverse variables $ \xp $ on $ \Torus^2 $ and belongs to the $ H^3(\Omega) \cap L^1(\Omega) $ space;
		moreover, the pairs \cref{phi0psi0,Phi0Psi0} satisfy that 
		$$ \nu:= \norm{ \phi_0, \psi_0}_{H^3(\Omega)} + \norm{ \Phi_0, \Psi_{01}}_{L^2(\R)} \leq \nu_0; $$
		
		\item[3)] the zero-mass condition \cref{zero-mass-con} holds true,
	\end{itemize}
	then the Cauchy problem \cref{NS}, \cref{ic} admits a unique classical solution $ (\rho,\uv)(x,t) $ globally in time, which is periodic in $ \xp \in \Torus^2 $ and satisfies that
	\begin{equation}\label{exist}
		\begin{aligned}
			(\rho, \uv) - (\rhot, \uvt) & \in C\big(0,+\infty; H^3(\Omega)\big), \\
			\rho-\rhot & \in L^2\big(0,+\infty; H^3(\Omega)\big), \\
			u_1 - \ut_1 & \in L^2\big(0,+\infty; L^2(\Omega)\big), \\
			\nabla (\uv-\uvt) & \in L^2\big(0,+\infty; H^3(\Omega)\big),
		\end{aligned}
	\end{equation}
	where $ (\rhot,\uvt) $ is the ansatz given by \cref{ansatz,uvt}. Moreover, it holds that
	\begin{equation}\label{behavior}
		\norm{ (\rho,\uv)(x,t)- (\rhos, \us_1, 0,0 )(x_1-st-X_\infty) }_{W^{1,\infty}(\R^3)} \rightarrow 0 \quad \text{as }\ t\rightarrow +\infty,
	\end{equation}
	where $ X_\infty (= Y_\infty) $ is the constant given by \cref{XY-infty,XY-0}; and the non-zero mode of the solution,
	\begin{align*}
		(\rho^\md, \uv^\md)(x,t) := (\rho,\uv)(x,t) - \int_{\Torus^2} (\rho,\uv)(x_1, \xp,t) d\xp, \qquad x\in\R^3, t>0,
	\end{align*}
	satisfies that
	\begin{equation}\label{exp-Thm}
		\norm{(\rho^\md, \uv^\md)(\cdot,t)}_{L^\infty(\R^3)} \leq C(\e+\nu) e^{-c_0 t} \qquad \forall t>0,
	\end{equation}
	where $ C>0 $ and $ c_0>0 $ are some generic constants, independent of $ \delta, \e, \gamma_0, \nu $ and $ t. $
\end{Thm}

\begin{Rem}
	The result of \cref{Thm} still holds true in two dimensions, since we allow the initial perturbation in \cref{ic} to be independent of either $ x_2 $ or $ x_3. $
\end{Rem}

\begin{Rem}
	Note that in the assumption 2) of \cref{Thm}, the $ L^2 $-integrable condition of the anti-derivative variable,
	\begin{equation}\label{con}
		(\Phi_0, \Psi_{01}) \in L^2(\R)
	\end{equation}
	is natural provided that $ (\phib_0, \psib_{01}) $ decays sufficiently fast as $ \abso{x_1} \to +\infty. $
	For instance, if
	\begin{equation}\label{decay-fast}
		\abso{(\phib_0, \psib_{01})(x)} \lesssim (1+ \abso{x_1}^2)^{-k/2}  \qquad \forall x\in \Omega \quad \text{with } \  k > 3/2,
	\end{equation}
	then the constants $ X_0 $ and $ Y_0 $ in \cref{XY-0} exist. The pair \cref{Phi0Psi0} is also well defined and satisfies that
	\begin{align}
		(\Phi_0, \Psi_{01})(x_1) & = -\int_{x_1}^{+\infty} \big(\rhos(y_1)-\rhos(y_1-X_0), \ms_1(y_1) - \ms_1(y_1-Y_0) \big) dy_1 \notag \\
		& \quad - \int_{x_1}^{+\infty} \int_{\Torus^2} ( \phib_0, \psib_{01})(y_1,\xp) d\xp dy_1,
		\quad x_1 \in\R. \label{PP}
	\end{align}
	As the viscous shock $ (\rhos, \ms_1) $ decays exponentially fast to the constant shock states at the far field (see \cref{Lem-shock}), then combining \cref{decay-fast}, one can get that the right-hand side (RHS) of \cref{Phi0Psi0} belongs to $ L^2(\R_-) $ and RHS of \cref{PP} belongs to $ L^2(\R_+) $, which yields \cref{con} immediately.
\end{Rem}

\begin{Rem}
	If we weaken the regularity of $ (\phib_0, \psib_0) $ from $ H^3(\Omega) $ to $ H^2(\Omega), $ then the problem \cref{NS}, \cref{ic} also admits a unique strong solution globally in time,  which satisfies that $ (\rho,\uv) - (\rhot, \uvt) \in C(0,+\infty;H^2(\Omega)) $ and
	\begin{align*}
		\norm{ (\rho,\uv)(x,t)- (\rhos, \us_1, 0,0 )(x_1-st-X_\infty) }_{L^\infty(\R^3)} \rightarrow 0 \quad \text{as }\ t\rightarrow +\infty.
	\end{align*}
	However, in this case the exponential decay rate \cref{exp-Thm} can not be obtained, since we are not able to control some interactions of the zero mode and non-zero mode of the perturbation with only $ H^2 $-regularity; see the details in the proof of \cref{Lem-exp-2}.
\end{Rem}

\begin{Rem}
	In comparison to the localized perturbations studied by \cite{G1986,MN}, the shift component \cref{Y-shift-p} in this paper is generated essentially by the periodic perturbations and depends on the structure of the system \cref{NS}, i.e. the shift is subject to the dynamics of the oscillations. Thus, the expression \cref{Y-shift-p} in Eulerian coordinates is different from that in Lagrangian ones (see \cite{HY1shock}). 
	It is interesting to find a time-independent formula to compute the limit $ Y_{\infty,p} $, which can help understand more clearly the influence of periodic perturbations on the shock stability. However, this is still open except for some 1D scalar models such as the inviscid convex conservation laws and the viscous Burgers' equation; see \cite{XYYsiam,YY2019,XYYindi}.
\end{Rem}
\begin{Rem}\label{Rem-loc}

	Recently, Wang-Wang \cite{WW2022} also studied the nonlinear stability of planar Naiver-Stokes shocks for \cref{NS} in the infinitely duct $ \Omega = \R \times\Torus^2 $.
	It is noted that both the analysis and the stability result of \cite{WW2022} are different from this paper, and besides, we consider an additional periodic perturbation. In fact, \cite{WW2022} utilized a weighted $ L^2 $-relative entropy method to achieve a nonlinear stability of planar Naiver-Stokes shocks with generic $ H^2(\Omega) $-perturbations. 
	Their analysis is powerful since they do not require the zero mass or the spectral condition, while the final shock locations are not determined there. 
	In this paper, the anti-derivative technique plays an important role in determining the final shock locations as constants. 
	That is, we have proved the orbital stability of the planar Navier-Stokes shocks with asymptotic phase.
	
\end{Rem}

\vspace{0.3cm}

\section{Preliminaries}\label{Sec-pre}

For simple notations, throughout the paper we denote 
\begin{align*}
	\delta:= \abso{\jump{\rho}} > 0, \qquad \e := \norm{v_0, \wv_0}_{H^6(\Torus^3)},
\end{align*}
\begin{equation}\label{shock-ts}
	(\rhot^s, \mvt^s)(x_1,t) := (\rhos, \mvs)(x_1-st-X_\infty) \quad \text{and} \quad \uvt^s := \frac{\mvt^s}{\rhot^s},
\end{equation}
where $ X_\infty (=Y_\infty) $ is the constant in \cref{XY-infty}.
And we list some useful lemmas as follows.

\begin{Lem}[\cite{MN,SX}]\label{Lem-shock}
	Assume that \cref{RH,entropy} hold.
	Then the smooth viscous shock, $ \big( \rhos, \ms_1 \big)(x_1-st) $, solving \cref{ode-shock} exists uniquely (up to a shift) and satisfies the following properties,
	\begin{itemize}
		
		\item[i) ] $ (\rhos)'(x_1) < 0 $ and $ (\us_1)'(x_1) < 0 $  for all $ x_1 \in\R $;

		\item[ii) ] $ \delta^2 e^{-c_1 \delta \abso{x_1}} \lesssim \abso{(\us_1)'(x_1)} \lesssim \delta^2 e^{-c_2 \delta \abso{x_1}} $ for all $ x_1 \in\R, $ where $ c_1 > c_2 $ are two positive constants, independent of $ \delta $ and $ x_1 $;
		
		\item[iii) ] $ \abso{(\us_1)''(x_1)} \lesssim \delta \abso{(\us_1)'(x_1)} $ for all $ x_1 \in\R $.
	\end{itemize}
\end{Lem}

\vspace{.1cm}

As indicated in \cite{HY2020}, the 3D functions that are periodic in $ \xp = (x_2, x_3) $ may not satisfy the 3D Gagliardo-Nirenberg (G-N) inequalities in general (counterexample: the 1D functions in the $ C_c^\infty(\R) $ space are periodic in $ \xp $ and satisfy the 1D G-N inequalities, but not the 3D ones; see \cite{HY2020} for the details). Thus we should utilize the following G-N type inequality relevant to the unbounded domain $ \Omega = \R \times\Torus^2. $

\begin{Lem}[\cite{HY2020}]\label{Lem-GN}
	Assume that $ u(x) $ is periodic in $ \xp = (x_2, x_3) $ and belongs to the $ L^q(\Omega) $ space with $ \nabla^m u \in L^r(\Omega), $ where $ 1\leq q,r\leq +\infty $ and $ m\geq 1 $. Then there exists a decomposition $ u(x) = \sum\limits_{k=1}^{3} u^{(k)}(x) $ such that 
	\begin{itemize}
		\item[1)] each $ u^{(k)} $ satisfies that	
		\begin{align*}
			\norm{\nabla^j u^{(k)}}_{L^p(\Omega)} \lesssim \norm{\nabla^j u }_{L^p(\Omega)},
		\end{align*}
		where $ j\geq 0 $ is any integer and $ p\in[1,+\infty] $ is any number;
		
		\item[2)] each $ u^{(k)} $ satisfies the $ k $-dimensional G-N inequality, namely,
		\begin{equation}\label{G-N-type-1}
			\norm{\nabla^j u^{(k)}}_{L^p(\Omega)} \lesssim \norm{\nabla^{m} u}_{L^{r}(\Omega)}^{\theta_k} \norm{u}_{L^{q}(\Omega)}^{1-\theta_k},
		\end{equation}
		where $ 0\leq j< m $ is any integer and $ 1\leq p \leq +\infty $ is any number, satisfying 
		$$ \frac{1}{p} = \frac{j}{k} + \Big(\frac{1}{r}-\frac{m}{k}\Big) \theta_k + \frac{1}{q}\big(1-\theta_k\big) \quad \text{with} \quad \frac{j}{m} \leq \theta_k \leq 1. $$
	\end{itemize}
	Moreover, it holds that 
	\begin{equation}\label{G-N-type-2}
		\norm{\nabla^j u}_{L^p(\Omega)} \lesssim \sum_{k=1}^3 \norm{\nabla^{m} u}_{L^{r}(\Omega)}^{\theta_k} \norm{u}_{L^{q}(\Omega)}^{1-\theta_k}.
	\end{equation}
	%	The constants $ C>0 $ in \cref{G-N-type-1,G-N-type-2} are independent of $ u. $
\end{Lem}

\vspace{.2cm}

\textbf{Galilean transformation.} 
It is straightforward to verify that the compressible NS equations \cref{NS}, the Rankine-Hugoniot conditions \cref{RH} and the entropy condition \cref{entropy} are invariant under the Galilean transformation,
\begin{equation}\label{Galilean}
	(\rho^*, u_1^*, u_2^*, u_3^*)(x^*,t^*) = (\rho, u_1-a, u_2, u_3)(x_1^*+a t^*, x_2^*, x_3^*, t^*), \quad x^* \in \R^3, t^* \geq 0,
\end{equation}
%with 
%\begin{align*}
%	x^* := x -  (at,0,0).
%\end{align*}
where $ a\in\R $ is any constant. 
After the transformation, the shock wave $ (\rhob^s, \ub^s_1)(x_1,t) $ turns to a new one, $ (\rhob^{s*}, \ub_1^{s*})(x_1^*,t^*), $ which propagates along the $ x_1^* $-axis with the shock speed $ s^* := s - a $ and connects the shock states $ (\rhob_\pm^*, \ub_{1\pm}^*) := \big(\rhob_\pm,  \ub_{1\pm} - a \big) $ as $ x_1^* \to \pm \infty. $ By setting $ a = \frac{\ub_{1-} + \ub_{1+}}{2}, $ one can get that $ \ub_{1-}^* = - \ub_{1+}^*. $

Thus, we can assume without loss of generality that 
\begin{equation}\label{u-pm}
	\ub_{1-} = - \ub_{1+} > 0.
\end{equation}
This, together with \cref{RH}$ _1 $ and \cref{entropy}, implies that
\begin{equation}\label{small-u}
	\abso{\ub_{1-}} = \abso{\ub_{1+}} \lesssim \abso{\rhob_+ - \rhob_-} = \delta.
\end{equation}
Then it holds that
\begin{equation}\label{small-us}
	\abso{\ut^s_1(x_1)} \leq \abso{\ub_{1\pm}} \lesssim \delta \qquad \forall x_1 \in \R.
\end{equation}
Moreover, if $ \delta>0 $ is small, it follows from the entropy condition \cref{entropy} that
\begin{equation}\label{positive-s}
	s>\lambda_+(\rhob_+, \ub_{1+}) = \ub_{1+} + \sqrt{p'(\rhob_+)} \geq \frac{1}{2} \sqrt{p'(\rhob_+)} >0.
\end{equation}
Using Lemmas \ref{Lem-per} and \ref{Lem-shock} and \cref{small-us}, one has that
\begin{align}
	\sup_{ t>0} \norm{\nabla\rhot}_{W^{3,\infty}(\R^3)} + \sup_{ t>0} \norm{\uvt}_{W^{4,\infty}(\R^3)} &\lesssim \delta + \norm{(v_\pm, \wv_\pm)}_{W^{4,\infty}(\Torus^3)}  \notag \\
	& \lesssim \delta + \e. \label{small-tilde}
\end{align}
Also, it follows from \cref{coincide-XY} and Lemmas \ref{Lem-per}, \ref{Lem-shift} and \ref{Lem-shock} that the difference between the ansatz and the background viscous shock satisfies that
\begin{align}
	\norm{(\rhot, \mvt)-(\rhot^s,\mvt^s)}_{W^{4,\infty}(\R^3)} &\lesssim \norm{(v_\pm, \wv_\pm)}_{W^{4,\infty}(\Torus^3)}  \notag \\
	& \quad + \delta \big(\abso{X(t)-X_\infty} + \abso{Y(t)-X_\infty}\big)  \notag \\
	&  \lesssim \e e^{-ct}. \label{ansatz-shock}
\end{align}
Thus, it holds that
\begin{align}
	\norm{\uvt - \uvt^s}_{W^{4,\infty}(\R^3)} & \lesssim \norm{\rhot - \rhot^s}_{W^{4,\infty}(\R^3)} + \norm{ \mvt-\mvt^s}_{W^{4,\infty}(\R^3)} \notag \\
	& \lesssim \e e^{-ct}.  \label{ansatz-shock-1}
\end{align}
If $ \e > 0 $ is small, one can get that 
\begin{equation}\label{bdd-rhot}
	\frac{1}{2}\rhob^+ \leq \inf_{\substack{x\in \R^3 \\ t>0}} \rhot(x,t)  \leq \sup_{\substack{x\in \R^3 \\ t>0}} \rhot(x,t) \leq 2 \rhob^-.
\end{equation}

%To simplify the notation, it is convenient to change the coordinate,
%$$ x_{1*} = x_1 - \sigma, \quad \xp_* = \xp, \quad t_* = t, $$
%to assume that $ \sigma = 0. $

\vspace{0.3cm}

\section{Reformulated problem}\label{Sec-prob}

\vspace{.2cm}

Define the perturbations, 
\begin{equation}\label{def-pertur}
	 \phi := \rho-\rhot, \quad \psi = (\psi_1, \psi_2, \psi_3) := \mv - \tilde{\mv}.
\end{equation}
Then it follows from \cref{NS,ic,equ-ansatz,ic-ansatz} that
\begin{equation}\label{equ-phipsi}
	\begin{cases}
		\p_t \phi + \dv \psi = \dv \Fv_1 + f_2 := g_1, \\
		\p_t \psi + \dv \big( \frac{\mv\otimes\mv}{\rho} - \frac{\mvt\otimes\mvt}{\rhot} \big) + \nabla (p(\rho) - p(\rhot))  - \mu \lap \big( \frac{\mv}{\rho} - \frac{\mvt}{\rhot}\big) \\
		\quad  - (\mu+\lambda) \nabla \dv \big( \frac{\mv}{\rho} - \frac{\mvt}{\rhot}\big) = \sum\limits_{i=1}^3 \p_i \Fv_{3,i} + \fv_4 := \gv_2 = (g_{2,1},g_{2,2}, g_{2,3}),
	\end{cases}
\end{equation}
and 
\begin{equation}\label{ic-phipsi}
	\begin{aligned}
		(\phi, \psi)(x,0) = (\phi_0,\psi_0)(x) \in H^3(\Omega).
	\end{aligned}
\end{equation}
Here $ (\phi_0,\psi_0) $ is given by \cref{phi0psi0}.
%Note that 
%\begin{align*}
%	\norm{\phi_0}_{H^3(\Omega)} & \lesssim \norm{\phi_0}_{H^3(\Omega)} + \norm{(\rhot^s)'}_{H^2(\R)} \abso{X_0} \\
%	& \lesssim 
%\end{align*}

Define the perturbation of velocity, $ \zeta = \uv-\uvt. $ The system of $ (\phi, \zeta) $ can be written as
\begin{equation}\label{equ-phizeta}
	\begin{cases}
		\p_t  \phi + \rho \dv \zeta + \uv \cdot \nabla \phi + \dv \uvt \phi + \nabla \rhot \cdot \zeta = g_1,  \\
		\rho \p_t \zeta + \rho \uv \cdot \nabla \zeta + \nabla \big(p(\rho) -  p(\rhot)\big)  + \rho \zeta \cdot \nabla \uvt + \phi (\p_t \uvt + \uvt\cdot \nabla \uvt) \\
		\qquad \qquad \qquad \quad - \mu \lap \zeta - (\mu+\lambda) \nabla \dv \zeta = \gv_2 - g_1 \uvt.
	\end{cases}
\end{equation}
And the initial data of $ (\phi,\zeta) $ is
\begin{equation}\label{ic-pertur}
	(\phi,\zeta)(x,0) = (\phi_0,\zeta_0)(x) \in H^3(\Omega),
\end{equation}
where 
\begin{align*}
	\zeta_0(x) := \frac{\mv_0(x)}{\rho_0(x)} - \frac{\mvt(x,0)}{\rhot(x,0)} = \frac{1}{\rho_0(x)} \big(\psi_0(x) - \uvt(x,0) \phi_0(x) \big).
\end{align*}

\vspace{.3cm}

For any $ T>0, $ we take the solution space for \cref{equ-phizeta} as follows,
\begin{align*}
	\mathbb{B}(0,T) := \big\{ & (\phi, \zeta):  (\phi, \zeta) \text{ is periodic in } \xp=(x_2,x_3) \in \Torus^2, \\
	& \qquad \quad (\phi, \zeta) \in C(0,T; H^3(\Omega)), \\ 
	&  \qquad \quad  \zeta_1 \in L^2\big(0,T; L^2(\Omega)\big), \\
%	& \qquad \quad \nabla \phi \in L^2\big(0,T; H^2(\Omega)\big), \\
	& \qquad \quad  (\phi, \nabla \zeta) \in L^2\big(0,T; H^3(\Omega)\big) \big\}.
\end{align*}

Then we show the following result.

\begin{Thm}\label{Thm-pertur}
	Under the assumptions of \cref{Thm}, there exist $ \delta_0>0 $, $ \gamma_0>0 $ and $ \nu_0>0 $ such that if
	\begin{align*}
		& \abso{\rhob^+ -\rhob^-} \leq \delta_0, \quad \norm{ v_0,\wv_0}_{H^6(\Torus^3)} \leq \gamma_0 \abso{\rhob^+ -\rhob^-}, \\
		& \quad \text{ and } \ \norm{ \phi_0, \psi_0}_{H^3(\Omega)} + \norm{ \Phi_0, \Psi_{01}}_{L^2(\R)} \leq \nu_0,
	\end{align*}
	then the Cauchy problem \cref{equ-phizeta}, \cref{ic-pertur} admits a unique solution $ (\phi, \zeta) \in \mathbb{B}(0,+\infty), $ satisfying that
	\begin{equation}\label{asymp}
		\norm{ (\phi, \zeta)}_{W^{1,\infty}(\Omega)} \to 0 \qquad \text{as } t\to+\infty.
	\end{equation}
	Moreover, the non-zero mode of $ (\phi,\zeta) $,
	\begin{align*}
		(\phi^\md, \zeta^\md)(x,t) := (\phi,\zeta)(x,t) - \int_{\Torus^2} (\phi,\zeta)(x_1,\xp,t) d\xp, 
	\end{align*}
	decays exponentially fast in time, i.e.
	\begin{equation}\label{exp-pert}
		\norm{(\phi^\md, \zeta^\md) }_{L^\infty(\Omega)} \leq C(\e+\nu) e^{-c_0 t} \qquad \forall t\geq 0,
	\end{equation}
	where $ C>0 $ and $ c_0>0 $ are some generic constants, independent of $ \delta, \e, \gamma_0, \nu $ and $ t $.
\end{Thm}

The proof of \cref{Thm-pertur} is based on an $ L^2 $-energy method and the anti-derivative technique. And the framework is feasible thanks to the following decomposition idea, which was initiated by the author in \cite{Yuan2021}.

\vspace{.2cm}

\textbf{Key decomposition.} 
For any $ f(x)\in L^\infty(\Omega) $ that is periodic in $ \xp=(x_2,x_3) \in \Torus^2, $ we use the notations, $ f^{\od} $ and $ f^{\md} $, to denote the one-dimensional zero mode,
\begin{equation}\label{od}
	f^{\od}(x_1) := \int_{\Torus^2} f(x_1,\xp) d \xp, \qquad x_1 \in \R,
\end{equation}
and multi-dimensional non-zero mode,
\begin{equation}\label{md}
	 f^{\md}(x) := f(x) - f^{\od}(x_1), \qquad x\in \R^3,
\end{equation}
respectively.

%From \cref{zero-mass-f,zero-mass-ic}, one has that
%\begin{equation}\label{zero-mass}
%	\int_\Omega (\phi,\psi_1)(x,t) dx \equiv 0 \qquad \forall t\geq 0.
%\end{equation}

\vspace{.2cm}

Note that for any $ f\in L^\infty(\Omega) $ that is periodic in $ \xp=(x_2,x_3) \in \Torus^2, $ it holds that
\begin{equation}\label{zero-md}
	\int_{\Torus^2} f^\md(x_1,\xp) d\xp = 0 \qquad \forall x_1 \in\R.
\end{equation}
\begin{Lem}\label{Lem-od-md}
	For any $ p\in [1,+\infty], $ it holds that
	\begin{equation}\label{cauchy-ineq}
		\begin{aligned}
		\norm{f^\od}_{L^p(\R)} & \lesssim \norm{f}_{L^p(\Omega)}, \\
		\norm{f^\md}_{L^p(\Omega)} & \lesssim \norm{f}_{L^p(\Omega)} + \norm{f^\od}_{L^p(\R)} \lesssim \norm{f}_{L^p(\Omega)},
		\end{aligned}
	\end{equation}
	and 
	\begin{equation}\label{poincare}
		\norm{f^\md}_{L^p(\Omega)} \lesssim \norm{\nabla_{\xp} f^\md}_{L^p(\Omega)}.
	\end{equation}
%	where the implicit constants depend on $ p. $
\end{Lem}
\begin{proof}
	It is direct to use the Cauchy's inequality to prove \cref{cauchy-ineq}. And using \cref{zero-md} and the Poincar\'{e} inequality on $ \Torus^2, $ one has that
	\begin{align*}
		\int_{\Torus^2} \abso{f^\md}^p(x_1,\xp) d\xp \lesssim \int_{\Torus^2} \abso{\nabla_{\xp} f^\md}^p(x_1,\xp) d\xp \qquad \forall x_1 \in \R.
	\end{align*}
	Then \cref{poincare} follows immediately from the integration of the inequality on $ \R. $
\end{proof}

\vspace{.2cm}

By the zero-mass property \cref{zero-mass}, the perturbation associated with the density and the first component of momentum, $ (\phi,\psi_1) $, satisfies that
\begin{equation}\label{zero-od}
	\int_\R \big(\phi^{\od}, \psi_1^{\od}\big)(x_1,t) dx_1 = 0 \qquad \forall t\geq 0.
\end{equation}
%and
%\begin{equation}\label{zero-md}
%	\int_{\Torus^2} \big(\phi^{\md}, \psi_1^{\md}\big)(x,t) d\xp = 0 \qquad \text{for all } \ x_1 \in\R, t\geq 0. 
%\end{equation}
%\begin{Lem}
%	For any $ p\in [1,+\infty), $ it holds that
%	\begin{align}
%		\norm{\p_1^k \phi^{\od}}_{L^p(\R)} & \lesssim \norm{\p_1^k \phi}_{L^p(\Omega)}, \quad \ \norm{\phi^{\md}}_{L^p(\Omega)} \lesssim \norm{\nabla \phi}_{L^p(\Omega)} , \\
%		\norm{\p_1^k \psi_1^{\od}}_{L^p(\R)} & \lesssim \norm{\p_1^k \psi_1}_{L^p(\Omega)}, \quad \norm{\psi_1^{\md}}_{L^p(\R)} \lesssim \norm{\nabla \psi_1}_{L^p(\Omega)}.
%	\end{align}
%\end{Lem}
Then the anti-derivative variable,
\begin{equation}\label{anti-de}
	(\Phi, \Psi_1)(x_1,t):= \int_{-\infty}^{x_1} \big(\phi^{\od}, \psi_1^{\od}\big)(y_1,t) dy_1, \quad x_1\in \R, t\geq 0,
\end{equation}
vanishes as $ \abso{x_1} \to +\infty $ for all $ t\geq 0. $
By integrating \cref{equ-phipsi,ic-phipsi} with respect to $ \xp $ over $ \Torus^2 $, we arrive at the system of $ (\phi^\od,\psi^\od) $,
\begin{equation}\label{equ-od}
	\begin{cases}
		\p_t \phi^\od + \p_1 \psi_1^\od = g_1^\od, \\
		\p_t \psi^\od + \p_1 \big[ \int_{\Torus^2} \big( \frac{m_1 \mv}{\rho} - \frac{\mt_1\mvt}{\rhot} \big) d\xp \big] + \p_1 \big[\int_{\Torus^2} \big(p(\rho) - p(\rhot) \big) d\xp \big] \E_1 \\
		\quad - \mu \p_1^2 \big[ \int_{\Torus^2} \big( \frac{\mv}{\rho} - \frac{\mvt}{\rhot} \big) d\xp \big] - (\mu+\lambda) \p_1^2 \big[ \int_{\Torus^2} \big( \frac{m_1}{\rho} - \frac{\mt_1}{\rhot} \big) d\xp \big] \E_1 = \gv_{2}^\od,
	\end{cases} \ x_1 \in \R, t>0, 
\end{equation}
and the initial data,
\begin{equation}\label{ic-od}
		(\phi^\od,\psi^\od)(x_1,0) = (\phi^\od_0, \psi^\od_0)(x_1) \in H^3(\R).
\end{equation}
Formally, by integrating the first two equations of \cref{equ-od} with respect to $ x_1 $ from $ -\infty $ to $ x_1 $, we can get that
\begin{equation}\label{equ-anti}
	\begin{cases}
		\p_t \Phi + \p_1 \Psi_1 = G_1, \\
		\p_t \Psi_1 + 2\ut^s_1 \p_1 \Psi_1 + \alpha \p_1 \Phi - \mut \p_1 \big[ \frac{1}{\rhot^s} \big( \p_1 \Psi_1 - \ut_1^s \p_1 \Phi \big) \big] \\
		\qquad \qquad \quad = G_2 + \int_{\Torus^2} (q_1+ q_2) d\xp + \p_1  \big[\int_{\Torus^2} (q_3+q_4) d\xp\big],
	\end{cases} \quad x_1 \in \R, t>0, 
\end{equation}
where 
\begin{equation}\label{alpha}
	\alpha  = p'(\rhot^s) - \abso{\ut_1^s}^2,
\end{equation}
\begin{equation}\label{G}
	\begin{aligned}
	G_1 & = \int_{-\infty}^{x_1} g_1^{\od}(y_1,t) dy_1 = F_{1,1}^{\od} + \int_{-\infty}^{x_1} f_2^{\od}(y_1,t) dy_1, \\
	G_2 & = \int_{-\infty}^{x_1} g_{2,1}^{\od}(y_1,t) dy_1 = F_{3,11}^{\od} + \int_{-\infty}^{x_1} f_{4,1}^{\od}(y_1,t) dy_1,
	\end{aligned}
\end{equation}
and
\begin{equation}\label{q}
	\begin{aligned}
		q_1 & = -2(\ut_1-\ut_1^s) \psi_1 + \big(\ut_1^2- \abso{\ut_1^s}^2\big) \phi - \big(p'(\rhot) - p'(\rhot^s) \big) \phi \\
		& = O(1) \e e^{-c t} (\abso{\phi} + \abso{\psi_1}), \\
		q_2 & = - \Big(\frac{m_1^2}{\rho} - \frac{\mt_1^2}{\rhot} - 2\ut_1 \psi_1 + \ut_1^2 \phi\Big) - \big(p(\rho)-p(\rhot) - p'(\rhot) \phi \big) \\
		& = O(1) (\abso{\phi}^2 + \abso{\psi_1}^2), \\
		q_3 & = \mut \Big[ \Big( \frac{1}{\rhot} - \frac{1}{\rhot^s} \Big) \psi_1 - \Big( \frac{\ut_1}{\rhot} - \frac{\ut_1^s}{\rhot^s}\Big) \phi \Big]  \\
		& = O(1) \e e^{-c t} (\abso{\phi} + \abso{\psi_1}), \\
		q_4 & = \mut  \Big(\frac{m_1}{\rho} - \frac{\mt_1}{\rhot} - \frac{\psi_1}{\rhot} + \frac{\ut_1}{\rhot} \phi\Big) \\
		& = O(1)(\abso{\phi}^2 + \abso{\psi_1}^2).
	\end{aligned} 
\end{equation}
Here we have used \cref{ansatz-shock,ansatz-shock-1}. And it follows from  \cref{ic-od} and the assumption 2) of \cref{Thm} that the initial data of $ (\Phi,\Psi_1) $,
\begin{equation}\label{ic-anti}
	(\Phi,\Psi_1)(x_1,0) = (\Phi_0, \Psi_{01})(x_1):= \int_{-\infty}^{x_1} (\phi^\od_0, \psi^\od_{01})(y_1) dy_1
\end{equation}
belongs to the $ H^4(\R) $ space.

\vspace{.3cm}

We show in the following lemma that the errors of the ansatz decay exponentially fast as $ t \to +\infty. $ The proof is similar to \cite{HY1shock} and we place it in \cref{App-g}.

\begin{Lem}\label{Lem-F}
	Under the assumptions of \cref{Thm}, there exist $ \delta_0>0 $ and $ \gamma_0>0 $ such that if $ \delta\leq \delta_0 $ and $ \e \leq \gamma_0 \delta $,
	then the errors of the ansatz in \cref{equ-phipsi,equ-anti} satisfy that
	\begin{equation}\label{est-g}
		\norm{g_1}_{H^3(\Omega)} + \norm{\gv_2}_{H^2(\Omega)} \lesssim \e \delta^{\frac{1}{2}} e^{-ct} \qquad \forall t>0,
	\end{equation}
	and 
	\begin{equation}\label{est-G}
		\norm{G_1}_{L^2(\R)} + \norm{ G_2}_{L^2(\R)} \lesssim \e \delta^{-\frac{1}{2}} e^{-ct} \lesssim \gamma_0 e^{-ct}  \qquad \forall t>0,
	\end{equation}
	respectively, where $ c>0 $ is the constant in \cref{exp-decay}.
\end{Lem}

\vspace{.3cm}

Then we show the a priori estimates and the local existence theorem of the problem \cref{equ-phizeta}, \cref{ic-pertur}.

\begin{Prop}[A priori estimates]\label{Prop-apriori}
	Under the assumptions of \cref{Thm-pertur}, for any fixed $ T>0, $ assume that $ ( \phi, \zeta) \in \mathbb{B}(0,T) $ solves the problem \cref{equ-phizeta}, \cref{ic-pertur}, and the anti-derivative variable,
	\begin{equation}
		(\Phi, \Psi_1)(x_1,t) = \int_{-\infty}^{x_1} \int_{\Torus^2} (\phi, \psi_1)(y_1,\xp,t) d\xp dy_1, \quad x_1 \in \R, t>0,
	\end{equation}
	exists and belongs to the $ C(0,T; H^4(\R)) $ space, where $ \psi_1= m_1- \mt_1 = (\rhot+\phi)(\ut_1 + \zeta_1) - \rhot \ut_1 $.
	Then there exist $  \delta_0>0, \gamma_0>0 $ and $ \nu_0>0 $ such that, if 
	\begin{equation*}
		\delta \leq \delta_0, \quad \e \leq \gamma_0 \abso{\rhob^+ -\rhob^-},
	\end{equation*}
	and
	\begin{equation}\label{apriori}
		\sup\limits_{t\in(0,T)} \big( \norm{\Phi,\Psi_1}_{L^2(\R)} + \norm{ \phi,\zeta}_{H^3(\Omega)} \big) \leq \nu_0,
	\end{equation}
	then it holds that
	\begin{equation}\label{est-apriori}
		\begin{aligned}
			& \sup_{t\in(0,T)} \big( \norm{\Phi,\Psi_1}^2_{L^2(\R)} + \norm{ \phi,\zeta}_{H^3(\Omega)}^2 \big) \\
			& \qquad \quad + \int_0^T \Big( \norm{\abso{(\ut_1^s)'}^{\frac{1}{2}} \Psi_1}_{L^2(\R)}^2 + \norm{\zeta_1}_{L^2(\Omega)}^2 + \norm{\phi, \nabla\zeta}_{H^3(\Omega)}^2 \Big) dt \\
			& \qquad \lesssim \norm{\Phi_0, \Psi_{01}}_{L^2(\R)}^2 + \norm{ \phi_0, \zeta_0 }_{H^3(\Omega)}^2 + \e.
		\end{aligned}
	\end{equation}
%	where the constant $ C_0>0 $ is independent of $ \e, \delta, \nu $ and $ T. $
\end{Prop}

\begin{Thm}[Local existence]\label{Thm-local}
	Under the assumptions of \cref{Thm-pertur}, there exist $ T_0>0 $, $ \delta_0>0 $, $ \gamma_0>0 $ and $ \nu_0>0 $ such that if
	\begin{align*}
		& \delta \leq \delta_0, \quad \e \leq \gamma_0 \abso{\rhob^+ -\rhob^-}\ \text{ and } \ \norm{ \phi_0, \psi_0}_{H^3(\Omega)} + \norm{ \Phi_0, \Psi_{01}}_{L^2(\R)} \leq \nu_0,
	\end{align*}
%	there exist constants $ \nu_0>0, $ and $ T_0>0 $ such that if 
%	\begin{align*}
%		\norm{v_0, \wv_0}_{H^5(\Torus^3)} + \norm{\phi_0, \psi_0}_{H^2(\Omega)} + \norm{\Phi_0, \Psi_{01}}_{L^2(\R)} \leq \nu_0,
%	\end{align*}
	then the problem \cref{equ-phizeta}, \cref{ic-pertur} admits a unique solution $ (\phi, \zeta) \in \mathbb{B}(0,T_0). $ Moreover, the anti-derivative,
	\begin{align*}
		(\Phi,\Psi_1)(x_1,t) := \int_{-\infty}^{x_1} \int_{\Torus^2} (\phi,\psi_1)(x,t) d\xp, \qquad x_1 \in\R, t>0,
	\end{align*}
	exists and belongs to the $ C(0,T_0; H^4(\R)) $ space, where $ \psi_1 = m_1- \mt_1 = (\rhot+\phi) (\ut_1 + \zeta_1) - \rhot \ut_1 $.
\end{Thm}

\cref{Prop-apriori,Thm-local} are proved in \cref{Sec-esti,Sec-local}, respectively. In \cref{Sec-stability}, we first prove the time-asymptotic behaviors, \cref{asymp,exp-pert}, then we can finish the proof of Theorems \ref{Thm-pertur} and \ref{Thm}.

\vspace{.3cm}

\section{A priori estimates}\label{Sec-esti}

This section is devoted to the proof of \cref{Prop-apriori}. 

\vspace{.2cm}

Recall the notations $ \delta=\abso{\rhob_+-\rhob_-} $ and $ \e = \norm{(v_0, \wv_0)}_{H^6(\Torus^3)} $. In addition, we denote
\begin{equation}\label{nu}
%	\delta := \abso{\rhob_+ - \rhob_-}, \ \e := \norm{v_0, \wv_0}_{H^5(\Torus^3)} \ \ \text{and} \ \ 
	\nu := \sup\limits_{t\in(0,T)} \big( \norm{\Phi,\Psi_1}_{L^2(\R)} + \norm{ \phi,\zeta}_{H^3(\Omega)} \big).
\end{equation}
It follows from \cref{Lem-GN} that
\begin{align}
	\sup_{t\in(0,T)} \norm{ \phi,\zeta}_{W^{1,\infty}(\Omega)} 
	& \lesssim \sup_{t\in(0,T)} \Big\{ \norm{\nabla( \phi,\zeta)}^{\frac{1}{2}}_{H^1(\Omega)} \norm{ \phi,\zeta}^{\frac{1}{2}}_{H^1(\Omega)} + \norm{\nabla^2 ( \phi,\zeta)}_{H^1(\Omega)}^{\frac{1}{2}} \norm{( \phi,\zeta)}_{H^1(\Omega)}^{\frac{1}{2}} \notag \\
	& \qquad\qquad\quad + \norm{\nabla^2 (\phi,\zeta)}^{\frac{3}{4}}_{H^1(\Omega)} \norm{ \phi,\zeta}^{\frac{1}{4}}_{H^1(\Omega)} \Big\} \notag \\
	& \lesssim \sup_{t\in(0,T)} \norm{ \phi,\zeta}_{H^3(\Omega)} \notag \\
	&  \lesssim \nu. \label{small-phizeta}
\end{align}
And the perturbation of momentum, $ \psi = \rho \zeta + \phi \uvt $, satisfies that
\begin{equation*}
	\sup_{t\in(0,T)} \norm{\psi}_{W^{1,\infty}(\Omega)} \lesssim \sup_{t\in(0,T)} \norm{\psi}_{H^3(\Omega)} \lesssim \sup_{t\in(0,T)} \norm{ \phi,\zeta}_{H^3(\Omega)} \lesssim \nu.
\end{equation*}
Combining the Sobolev inequality and \cref{Lem-od-md}, one can prove that 
\begin{align}
	\sup_{t\in(0,T)} \norm{\Phi,\Psi_1}_{W^{3, \infty}(\R)} & \lesssim \sup_{t\in(0,T)} \norm{\Phi,\Psi_1}_{H^4(\R)} \notag \\
	& \lesssim \sup_{t\in(0,T)} \big( \norm{\Phi,\Psi_1}_{L^2(\R)} + \norm{\phi^{\od},\psi_1^{\od}}_{H^3(\R)} \big) \notag \\
	& \lesssim \sup_{t\in(0,T)} \big( \norm{\Phi,\Psi_1}_{L^2(\R)} + \norm{\phi,\psi_1}_{H^3(\Omega)} \big) \notag \\
	& \lesssim \nu. \label{small-PhiPsi}
\end{align}
By \cref{bdd-rhot,small-tilde,small-phizeta}, if $ \delta>0, $ $ \gamma_0>0 $ and $ \nu >0 $ are small, then
\begin{equation}\label{small-uv}
	  \sup_{t\in(0,T)} \norm{\nabla\rho(x,t)}_{L^\infty(\Omega)} + \sup_{t\in(0,T)} \norm{\uv(x,t)}_{W^{1,\infty}(\Omega)} \lesssim \delta + \e + \nu.
\end{equation}
and 
\begin{align*}
	\frac{1}{4}\rhob^+ \leq \inf_{\substack{x\in\Omega \\ t\in(0,T)}} \rho(x,t)  \leq \sup_{\substack{x\in\Omega \\ t\in(0,T)}} \rho(x,t) \leq 4 \rhob^-.
\end{align*}
Moreover, if $ \delta >0 $ is small, then by \cref{small-us}, the term \cref{alpha} satisfies that
\begin{equation}\label{posit-alpha}
	\alpha = p'(\rhot^s) - \abso{\ut_1^s}^2 \geq \frac{1}{2} p'(\rhot^s) >0 \qquad \forall x\in\Omega, t\geq 0.
\end{equation}

\begin{Lem}\label{Lem-est1}
	Under the assumptions of \cref{Prop-apriori}, there exist $ \delta_0>0, \gamma_0>0 $ and $ \nu_0>0 $ such that if $ \delta \leq \delta_0 $, $ \e \leq \gamma_0 \delta $ and $ \nu \leq \nu_0, $ then
	\begin{equation}\label{est1}
		\begin{aligned}
			& \sup_{t\in(0,T)} \norm{\Phi,\Psi_1}^2_{L^2(\R)} + \int_0^T \Big(\norm{\abso{(\ut_1^s)'}^{\frac{1}{2}} \Psi_1}_{L^2(\R)}^2 + \norm{\p_1 \Psi_1}_{L^2(\R)}^2 \Big)  dt \\
			& \qquad  \lesssim \norm{\Phi_0,\Psi_{01}}_{L^2(\R)}^2 + (\delta + \nu) \int_0^T \norm{\p_1 \Phi}_{L^2(\R)}^2 dt + \nu \int_0^T \norm{\nabla \phi, \nabla \psi_1}_{L^2(\Omega)}^2 dt + \e.
		\end{aligned}
	\end{equation}
\end{Lem}

\begin{proof}
	Multiplying $ \Phi $ on \cref{equ-anti}$ _1 $ and $ \frac{\Psi_1}{\alpha} $ on \cref{equ-anti}$ _2, $ respectively, and summing the resulting equations up, one can get that
	\begin{align}
		& \p_t \Big( \frac{\Phi^2}{2} + \frac{\Psi_1^2}{2\alpha} \Big) + \beta \Psi_1^2 + \frac{\mut}{\alpha \rhot^s} \abso{\p_1 \Psi_1}^2 = \p_1 (\cdots) + \underbrace{\frac{\mut \p_1 \alpha}{\alpha^2 \rhot^s} \Psi_1 (\p_1 \Psi_1-\ut_1^s \p_1 \Phi)}_{I_1} + \underbrace{\frac{\mut \ut_1^s}{\alpha \rhot^s } \p_1 \Phi \p_1 \Psi_1}_{I_2} \notag \\
		& \qquad + \underbrace{\Phi G_1 + \frac{\Psi_1}{\alpha}  G_2}_{I_3} + \underbrace{\frac{\Psi_1 }{\alpha} \int_{\Torus^2} (q_1+q_2) d\xp - \p_1 \Big(\frac{\Psi_1}{\alpha}\Big) \int_{\Torus^2} (q_3+q_4) d\xp}_{I_4}, \label{eq1}
	\end{align} 
	where 
	\begin{equation}\label{beta}
		\beta = - \p_t \Big(\frac{1}{2\alpha} \Big) - \p_1 \Big( \frac{\ut_1^s}{\alpha} \Big)
	\end{equation}
	and $ (\cdots) = - \Phi \Psi_1 - \frac{\ut_1^s}{\alpha} \Psi_1^2 + \frac{\mut}{\alpha \rhot^s} \Psi_1 (\p_1 \Psi_1 - \ut_1^s \p_1 \Phi)  + \frac{\Psi_1}{\alpha} \int_{\Torus^2} (q_3 + q_4) d\xp. $
	
	By  \cref{ode-shock}$ _1 $, \cref{alpha} and Lemmas \ref{Lem-per}, \ref{Lem-shift} and \ref{Lem-shock}, one can verify that 
	\begin{align*}
		\beta & = \frac{1}{2\alpha^2} \big( \p_t \alpha + 2 \ut_1^s \p_1 \alpha - 2 (\ut_1^s)' \alpha \big) \\
		& = - \frac{1}{2\alpha^2} \Big( \frac{s \rhot^s }{s-\ut_1^s}p''(\rhot^s) + 2 p'(\rhot^s) \Big) (\ut_1^s)' + O(1) \abso{\ut_1^s} \abso{(\ut^s_1)'}.
	\end{align*}
	Thus, if $ \delta>0 $ is small, by \cref{small-us,positive-s} and the fact that $ (\ut_1^s)'<0 $, one has that
	\begin{equation}\label{posit-beta}
		\beta \gtrsim \abso{(\ut^s_1)'} \qquad \forall x\in\Omega, t\geq 0. 
	\end{equation}	
	We now estimate each $ I_i $ from $ i=1 $ to $ 4 $.
	It follows from \cref{Lem-shock,small-us} that
	\begin{align}
		\norm{I_1}_{L^1(\R)} & \lesssim \int_\R \abso{(\ut^s_1)'} \abso{\Psi_1} (\abso{\p_1 \Psi_1}+\abso{\p_1 \Phi}) dx_1 \notag \\
		& \lesssim \delta  \norm{ \abso{(\ut^s_1)'}^{\frac{1}{2}} \Psi_1 }_{L^2(\R)}^2 + \delta \norm{\p_1 \Phi, \p_1 \Psi_1}_{L^2(\R)}^2, \label{I-1}
	\end{align}
	and
	\begin{align}
		\norm{I_2}_{L^1(\R)} \lesssim \delta \norm{\p_1 \Phi, \p_1 \Psi_1}_{L^2(\R)}^2. \label{I-2}
	\end{align}
	Using \cref{est-G} with the assumption that $ \e \leq \gamma_0 \delta $, one has that
	\begin{align}
		\norm{I_3}_{L^1(\R)} & \lesssim \e \delta^{-1} e^{-ct} \sup_{t\in(0,T)} \norm{\Phi, \Psi_1}_{L^2(\R)}^2 + \e e^{-ct} \notag \\
		& \lesssim \gamma_0 e^{-ct} \sup_{t\in(0,T)} \norm{\Phi, \Psi_1}_{L^2(\R)}^2 + \e e^{-ct}. \label{I-3}
	\end{align}
	And using \cref{q,small-PhiPsi,nu}, the term $ I_4 $ satisfies that
	\begin{align}
		\norm{I_4}_{L^1(\R)} & \lesssim \norm{\Psi_1}_{W^{1,\infty}(\R)} \big( \e e^{-ct} \norm{\phi,\psi_1}_{L^2(\Omega)} + \norm{\phi, \psi_1}_{L^2(\Omega)}^2 \big) \notag \\
		& \lesssim \e \nu^2 e^{-ct} + \nu \norm{\phi, \psi_1}_{L^2(\Omega)}^2. \label{I-4}
 	\end{align} 
 	By \cref{zero-md}, the non-zero mode of $ (\phi,\psi_1) $ satisfies that
 	\begin{align*}
 		\int_{\Torus^2} (\phi^\md, \psi_1^\md) d\xp = 0 \qquad \forall x_1 \in \R, t\geq 0.
 	\end{align*}
 	Then it follows from \cref{poincare} that
 	\begin{align}
 	\norm{\phi}_{L^2(\Omega)}^2 & = \norm{\phi^{\od}}_{L^2(\R)}^2 + \norm{\phi^{\md}}_{L^2(\Omega)}^2 \notag \\
 	& \lesssim \norm{\phi^{\od}}_{L^2(\R)}^2 + \norm{\nabla_{\xp}\phi^{\md}}_{L^2(\Omega)}^2 \notag \\
 	& \lesssim \norm{\p_1 \Phi}_{L^2(\R)}^2 + \norm{\nabla \phi}_{L^2(\Omega)}^2, \label{key1}
 	\end{align}
 	where the fact that $ \nabla_{\xp} \phi^{\md} = \nabla_{\xp} \phi $ is used. And it is similar to prove that
 	\begin{equation}\label{key2}
 	\norm{\psi_1}_{L^2(\Omega)}^2 \lesssim \norm{\p_1 \Psi_1}_{L^2(\R)}^2 + \norm{\nabla \psi_1}_{L^2(\Omega)}^2.
 	\end{equation}
 
	By integrating \cref{eq1} over $ \R\times (0,T) $ and collecting \cref{posit-alpha,posit-beta,I-1,I-2,I-3,I-4,key1,key2}, one can obtain \cref{est1} directly.
\end{proof}

\vspace{.2cm}

\begin{Lem}\label{Lem-est2}
	Under the assumptions of \cref{Prop-apriori}, there exist $ \delta_0>0, \gamma_0>0 $ and $ \nu_0>0 $ such that if $ \delta \leq \delta_0 $, $ \e \leq \gamma_0 \delta $ and $ \nu \leq \nu_0, $ then
	\begin{equation}\label{est2}
		\begin{aligned}
			& \sup_{t\in(0,T)} \norm{\p_1 \Phi}_{L^2(\R)}^2  + \int_0^T \norm{\p_1 \Phi}^2_{L^2(\R)} dt \\
			& \qquad \lesssim \norm{\Phi_0,\Psi_{01}}_{L^2(\R)}^2 + \norm{\phi_0}_{L^2(\Omega)}^2 + \nu \int_0^T \norm{\nabla \phi, \nabla \psi_1}^2_{L^2(\Omega)} dt + \e.
		\end{aligned}
	\end{equation}
\end{Lem}

\begin{proof}
	It follows from \cref{equ-anti}$ _1 $ that
	\begin{align*}
		\p_t\Psi_1 \p_1 \Phi = \p_t (\Psi_1 \p_1 \Phi) + \p_1 (\Psi_1 \p_1 \Psi_1) - \Psi_1 \p_1 G_1 - \abso{\p_1 \Psi_1}^2,
	\end{align*} 
	and
	\begin{align*}
		& - \mut \p_1 \Big[\frac{1}{\rhot^s} \big(\p_1 \Psi_1 - \ut_1^s \p_1 \Phi\big) \Big] \p_1 \Phi \\
		& \qquad = \p_t \Big( \frac{\mut}{2\rhot^s} \abso{\p_1 \Phi}^2 \Big) - \frac{\mut}{\rhot^s} \p_1 G_1 \p_1 \Phi + \frac{\mut \p_1 \rhot^s}{\abso{\rhot^s}^2} \p_1 \Psi_1 \p_1 \Phi \\
		& \qquad \quad - \frac{\mut}{2} \Big[ \p_t \Big( \frac{1}{ \rhot^s}\Big) - \p_1 \Big(\frac{ \ut_1^s}{\rhot^s} \Big) \Big] \abso{\p_1 \Phi}^2  + \p_1 \Big(\frac{\mut \ut_1^s}{2\rhot^s} \abso{\p_1 \Phi}^2 \Big).
	\end{align*}
	Then by multiplying $ \p_1 \Phi $ on \cref{equ-anti}$ _2 $, one can get that
	\begin{equation}\label{eq2}
		\begin{aligned}
		& \p_t \Big( \frac{\mut}{2\rhot^s} \abso{\p_1 \Phi}^2 + \Psi_1\p_1 \Phi \Big)  + \alpha \abso{\p_1 \Phi}^2 = \p_1 (\cdots) \\
		& \qquad \underbrace{ - \Big(2\ut_1^s + \frac{\mut \p_1\rhot^s}{\abso{\rhot^s}^2}\Big) \p_1 \Phi \p_1 \Psi_1 + \frac{\mut}{2} \Big[ \p_t \Big( \frac{1}{ \rhot^s}\Big) - \p_1 \Big(\frac{ \ut_1^s}{\rhot^s} \Big) \Big] \abso{\p_1 \Phi}^2  + \abso{\p_1 \Psi_1}^2}_{I_5} \\
		& \qquad \underbrace{ - \Phi \p_1 G_2 + \Big(\Psi_1 + \frac{\mut}{\rhot^s} \p_1 \Phi \Big) \p_1 G_1 }_{I_6} \\
		& \qquad \underbrace{+ \p_1 \Phi \int_{\Torus^2} (q_1+q_2) d\xp - \p_1^2 \Phi \int_{\Torus^2} (q_3+q_4) d\xp}_{I_7}, 
	\end{aligned}
	\end{equation}
	where $ (\cdots) = -\Psi_1 \p_1 \Psi_1 + \Phi G_2 - \frac{\mut \ut_1^s}{2\rhot^s} \abso{\p_1 \Phi}^2 + \p_1 \Phi \int_{\Torus^2} (q_3+q_4) d\xp. $ 
	
	By \cref{Lem-shock,small-us}, one has that 
	\begin{equation}\label{I-5}
		\norm{I_5}_{L^1(\R)} \lesssim \delta \norm{\p_1 \Phi}_{L^2(\R)}^2 + \norm{\p_1 \Psi_1}_{L^2(\R)}^2.
	\end{equation}
	It follows from \cref{Lem-F} that
	\begin{align}
		\norm{I_6}_{L^1(\R)} & \lesssim \e \delta^{\frac{1}{2}} e^{-ct} (\norm{\Phi}_{H^1(\R)} + \norm{\Psi_1}_{L^2(\R)}) \notag \\
		& \lesssim \e e^{-ct} \sup_{t\in(0,T)} \Big(\norm{\Phi}_{H^1(\R)}^2 + \norm{\Psi_1}_{L^2(\R)}^2 \Big) + \e e^{-ct}. \label{I-6}
	\end{align}
	And similar to \cref{I-4}, one can use \cref{small-PhiPsi,key1,key2,nu} to verify that
	\begin{align}
		\norm{I_7}_{L^1(\R)} & \lesssim \norm{\p_1 \Phi}_{W^{1,\infty}(\R)} \big( \e e^{-ct} \norm{\phi,\psi_1}_{L^2(\Omega)} + \norm{\phi, \psi_1}_{L^2(\Omega)}^2 \big) \notag \\
		& \lesssim \e e^{-ct} + \nu \norm{\p_1\Phi, \p_1 \Psi_1}_{L^2(\R)}^2 + \nu \norm{\nabla \phi, \nabla \psi_1}_{L^2(\Omega)}^2. \label{I-7}
	\end{align}	
	
	Note that $ \Phi_0' = \phi_0^{\od} $, then it follows from \cref{Lem-od-md} that $ \norm{\Phi_0'}_{L^2(\R)} \lesssim \norm{\phi_0}_{L^2(\Omega)}. $
	Collecting \cref{I-5,I-6,I-7} and using \cref{Lem-est1}, one can integrate \cref{eq2} over $ \R \times (0,T) $ to obtain \cref{est2}.
	
\end{proof}

\vspace{.2cm}

\begin{Lem}\label{Lem-est3}
	Under the assumptions of \cref{Prop-apriori}, there exist $ \delta_0>0 $ and $ \gamma_0>0 $ such that if $ \delta\leq \delta_0 $ and $ \e \leq \gamma_0 \delta $, then
	\begin{equation}\label{bdd-nab-psi}
		\norm{\nabla \psi_1}_{L^2(\Omega)} \lesssim  \norm{\p_1 \Phi, \p_1 \Psi_1}_{L^2(\R)} + \norm{\nabla \zeta_1}_{L^2(\Omega)}  + (\delta + \e + \nu)  \norm{\nabla\phi}_{L^2(\Omega)},
	\end{equation}
	and
	\begin{equation}\label{bdd-z1}
		\begin{aligned}
			\norm{\zeta_1}_{L^2(\Omega)} \lesssim  \norm{\p_1 \Phi,\p_1 \Psi_1}_{L^2(\R)} + \norm{\nabla \zeta_1}_{L^2(\Omega)} + (\delta + \e + \nu)  \norm{\nabla\phi}_{L^2(\Omega)}.
		\end{aligned}
	\end{equation}
\end{Lem}

\begin{proof}
	Note that $ \psi_1 = \rho \zeta_1 + \ut_1 \phi $. It follows from \cref{small-tilde,small-phizeta,key1,key2} that 
	\begin{align*}
		\norm{\nabla \psi_1}_{L^2(\Omega)} & \lesssim \norm{\nabla \zeta_1}_{L^2(\Omega)} + (\delta + \e + \nu) \norm{\nabla\phi}_{L^2(\Omega)} + (\delta+\e) \norm{\phi,\zeta_1}_{L^2(\Omega)} \\
%		& \lesssim \norm{\nabla \zeta_1}_{L^2(\Omega)} + (\delta + \e + \nu)  \norm{\nabla\phi}_{L^2(\Omega)} + (\delta+\e) \norm{\phi,\zeta_1}_{L^2(\Omega)} \\
		& \lesssim \norm{\nabla \zeta_1}_{L^2(\Omega)}  + (\delta + \e + \nu)  \norm{\nabla\phi}_{L^2(\Omega)}  + (\delta+\e) \norm{\nabla\psi_1}_{L^2(\Omega)} \\
		& \quad + (\delta+\e) \norm{\p_1 \Phi, \p_1 \Psi_1}_{L^2(\R)}.
	\end{align*}
%	where we have used \cref{key1,key2}.
	Thus, if $ \delta>0 $ and $ \gamma_0>0 $ are small, then \cref{bdd-nab-psi} holds true. Then with \cref{bdd-nab-psi}, one can use  \cref{key1,key2,small-tilde} again to obtain
	\begin{align*}
		\norm{\zeta_1}_{L^2(\Omega)} & \lesssim \norm{\psi_1}_{L^2(\Omega)} + (\delta+\e) \norm{\phi}_{L^2(\Omega)} \\ 	
		& \lesssim \norm{\p_1 \Psi_1}_{L^2(\R)} + \norm{\nabla \psi_1}_{L^2(\Omega)} + (\delta+\e) \big(\norm{\p_1 \Phi}_{L^2(\R)} + \norm{\nabla\phi }_{L^2(\Omega)}\big)  \\
		& \lesssim \norm{\p_1 \Phi, \p_1 \Psi_1}_{L^2(\R)} + \norm{\nabla \zeta_1}_{L^2(\Omega)} + (\delta + \e + \nu)  \norm{\nabla\phi}_{L^2(\Omega)},
	\end{align*}
	which completes the proof.
\end{proof}

\vspace{.3cm}

Now we return to the original system \cref{equ-phizeta} to estimate the 3D perturbation $ (\phi, \zeta) $.

\begin{Lem}\label{Lem-est4}
Under the assumptions of \cref{Prop-apriori}, there exist $ \delta_0>0, \gamma_0>0 $ and $ \nu_0>0 $ such that if $ \delta \leq \delta_0 $, $ \e \leq \gamma_0 \delta $ and $ \nu \leq \nu_0, $ then
\begin{equation}\label{est3}
	\begin{aligned}
		& \sup_{t\in(0,T)} \norm{\phi, \zeta}^2_{L^2(\Omega)} + \int_0^T \norm{\nabla \zeta}_{L^2(\Omega)}^2  dt \\
		& \qquad \lesssim \norm{\Phi_0, \Psi_{01}}^2_{L^2(\R)} + \norm{\phi_0,  \zeta_0}^2_{L^2(\Omega)}  +  (\delta+\e+\nu) \int_0^T \norm{\nabla\phi }_{L^2(\Omega)}^2 dt + \e.
	\end{aligned} 
\end{equation}

\end{Lem}

\begin{proof}

Define 
\begin{equation*}
	\Xi(\rho,\rhot) := \int_{\rhot}^\rho \frac{p(s)-p(\rhot)}{s^2} ds = \frac{\gamma}{(\gamma-1)\rho} \left[p(\rho)-p(\rhot) -p'(\rhot) (\rho-\rhot)\right],
\end{equation*}
which satisfies $ \Xi(\rho,\rhot) \sim \abso{\phi}^2 .$ 
Then it follows from \cref{NS}$ _1 $ that
\begin{equation}\label{eq3}
\begin{aligned}
& \p_t(\rho \Xi) + \dv(\rho \Xi \uv) = \rho \big(\p_t \Xi + \uv\dnab \Xi\big) \\
& \qquad = - \left[p(\rho)-p(\rhot) -p'(\rhot) \phi \right] \dv \uvt -\dv\left[ \big(p(\rho)-p(\rhot)\big) \zeta \right] \\
& \qquad \quad + \zeta \cdot \nabla (p(\rho)-p(\rhot)) + \frac{p'(\rhot)}{\rhot} (g_1 - \zeta \cdot \nabla\rhot) \phi.
\end{aligned}
\end{equation}
Multiplying $ \cdot\zeta $ on \cref{equ-phizeta}$ _2 $ yields that 
\begin{equation}\label{eq4}
\begin{aligned}
	& \p_t \Big( \frac{1}{2} \rho \abso{\zeta}^2 \Big) + \mu \abso{\nabla\zeta}^2  + (\mu+\lambda)\abso{\dv \zeta}^2 = \dv (\cdots) - \zeta \cdot \nabla (p(\rho)-p(\rhot))  \\
	& \qquad  - \rho (\zeta\cdot\nabla \uvt) \cdot \zeta - \phi (\p_t\uvt + \uvt \dnab \uvt) \cdot \zeta + (\gv_2-g_1\uvt) \cdot \zeta,
\end{aligned}
\end{equation}
where $ (\cdots) = \frac{\mu}{2} \nabla \big(\abso{\zeta}^2\big) + (\mu+\lambda)\zeta\dv\zeta - \frac{1}{2} \rho \abso{\zeta}^2 \uv. $
Then summing up \cref{eq3,eq4} yields that
\begin{align}
	\p_t \Big(\rho \Xi + \frac{1}{2} \rho \abso{\zeta}^2 \Big) + \mu \abso{\nabla\zeta}^2  + (\mu+\lambda)\abso{\dv \zeta}^2 = \dv (\cdots) + I_8 + I_9, \label{eq4.5}
\end{align}
where $ \dv(\cdots) $ denotes the summation of some terms which vanish after integration on $ \Omega $, and
\begin{align*}
	I_8 & = - \left[p(\rho)-p(\rhot) -p'(\rhot) \phi \right] \dv \uvt - \frac{p'(\rhot)}{\rhot} \nabla \rhot \cdot \zeta \phi \\
	& \quad\ - \rho (\zeta\cdot\nabla \uvt) \cdot \zeta - \phi (\p_t\uvt + \uvt \dnab \uvt) \cdot \zeta, \\
	I_9 & = \frac{p'(\rhot)}{\rhot} g_1 \phi +  (\gv_2-g_1\uvt) \cdot \zeta.
\end{align*}
Note that 
\begin{equation*}
	\begin{aligned}
	\rhot & = \rhot^s + \big(\rho^s_{st+X}-\rhot^s \big) + v_- (1-\eta_{st+X}) + v_+ \eta_{st+X}, \\
	\mvt & = \mt_1^s \E_1 + \big[(m_1^s)_{st+Y} -\mt_1^s\big] \E_1 + \wv_- (1-\eta_{st+Y}) + \wv_+ \eta_{st+Y}.
\end{aligned}
\end{equation*}
This, together with Lemmas \ref{Lem-per} and \ref{Lem-shift}, yields that
\begin{equation}\label{star}
	\begin{aligned}
	\nabla \rhot & = (\rhot^s)' \E_1 + O(1) \e e^{-ct}, \\
	\p_i \uvt & =  (\ut_1^s)' \delta_{1i} \E_1 + O(1) \e e^{-ct} \qquad \text{for } \ i=1,2,3, \\
	\p_t \uvt & = \p_t \ut_1^s \E_1 + O(1) \e e^{-ct}.
	\end{aligned}
\end{equation}
Using \cref{small-tilde,key1,bdd-z1}, one can get that
\begin{align}
	\norm{I_8}_{L^1(\Omega)} & \lesssim (\delta+\e) \norm{\phi, \zeta_1}_{L^2(\Omega)}^2 + \e e^{-ct} \sup_{t\in(0,T)} \norm{\zeta}_{L^2(\Omega)}^2 \notag \\
	& \lesssim \norm{\p_1 \Phi, \p_1\Psi_1}_{L^2(\R)} + (\delta+\e) \norm{\nabla\phi, \nabla\zeta_1}_{L^2(\Omega)} + \e \nu e^{-ct}. \label{I-8}
\end{align}
And it follows from \cref{Lem-F} and \cref{nu} that
\begin{align}
	\norm{I_9}_{L^1(\Omega)} \lesssim \e \delta^{\frac{1}{2}} e^{-ct} \norm{\phi,\zeta}_{L^2(\Omega)} \lesssim \e\delta^{\frac{1}{2}} \nu e^{-ct}. \label{I-9}
\end{align}

Thus, by integrating \cref{eq4.5} over $ \Omega \times (0,T) $ and making use of \cref{bdd-nab-psi}, \cref{I-8}, \cref{I-9} and Lemmas \ref{Lem-est1} and \ref{Lem-est2}, one can obtain \cref{est3}.
%\begin{equation}\label{ineq-1-1}
%	\begin{aligned} 
%& \sup_{t\in(0,T)} \norm{(\phi, \zeta)}_{L^2(\Omega)}^2 + \int_0^T \norm{\nabla \zeta}_{L^2(\Omega)}^2 dt \\
%& \qquad \lesssim \norm{\phi_0, \zeta_0}_{L^2(\Omega)}^2 + (\delta+\e) \int_0^T \norm{\phi,\zeta_1}_{L^2(\Omega)}^2 dt + \e\delta^{\frac{1}{2}} \nu.
%%& \qquad \lesssim \norm{\phi_0, \zeta_0}_{L^2(\Omega)}^2 + \delta \int_0^T \norm{\p_1 \Phi, \p_1 \Psi_1}_{L^2(\R)}^2 dt + \delta \int_0^T \norm{\nabla\phi, \nabla \psi_1}_{L^2(\Omega)}^2 dt, \label{eq-3}
%\end{aligned}
%\end{equation}

\end{proof}

\vspace{.2cm}

\begin{Lem}\label{Lem-est5}
Under the assumptions of \cref{Prop-apriori}, there exist $ \delta_0>0, \gamma_0>0 $ and $ \nu_0>0 $ such that if $ \delta \leq \delta_0 $, $ \e \leq \gamma_0 \delta $ and $ \nu \leq \nu_0, $ then
\begin{equation}\label{est5}
	\begin{aligned}
		& \sup_{t\in(0,T)} \norm{\nabla \phi}^2_{L^2(\Omega)} + \int_0^T  \big(\norm{\nabla\phi}^2_{L^2(\Omega)} + \norm{\phi,\zeta_1}^2_{L^2(\Omega)}\big) dt \\
		& \qquad \lesssim \norm{\Phi_0, \Psi_{01}}^2_{L^2(\R)} + \norm{\phi_0}_{H^1(\Omega)}^2 + \norm{\zeta_0}^2_{L^2(\Omega)} + \nu \int_0^T \norm{\nabla^3 \zeta}_{L^2(\Omega)}^2 dt + \e.
	\end{aligned}
\end{equation}
\end{Lem}

\begin{proof}
By taking the gradient $ \nabla $ on \cref{equ-phizeta}$_1$ and multiplying the resulting equation by $ \cdot \frac{\nabla\phi }{\rho^2}, $ one has that
\begin{equation}\label{eq5}
	\begin{aligned}
	& \p_t \Big(\frac{\abso{\nabla \phi}^2}{2\rho^2}\Big) + \frac{1}{\rho} \nabla \dv \zeta \cdot \nabla\phi  = \dv \Big( -\frac{\abso{\nabla\phi }^2}{2\rho^2} \uv\Big) + I_{10}  + I_{11},  
\end{aligned}
\end{equation}
where 
\begin{equation*} 
	\begin{aligned}
	I_{10} & = \frac{-1}{\rho^2} \nabla\phi \cdot \big( \dv \zeta \nabla \rhot + \nabla\uvt \nabla \phi + \nabla \dv \uvt \phi - \frac{1}{2} \dv \uvt \nabla \phi \\
	& \qquad\qquad\qquad + \nabla\zeta \nabla\rhot + \nabla^2 \rhot \zeta - \nabla g_1 \big), \\
%	& = O(1) \abso{\nabla\phi} \big[ (\e+\delta) (\abso{\nabla \zeta} + \abso{\nabla\phi} + \abso{\phi} + \abso{\zeta_1}) + \e e^{-ct} \abso{\zeta} + \abso{\nabla g_1} \big],
	I_{11} & = \frac{\dv \zeta}{2\rho^2} \abso{\nabla\phi }^2 - \frac{1}{\rho^2} \nabla\phi  \cdot \nabla \zeta \nabla\phi.
\end{aligned}
\end{equation*}
Using \cref{est-g,key1,bdd-z1}, one can verify that
\begin{align}
	\norm{I_{10}}_{L^1(\Omega)} & \lesssim (\e+\delta) \big(\norm{\phi}_{H^1(\Omega)}^2 + \norm{\nabla \zeta}_{L^2(\Omega)}^2 + \norm{ \zeta_1}_{L^2(\Omega)}^2 \big)  + \e e^{-ct} \norm{\zeta,\nabla\phi}_{L^2(\Omega)} \notag \\
	& \lesssim (\e+\delta) \big(\norm{\p_1\Phi,\p_1 \Psi_1}_{L^2(\R)}^2 +  \norm{\nabla\phi,\nabla\zeta}_{L^2(\Omega)}^2\big) +  \e \nu e^{-ct}. \label{I-10}
\end{align}
It follows from \cref{Lem-GN} and \cref{nu} that
\begin{align}
	\norm{\nabla\zeta}_{L^\infty(\Omega)} 
	& \lesssim \norm{\nabla^2 \zeta }_{L^2(\Omega)}^{\frac{1}{2}} \norm{\nabla\zeta}_{L^2(\Omega)}^{\frac{1}{2}} + \norm{\nabla^2 \zeta }_{L^2(\Omega)} + \norm{\nabla^3 \zeta }_{L^2(\Omega)}^{\frac{3}{4}} \norm{\nabla\zeta}_{L^2(\Omega)}^{\frac{1}{4}}  \notag \\
	& \lesssim \nu + \nu^{\frac{1}{4}} \norm{\nabla^3 \zeta }_{L^2(\Omega)}^{\frac{3}{4}}. \label{Sobolev}
\end{align}
Then $ I_{11} $ satisfies that 
\begin{align}
	\norm{I_{11}}_{L^1(\Omega)} & \lesssim \norm{\nabla\zeta}_{L^\infty(\Omega)} \norm{\nabla\phi}_{L^2(\Omega)}^2 \notag \\
	& \lesssim \nu \norm{\nabla\phi}_{L^2(\Omega)}^2  + \nu \norm{\nabla\phi}_{L^2(\Omega)}^{\frac{5}{4}} \norm{\nabla^3 \zeta }_{L^2(\Omega)}^{\frac{3}{4}}. \label{I-11}
\end{align}

%and we have used the fact that
%\begin{align*}
%\frac{\p_t\rho}{\rho^3} -\dv \Big( \frac{\uv}{2\rho^2} \Big) =  - \frac{3}{2\rho^2} \dv \uv = - \frac{3}{2\rho^2} \ut_1' - \frac{3}{2\rho^2} \dv \zeta.
%\end{align*}

On the other hand, since it holds that
\begin{align}
	\p_t \zeta \cdot \nabla \phi = \p_t \big(\zeta\cdot\nabla \phi \big) - \dv (\zeta\p_t\phi ) + \dv \zeta  \p_t \phi, \label{fact1}
\end{align}
and 
\begin{align}
	& \frac{1}{\rho} \nabla \phi \cdot \left[ \mu \lap \zeta + (\mu+\lambda) \nabla\dv \zeta \right] \notag \\
	& \qquad = \frac{\mut}{\rho} \nabla\phi  \cdot \nabla\dv \zeta + \frac{\mu}{\rho} \nabla\phi  \cdot \big( \lap \zeta - \nabla\dv \zeta \big) \notag \\
	& \qquad = \frac{\mut}{\rho} \nabla\phi  \cdot \nabla\dv \zeta + \dv \Big(\frac{\mu  \nabla \phi \times \text{curl} \zeta }{\rho}  \Big) + \frac{\mu \nabla \rho \cdot \big(\nabla \phi \times \text{curl} \zeta \big) }{\rho^2}, \label{tool1}
\end{align}
then multiplying $ \frac{\nabla\phi }{\rho}\cdot $ on \cref{equ-phizeta}$ _2 $ yields that
\begin{equation}\label{eq6}
	\begin{aligned}
		\p_t \big(\zeta\cdot\nabla \phi \big) + \frac{p'(\rho)}{\rho} \abso{\nabla \phi}^2 - \frac{\mut}{\rho} \nabla\phi \cdot \nabla\dv \zeta =  \dv (\cdots) + \sum_{i=12}^{14} I_i,
	\end{aligned}
\end{equation}
where $ (\cdots) = \zeta\p_t\phi + \frac{ \mu  }{\rho} \nabla \phi \times \text{curl} \zeta $ and
\begin{align*}
	I_{12} & = - \dv \zeta  \p_t \phi, \notag \\
	I_{13} & = \frac{\mu \nabla \rho \cdot \big(\nabla \phi \times \text{curl} \zeta \big) }{\rho^2}  = \frac{\mu \nabla \rhot \cdot \big(\nabla \phi \times \text{curl} \zeta \big) }{\rho^2},  \\
	I_{14} & = - \frac{\nabla\phi}{\rho} \cdot \big[ \rho \uv \cdot \nabla\zeta + (p'(\rho) - p'(\rhot)) \nabla\rhot + \rho \zeta \dnab \uvt \\
	& \qquad\qquad\quad + \phi (\p_t \uvt + \uvt \dnab \uvt) - (\gv_2 - g_1 \uvt) \big].
\end{align*}
Here and hereafter ``$\times$'' denoting the cross product.

\vspace{.2cm}

It follows from \cref{equ-phizeta}$_1 $, \cref{key1,bdd-z1,star,est-g} that
\begin{align}
	\norm{I_{12}}_{L^1(\Omega)} & \lesssim \norm{\nabla\zeta}_{L^2(\Omega)} \big[ \norm{\nabla \zeta}_{L^2(\Omega)} + (\delta+\e+\nu) \norm{\phi}_{H^1(\Omega)} + \delta \norm{\zeta_1}_{L^2(\Omega)} \notag \\
	& \qquad\qquad\qquad + \e e^{-ct} \norm{\zeta}_{L^2(\Omega)} + \norm{g_1}_{L^2(\Omega)} \big] \notag \\
	& \lesssim \norm{\nabla\zeta}_{L^2(\Omega)}^2 + \norm{\p_1 \Phi, \p_1 \Psi_1}_{L^2(\R)}^2 + (\delta+\e+\nu) \norm{\nabla \phi}_{L^2(\Omega)}^2 + \e e^{-ct}. \label{I-12}
\end{align}
Similarly, one can prove that
\begin{align}
	\norm{I_{13}}_{L^1(\Omega)} & \lesssim (\delta+\e) \big(\norm{\nabla\phi}^2_{L^2(\Omega)} + \norm{\nabla\zeta}^2_{L^2(\Omega)} \big), \label{I-13} \\
	\norm{I_{14}}_{L^1(\Omega)} 
	& \lesssim \norm{\nabla\phi}_{L^2(\Omega)} \big[ \delta \norm{\nabla \zeta}_{L^2(\Omega)} + (\delta+\e) \norm{\phi,\zeta_1}_{L^2(\Omega)} + \e e^{-ct} \norm{\zeta}_{L^2(\Omega)} + \norm{g_1,\gv_2}_{L^2(\Omega)} \big] \notag \\
	& \lesssim \norm{\nabla\zeta}_{L^2(\Omega)}^2 + \norm{\p_1 \Phi, \p_1 \Psi_1}_{L^2(\R)}^2 + (\delta+\e) \norm{\nabla \phi}_{L^2(\Omega)}^2 + \e e^{-ct}. \label{I-14}
\end{align}

Thus, collecting \cref{I-10,I-11,I-12,I-13,I-14} and using Lemmas \ref{Lem-est1}--\ref{Lem-est4}, one can integrate the summation of $ \mut \cdot $\cref{eq5} and \cref{eq6} over $ \Omega \times (0,T) $ to obtain that
\begin{align*}
	& \sup_{t\in(0,T)} \norm{\nabla \phi}^2_{L^2(\Omega)} + \int_0^T  \norm{\nabla\phi}^2_{L^2(\Omega)} dt \\
	& \qquad \lesssim \norm{\Phi_0, \Psi_{01}}^2_{L^2(\R)} + \norm{\phi_0}_{H^1(\Omega)}^2 + \norm{\zeta_0}^2_{L^2(\Omega)} + \nu \int_0^T \norm{\nabla^3 \zeta}_{L^2(\Omega)}^2 dt + \e.
\end{align*}
This, together with \cref{key1,bdd-z1}, can yield \cref{est5}.

\end{proof}

\vspace{.2cm}

\begin{Lem}\label{Lem-est6}
Under the assumptions of \cref{Prop-apriori}, there exist $ \delta_0>0, \gamma_0>0 $ and $ \nu_0>0 $ such that if $ \delta \leq \delta_0 $, $ \e \leq \gamma_0 \delta $ and $ \nu \leq \nu_0, $ then
\begin{equation}\label{est6}
\begin{aligned}
	& \sup_{t\in(0,T)} \norm{\nabla\zeta}^2_{L^2(\Omega)} + \int_0^T \norm{\nabla^2 \zeta}_{L^2(\Omega)}^2 dt \\
	& \qquad \lesssim \norm{\Phi_0, \Psi_{01}}^2_{L^2(\R)} + \norm{\phi_0,\zeta_0}_{H^1(\Omega)}^2 + \nu \int_0^T \norm{\nabla^3 \zeta}_{L^2(\Omega)}^2 dt + \e.
\end{aligned}
\end{equation}	
\end{Lem}

\begin{proof}
	
Multiplying $ -\frac{\lap \zeta}{\rho} \cdot $ on  \cref{equ-phizeta}$ _2 $ and using the fact that 
\begin{equation}\label{tool2}
	\begin{aligned}
	\frac{1}{\rho} \nabla\dv\zeta\cdot\lap\zeta  & = \frac{\abso{\nabla\dv\zeta}^2}{\rho}  + \dv\Big[\frac{\dv\zeta}{\rho}  \big(\lap\zeta-\nabla\dv\zeta \big) \Big] \\
	& \quad + \frac{1}{\rho^2} \dv\zeta \nabla\rho \cdot \big( \lap \zeta - \nabla \dv\zeta \big),
\end{aligned}
\end{equation}
one can obtain that
\begin{equation}\label{eq7} 
	\p_t \Big(\frac{\abso{\nabla\zeta}^2}{2}\Big) + \frac{\mu}{\rho} \abso{\lap\zeta}^2 + \frac{\mu+\lambda}{\rho} \abso{\nabla\dv\zeta}^2 = \dv (\cdots) + I_{15} + I_{16},
\end{equation}
where $ (\cdots) = \nabla\zeta \p_t\zeta - \frac{\mu+\lambda}{\rho} \dv\zeta \big(\lap\zeta-\nabla\dv\zeta \big) $ and
\begin{align*}
I_{15} & = \lap \zeta \cdot \Big[ \uv\dnab \zeta + \frac{1}{\rho} \nabla (p(\rho) - p(\rhot)) +  \zeta \cdot \nabla \uvt + \frac{\phi}{\rho} (\p_t \uvt + \uvt \cdot \nabla \uvt) \Big] \\
& \quad - \frac{\mu+\lambda}{\rho^2} \dv \zeta \nabla\rhot \cdot (\lap \zeta - \nabla\dv \zeta)  - \frac{1}{\rho} \lap \zeta \cdot (\gv_2 - g_1 \uvt),  \\
%& = O(1) \abso{\nabla^2\zeta} \big( \abso{\nabla\zeta} + (\delta+\e) (\abso{\phi}+\abso{\zeta_1}) + \abso{\nabla\phi} + \e e^{-ct} \abso{\zeta} \big),  \\
I_{16} & = - \frac{\mu+\lambda}{\rho^2} \dv\zeta \nabla\phi \cdot(\lap\zeta-\nabla\dv\zeta).
\end{align*}
By \cref{est-g,key1,bdd-z1}, one has that
\begin{align}
	\norm{I_{15}}_{L^1(\Omega)} & \lesssim \norm{\nabla^2 \zeta}_{L^2(\Omega)} \big( \norm{\nabla \zeta}_{L^2(\Omega)} + \norm{\phi}_{H^1(\Omega)} + \delta\norm{\zeta_1}_{L^2(\Omega)} + \e e^{-ct} \big) \notag \\
	& \lesssim \norm{\nabla^2 \zeta}_{L^2(\Omega)} \big( \norm{\p_1 \Phi, \p_1 \Psi_1}_{L^2(\R)} + \norm{\nabla\phi, \nabla \zeta}_{L^2(\Omega)} + \e e^{-ct} \big). \label{I-15}
\end{align}
And it follows from \cref{nu,Sobolev} that
\begin{align}
	\norm{I_{16}}_{L^1(\Omega)} & \lesssim \norm{\nabla \zeta}_{L^\infty(\Omega)} \norm{\nabla\phi }_{L^2(\Omega)} \norm{\nabla^2 \zeta}_{L^2(\Omega)} \notag \\
	& \lesssim \nu \norm{\nabla\phi }_{L^2(\Omega)} \norm{\nabla^2 \zeta}_{L^2(\Omega)} +  \nu \norm{\nabla^2 \zeta}_{L^2(\Omega)}^{\frac{3}{4}} \norm{\nabla\phi}_{L^2(\Omega)}^{\frac{1}{4}} \norm{\nabla^2 \zeta}_{L^2(\Omega)} \notag \\
	& \lesssim \nu \norm{\nabla^2 \zeta}_{L^2(\Omega)}^2 + \nu  \norm{\nabla \phi}_{L^2(\Omega)}^2 + \nu \norm{\nabla^3 \zeta}_{L^2(\Omega)}^2. \label{I-16}
\end{align}

Note that $ \norm{\lap \zeta}_{L^2(\Omega)}^2 \geq c_0 \norm{\nabla^2\zeta}_{L^2(\Omega)}^2 $ for some constant $ c_0>0. $ Then integrating \cref{eq7} over $ \Omega\times(0,T) $ and using \cref{I-15}, \cref{I-16} and Lemmas \ref{Lem-est1}-\ref{Lem-est5}, one can obtain  \cref{est6}.
%, one has that
%\begin{align*}
%	& \sup_{t\in(0,T)} \norm{\nabla\zeta}_{L^2(\Omega)}^2 + \int_0^T \norm{\nabla^2 \zeta}_{L^2(\Omega)}^2 dt \lesssim \norm{\nabla\zeta_0}_{L^2(\Omega)}^2 + \int_0^T \norm{\nabla \zeta, \nabla\phi}_{L^2(\Omega)}^2 dt \\
%	& \qquad + \int_0^T \norm{\phi, \zeta_1}_{L^2(\Omega)}^2 dt + \e \sup_{t\in(0,T)} \norm{\zeta}_{L^2(\Omega)}^2 + \int_0^T \norm{\nabla \zeta}_{L^\infty(\Omega)} \norm{\nabla\phi }_{L^2(\Omega)} \norm{\nabla^2 \zeta}_{L^2(\Omega)} dt + \e.
%\end{align*}

\end{proof}

%\vspace{.2cm}

\begin{Lem}\label{Lem-est7}
Under the assumptions of \cref{Prop-apriori}, there exist $ \delta_0>0, \gamma_0>0 $ and $ \nu_0>0 $ such that if $ \delta \leq \delta_0 $, $ \e \leq \gamma_0 \delta $ and $ \nu \leq \nu_0, $ then
\begin{equation}\label{est7}
	\begin{aligned}
		& \sup_{t\in(0,T)} \norm{\nabla^2 \phi}^2_{L^2(\Omega)} + \int_0^T  \norm{\nabla^2 \phi}_{L^2(\Omega)}^2 dt \\
		& \qquad \lesssim \norm{ \Phi_0, \Psi_{01}}^2_{L^2(\R)} + \norm{\zeta_0}_{H^1(\Omega)}^2 + \norm{\phi_0}_{H^2(\Omega)}^2 + \nu \int_0^T \norm{\nabla^3 \zeta}_{L^2(\Omega)}^2 dt + \e.
	\end{aligned}
\end{equation}
\end{Lem}

\begin{proof}
Let $ i \in \{1,2,3\} $ be fixed.
By taking the second derivative $ \nabla \p_i  $ on \cref{equ-phizeta}$_1$ and multiplying the resulting equation by $ \cdot \frac{\nabla \p_i \phi}{\rho^2}, $
one can get that
\begin{equation}\label{eq8}
	\p_t \Big(\frac{\abso{\nabla \p_i   \phi}^2}{2\rho^2}\Big) + \frac{1}{\rho} \nabla \p_i  \dv\zeta \cdot \nabla \p_i\phi  = \dv \Big(-\frac{\abso{\nabla\p_i\phi }^2 }{2\rho^2} \uv \Big) + I_{17},
\end{equation}
where
\begin{equation*}
	\begin{aligned}
		I_{17} & =  - \frac{1}{\rho^2} \nabla \p_i \phi \cdot \big[ \nabla \p_i (\rho \dv \zeta ) - \rho \nabla \p_i \dv \zeta + \nabla\p_i (\uv \cdot\nabla\phi) - (\uv \cdot \nabla) \nabla \p_i \phi \\
		& \qquad\qquad\qquad\quad +  \nabla \p_i ( \dv \uvt \phi ) +  \nabla\p_i (\nabla\rhot \cdot \zeta) - \frac{3}{2} \dv \uv \nabla\p_i \phi - \nabla\p_i g_1 \big].
	\end{aligned}
\end{equation*}
Using \cref{est-g,star}, one can verify that
\begin{align*}
	\norm{I_{17}}_{L^1(\Omega)} & \lesssim \norm{\nabla^2 \phi}_{L^2(\Omega)} \big[ (\delta+\e) \big(\norm{\nabla\zeta}_{H^1(\Omega)} + \norm{\phi}_{H^2(\Omega)}\big) + \norm{\nabla\zeta}_{L^\infty(\Omega)} \norm{\nabla^2 \phi}_{L^2(\Omega)} \\
	& \qquad \qquad\qquad + \norm{\nabla^2 \zeta}_{L^4(\Omega)} \norm{\nabla \phi}_{L^4(\Omega)} + \delta \norm{\zeta_1}_{H^2(\Omega)} + \e e^{-ct} \norm{\zeta}_{H^2(\Omega)} + \e e^{-ct} \big].
\end{align*}
It follows from \cref{Lem-GN} and \cref{nu} that
\begin{equation}\label{Sobolev1}
	\begin{aligned}
		\norm{\nabla\phi }_{L^4(\Omega)} & \lesssim \sum_{k=1}^3 \norm{\nabla^2\phi }_{L^2(\Omega)}^{\frac{k}{4}} \norm{\nabla\phi }_{L^2(\Omega)}^{1-\frac{k}{4}} \lesssim 
		\norm{\nabla\phi }_{H^1(\Omega)}, \\
		\norm{\nabla^2 \zeta}_{L^4(\Omega)} & \lesssim \sum_{k=1}^3 \norm{\nabla^3 \zeta}_{L^2(\Omega)}^{\frac{k}{4}} \norm{\nabla^2\zeta}_{L^2(\Omega)}^{1-\frac{k}{4}} \lesssim \nu + \nu^{\frac{1}{4}} \norm{\nabla^3 \zeta}_{L^2(\Omega)}^{\frac{3}{4}}.
	\end{aligned}
\end{equation}
Then by \cref{Sobolev,Sobolev1,bdd-z1,key1}, one can get that
\begin{equation}\label{I-17}
	\begin{aligned}
		\norm{I_{17}}_{L^1(\Omega)} & \lesssim \norm{\p_1 \Phi, \p_1 \Psi_1}_{L^2(\R)}^2 +  \norm{\nabla\zeta}_{H^1(\Omega)}^2 + \norm{\nabla\phi}_{L^2(\Omega)}^2 \\
		& \quad + (\delta+\e+\nu) \norm{\nabla^2 \phi}_{L^2(\Omega)}^2 + \nu \norm{\nabla^3 \zeta}_{L^2(\Omega)}^2 + \e e^{-ct}.
	\end{aligned}
\end{equation}

On the other hand, similar to \cref{tool1}, one can verify that
\begin{align*}
	& \frac{1}{\rho} \nabla \p_i \phi \cdot \left[ \mu \lap \p_i \zeta + (\mu+\lambda) \nabla\dv \p_i \zeta \right] \notag \\
	& \qquad = \frac{\mut}{\rho} \nabla \p_i\phi  \cdot \nabla\dv \p_i \zeta + \dv \Big(\frac{\mu  \nabla \p_i \phi \times \text{curl} \p_i \zeta }{\rho}  \Big) + \frac{\mu \nabla \rho \cdot \big(\nabla \p_i \phi \times \text{curl} \p_i \zeta \big) }{\rho^2}.
\end{align*} 
Then by taking the derivative $ \p_i  $ on $ \frac{1}{\rho}\cdot$\cref{equ-phizeta}$ _2 $ and multiplying the resulting equation by $ \cdot \nabla \p_i \phi, $ one can get that
\begin{equation}\label{eq9}
	\begin{aligned}
		& \p_t\big( \nabla\p_i\phi  \cdot \p_i \zeta \big) + \frac{p'(\rho)}{\rho} \abso{\nabla\p_i \phi}^2 - \frac{\mut}{\rho} \nabla \p_i\phi  \cdot \nabla \dv \p_i \zeta = \dv (\cdots) + I_{18},
	\end{aligned}
\end{equation}
where $ (\cdots) = \p_i \p_t \phi \p_i \zeta + \frac{\mu}{\rho} \nabla\p_i   \phi\times \text{curl}\p_i \zeta $ and
\begin{align*}
	I_{18} & =  - \nabla \p_i \phi \cdot \Big[ \p_i (\uv \cdot \nabla\zeta)   + \p_i \Big( \frac{p'(\rho)}{\rho} \Big) \nabla\phi + \p_i \Big( \frac{p'(\rho)-p'(\rhot)}{\rho} \nabla \rhot \Big)  + \p_i (\zeta \dnab \uvt)  \\
	& \qquad\qquad\qquad + \p_i \Big( \frac{\phi}{\rho}  \left(\p_t \uvt + \uvt \dnab \uvt\right) \Big) + \frac{\p_i \rho}{\rho^2} \big( \mu \lap \zeta + (\mu+\lambda) \nabla \dv \zeta \big) \\
	&  \qquad\qquad\qquad - \p_i \Big( \frac{\gv_2 - g_1 \uvt}{\rho}\Big) \Big] +  \frac{\mu }{\rho^2} \nabla\rho \cdot \left(\nabla \p_i\phi \times \text{curl} \p_i \zeta\right)  - \dv\p_i \zeta \p_i \p_t \phi.
\end{align*} 
It is direct to prove that
\begin{align}
	\norm{I_{18}}_{L^1(\Omega)} & \lesssim \norm{\nabla^2 \phi}_{L^2(\Omega)} \big[ (\delta+\e+\nu) \norm{\nabla\zeta, \phi}_{H^1(\Omega)} + \norm{\nabla\zeta}_{L^\infty(\Omega)} \norm{\nabla\zeta}_{L^2(\Omega)} \notag \\
	& \qquad\quad + \norm{\nabla\phi}_{L^4(\Omega)}^2 + \delta\norm{\zeta_1}_{H^1(\Omega)} + \e e^{-ct} \norm{\zeta}_{H^1(\Omega)} \notag \\
	& \qquad \quad  + \norm{\nabla\phi}_{L^4(\Omega)}\norm{\nabla^2\zeta}_{L^4(\Omega)} + \e e^{-ct} \big] + \norm{\nabla\p_t\phi}_{L^2(\Omega)}\norm{\nabla^2\zeta}_{L^2(\Omega)}. \label{I-18-1}
\end{align}
It follows from \cref{equ-phizeta}$ _1 $ that 
\begin{align}
	\norm{\nabla\p_t \phi}_{L^2(\Omega)} & \lesssim \norm{\nabla\zeta}_{H^1(\Omega)} + \norm{\nabla\zeta}_{L^\infty(\Omega)} \norm{\nabla\phi}_{L^2(\Omega)} + (\delta+\e) \norm{\phi,\zeta_1}_{H^1(\Omega)} \notag \\
	& \quad + (\delta+\e+\nu) \norm{\nabla^2\phi}_{L^2(\Omega)} + \e e^{-ct} \norm{\zeta}_{H^1(\Omega)} + \e e^{-ct}. \label{I-18-2}
\end{align}
Then combining \cref{I-18-1,I-18-2} and using \cref{Sobolev,Sobolev1,nu}, one can get that
\begin{equation}\label{I-18}
	\begin{aligned}
	\norm{I_{18}}_{L^1(\Omega)} & \lesssim \norm{\p_1 \Phi, \p_1 \Psi_1}_{L^2(\R)}^2 + \norm{\nabla\zeta}_{H^1(\Omega)}^2 + \norm{\nabla\phi}_{L^2(\Omega)}^2  + \e e^{-ct} \\
	& \quad  + (\delta+\e+\nu) \norm{\nabla^2 \phi}_{L^2(\Omega)}^2 + \nu \norm{\nabla^3 \zeta}_{L^2(\Omega)}^2. 
\end{aligned}
\end{equation}
Here we have used the fact that $ \norm{\nabla\phi}_{L^4(\Omega)} \lesssim \norm{\nabla\phi}_{H^1(\Omega)} \lesssim \nu. $

Collecting \cref{I-17}, \cref{I-18} and using Lemmas \ref{Lem-est1}-\ref{Lem-est6}, one can add up the summation, $ \mut\cdot $\cref{eq8} $ + $ \cref{eq9}, with respect to $ i $ from $ 1 $ to $ 3 $, and integrate the resulting equation over $ \Omega\times (0,T) $ to obtain \cref{est7}.
%\begin{align*}
%	& \norm{\nabla^2\phi }_{L^2(\Omega)}^2 + \int_0^T \norm{\nabla^2\phi }_{L^2(\Omega)}^2 dt \lesssim \norm{ \Phi_0, \Psi_{01}}^2_{L^2(\R)} + \norm{\zeta_0}_{H^1(\Omega)}^2 + \norm{\phi_0}_{H^2(\Omega)}^2 + \e \\
%	& \qquad \quad + \int_0^T \norm{\nabla \zeta}_{L^\infty(\Omega)} \big( \norm{\nabla^2\phi }_{L^2(\Omega)}^2 + \norm{\nabla^2 \zeta}_{L^2(\Omega)}^2 + \norm{\nabla\zeta}_{L^2(\Omega)}^2 + \norm{\nabla\phi }_{L^2(\Omega)}^2 \big) dt \\
%	& \qquad \quad + \int_0^T  \norm{\nabla^2 \phi }_{L^2(\Omega)} \norm{\nabla \phi}_{L^4(\Omega)} \big( \norm{\nabla \phi }_{L^4(\Omega)} + \norm{\nabla^2 \zeta }_{L^4(\Omega)} \big) dx dt.
%\end{align*}
%
%Thus, combining \cref{Sobolev}, it holds that
%\begin{align*}
%	& \norm{\nabla^2\phi }_{L^2(\Omega)}^2 + \int_0^T \norm{\nabla^2\phi }_{L^2(\Omega)}^2 dt \lesssim \norm{ \Phi_0, \Psi_{01}}^2_{L^2(\R)} + \norm{\zeta_0}_{H^1(\Omega)}^2 + \norm{\phi_0}_{H^2(\Omega)}^2 \\
%	& \qquad \qquad  + \nu \int_0^T  \norm{\nabla \phi }_{H^1(\Omega)}^2 dt + \nu \int_0^T \norm{\nabla^2 \zeta }_{L^4(\Omega)}^2 dx dt  + \e,
%\end{align*}
%which yields \cref{est7} immediately.

\end{proof}

\begin{Lem}\label{Lem-est8}
Under the assumptions of \cref{Prop-apriori}, there exist $ \delta_0>0, \gamma_0>0 $ and $ \nu_0>0 $ such that if $ \delta \leq \delta_0 $, $ \e \leq \gamma_0 \delta $ and $ \nu \leq \nu_0, $ then
\begin{equation}\label{ineq-est8}
	\begin{aligned}
		&  \sup_{t\in(0,T)} \norm{\nabla^2 \zeta}^2_{L^2(\Omega)} + \int_0^T  \norm{\nabla^3 \zeta}_{L^2(\Omega)}^2 dt \lesssim \norm{\Phi, \Psi_{01}}^2_{L^2(\R)} + \norm{\phi_0, \zeta_0}_{H^2(\Omega)}^2 + \e.
	\end{aligned}
\end{equation}	
\end{Lem}

\begin{proof}
%	First, similar to \cref{tool2}, one can verify that
%\begin{align*}
%	\frac{1}{\rho} \nabla\dv \p_i \zeta\cdot\lap\p_i \zeta  & = \frac{\abso{\nabla\dv\p_i\zeta}^2}{\rho}  + \dv\Big[\frac{\dv\p_i\zeta}{\rho}  \big(\lap\p_i\zeta-\nabla\dv\p_i\zeta \big) \Big] \\
%	& \quad + \frac{\dv\p_i\zeta}{\rho^2} \nabla\rho \cdot \big( \lap \p_i\zeta - \nabla \dv\p_i\zeta \big).
%\end{align*}
For fixed $ i \in \{1,2,3\}, $ by taking the derivative $ \p_i  $ on $ \frac{1}{\rho}\cdot $\cref{equ-phizeta}$ _2 $ and multiplying the resulting equation by $ \cdot(- \lap \p_i \zeta), $ one has that
\begin{equation}\label{eq-15}
	\frac{1}{2} \p_t \big( \abso{\nabla\p_i  \zeta}^2 \big) + \frac{\mu}{\rho} \abso{ \lap\p_i\zeta}^2 + \frac{\mu+\lambda}{\rho} \abso{\nabla\dv \p_i \zeta}^2 = \dv (\cdots) +  I_{19},
\end{equation}
where $ (\cdots) = \nabla\p_i  \zeta \p_t \p_i  \zeta - \frac{\mu+\lambda}{\rho} \dv \p_i\zeta \big( \lap \p_i \zeta - \nabla \dv \p_i \zeta \big) $ and
\begin{align*}
I_{19} & = \lap \p_i \zeta \cdot \p_i \Big[ \uv\cdot \nabla\zeta + \frac{p'(\rho)}{\rho} \nabla\phi  + \frac{p'(\rho)-p'(\rhot)}{\rho} \nabla\rhot  + \zeta \cdot \nabla\uvt + \frac{\phi}{\rho} \left(\p_t \uvt + \uvt \dnab \uvt\right) \Big] \\
& \quad + \lap \p_i \zeta \cdot \frac{\p_i \rho}{\rho^2} \big[ \mu \lap\zeta + (\mu+\lambda) \nabla \dv \zeta \big] - \lap \p_i \zeta \cdot \p_i \Big( \frac{\gv_2 - g_1 \uvt}{\rho} \Big) \\
& \quad -\frac{\mu+\lambda}{\rho^2} \dv\p_i \zeta \nabla \rho \cdot (\lap \p_i\zeta - \nabla \dv \p_i\zeta).
\end{align*}
Here we have used the fact that \cref{tool2} is still true when $ \zeta $ is replaced by $ \p_i \zeta. $

Using \cref{star,bdd-z1,est-g,key1,Sobolev1}, one can prove that
\begin{align}
	\norm{I_{19}}_{L^1(\Omega)} & \lesssim \norm{\nabla^3 \zeta}_{L^2(\Omega)} \big[ (\delta+\e+\nu) \norm{\nabla\zeta}_{H^1(\Omega)} + \norm{\nabla\phi}_{L^4(\Omega)}^2 + \norm{\phi}_{H^2(\Omega)} \notag \\
	& \qquad + \delta\norm{\zeta_1}_{H^1(\Omega)} + \e^{-ct} \norm{\zeta}_{H^1(\Omega)} + \norm{\nabla\phi}_{L^4(\Omega)} \norm{\nabla^2 \zeta}_{L^4(\Omega)}  + \e e^{-ct} \big] \notag \\
	& \lesssim \norm{\nabla^3 \zeta}_{L^2(\Omega)} \big[ \norm{\nabla\zeta,\nabla\phi}_{H^1(\Omega)} + \norm{\p_1 \Phi, \p_1 \Psi_1}_{L^2(\R)} \notag \\
	& \qquad + \nu \norm{\nabla^3 \zeta}_{L^2(\Omega)} + \e e^{-ct} \big]. \label{I-19}
\end{align}

Note that $ \norm{\lap \p_i\zeta}_{L^2(\Omega)}^2 \geq c_0 \norm{\nabla^2 \p_i  \zeta}_{L^2(\Omega)}^2 $ for some constant $ c_0>0 $. 
%and 
%\begin{align*}
%	\abso{I_{19}} & \lesssim \abso{\nabla^2 \p_i \zeta} \big[ \abso{\nabla^2 \zeta} + \abso{\nabla^2\phi} + \abso{\nabla\zeta} + \abso{\nabla\phi} + \abso{\phi} + \abso{\zeta_1} + \e e^{-ct} \abso{\zeta} \\
%	& \qquad\qquad \quad + \abso{\nabla\zeta}^2 + \abso{\nabla\phi}^2 + \abso{\nabla \phi} \abso{\nabla^2 \zeta} + \abso{\gv_2 -\uvt g_1} + \abso{\p_i (\gv_2 - \uvt g_1)} \big].
%\end{align*}
Then using \cref{I-19} and Lemmas \ref{Lem-est1} to \ref{Lem-est7}, one can integrate \cref{eq-15} over $ \Omega \times (0,T) $ and add up the resulting equations with respect to $ i $ from $ i=1 $ to $ 3 $ to obtain 
%that
%\begin{align*}
%	& \sup_{t\in(0,T)} \norm{\nabla^2 \zeta}_{L^2(\Omega)}^2 + \int_0^T \norm{\nabla^3 \zeta}_{L^2(\Omega)}^2 dt \lesssim \norm{ \Phi_0, \Psi_{01}}^2_{L^2(\R)} + \norm{\phi_0, \zeta_0}_{H^2(\Omega)}^2 + \e \\
%	& \qquad + \int_0^T \norm{\nabla^3 \zeta}_{L^2(\Omega)} \big( \norm{\nabla\zeta }_{L^4(\Omega)}^2 +\norm{\nabla \phi}_{L^4(\Omega)}^2 + \norm{\nabla \phi}_{L^4(\Omega)} \norm{\nabla^2 \zeta}_{L^4(\Omega)} \big) dt.
%\end{align*}
%This, together with \cref{Sobolev1,nu}, can yield 
\cref{ineq-est8}.

\end{proof}

\begin{Lem}\label{Lem-est9}
	Under the assumptions of \cref{Prop-apriori}, there exist $ \delta_0>0, \gamma_0>0 $ and $ \nu_0>0 $ such that if $ \delta \leq \delta_0 $, $ \e \leq \gamma_0 \delta $ and $ \nu \leq \nu_0, $ then
	\begin{equation}\label{ineq-est9}
		\begin{aligned}
			& \sup_{t\in(0,T)} \norm{\nabla^3\phi, \nabla^3\zeta }_{L^2(\Omega)}^2 + \int_0^T \norm{\nabla^3 \phi, \nabla^4 \zeta}_{L^2(\Omega)}^2  dt \\
			& \qquad \lesssim \norm{\Phi_0, \Psi_{01}}^2_{L^2(\R)} + \norm{\phi_0, \zeta_0}_{H^3(\Omega)}^2 + \e.
		\end{aligned}
	\end{equation}	
\end{Lem}
\begin{proof}
	The proof is similar to that of Lemmas \ref{Lem-est7} and \ref{Lem-est8} and is omitted for brevity.
\end{proof}

\vspace{.3cm}

\section{Local existence}\label{Sec-local}

We now prove the local existence, \cref{Thm-local}.

\begin{proof}[Proof of \cref{Thm-local}]
	For the problem \cref{equ-phizeta}, \cref{ic-pertur}, it is standard (see \cite[Theorem 5.2]{MN1980} for instance) to prove the local existence in time of the solution $ (\phi,\zeta) $, which is periodic in the transverse directions $ \xp\in \Torus^2 $ and satisfies that $ (\phi, \zeta) \in C(0,T_0;H^3(\Omega)) $ and $ \nabla\zeta \in L^2(0,T_0; H^3(\Omega)) $  for some constant $ T_0>0. $
	
	Denote 
	\begin{align*}
		(\rho, \uv) = (\rhot,\uvt) + (\phi,\zeta) \quad \text{and} \quad \psi := \rho \uv - \rhot \uvt = \rho\zeta+\uvt \phi.
	\end{align*}
	Note that $ (\phi, \psi) \in C(0,T_0; H^3(\Omega)) $ is the classical solution of the Cauchy problem \cref{equ-phipsi}, \cref{ic-phipsi}.
	Thus, to complete the proof of \cref{Thm-local}, it remains to show that the anti-derivative,
	\begin{align*}
		(\Phi,\Psi_1)(x_1,t) := \int_{-\infty}^{x_1} \int_{\Torus^2} (\phi,\psi_1)(y_1,\xp,t) d\xp dy_1, \quad x_1\in\R, t>0,
	\end{align*}
	exits and belongs to the $ C(0,T_0; H^4(\R)) $ space.	
	
	In fact, it follows from the Minkowski inequality that the zero mode of $ (\phi,\psi_1) $,
	\begin{equation}\label{regu-od}
		\big(\phi^{\od}, \psi_1^{\od}\big)(x_1,t) = \int_{\Torus^2} (\phi, \psi_1)(x,t) d\xp,
	\end{equation}
	belongs to the $ C\left(0,T_0; H^3(\R)\right) $ space.
	The first two equations of \cref{equ-od} can be rewritten as
	\begin{equation}\label{equ-linear}
		\begin{cases}
			\p_t \phi^{\od} + \p_1 \psi_1^{\od} = g_1^{\od}, \\
			\p_t \psi_1^{\od} - \mut \p_1 \Big(\frac{\p_1 \psi_1^{\od}}{\rhot^s} \Big) = \p_1 J + g_{2,1}^{\od}, 
		\end{cases} \quad x_1 \in\R, t>0,
	\end{equation}
	where 
	\begin{align*}
		J = \mut \int_{\Torus^2} \Big( \p_1 \zeta_1 - \frac{\p_1 \psi_1}{\rhot^s}  \Big) d\xp - \int_{\Torus^2} \Big[\Big( \frac{m_1^2}{\rho} - \frac{\mt_1^2}{\rhot} \Big) + \big( p(\rho) - p(\rhot) \big) \Big] d\xp.
	\end{align*}
	Note that
	\begin{align*}
		\zeta_1 - \frac{\psi_1}{\rhot^s} = - \frac{1}{\rho\rhot} \phi \psi_1 - \frac{\rhot-\rhot^s}{\rhot\rhot^s} \psi_1 - \frac{\ut_1}{\rho} \phi.
	\end{align*}
	Then by the fact that $ \phi,\psi_1, \zeta_1 \in C(0,T_0; H^3(\Omega)) $, one has that $ J \in C(0,T_0; H^2(\R)). $
%	And by \cref{ic-phipsi}, the initial data of the zero mode $ (\phi^\od, \psi_1^\od) $ is given by
%	\begin{align*}
%		(\phi^{\od}, \psi_1^{\od})(x_1,0) = (\phi_0^{\od}, \psi_{01}^{\od})(x_1) = \int_{\Torus^2} (\phi_0, \psi_{01})(x_1,\xp) d\xp, \qquad x_1 \in \R.
%	\end{align*}
	Moreover, it follows from the assumption 2) in \cref{Thm}, \cref{Lem-F} and the Minkowski inequality that the anti-derivative variables,
	\begin{align*}
		(\Phi_0,\Psi_{01})(x_1) = \int_{-\infty}^{x_1} \big(\phi_0^{\od},\psi_{01}^{\od}\big)(y_1) dy_1,
	\end{align*}
	and
	$$  (G_1, G_2)(x_1,t) = \int_{-\infty}^{x_1} (g_1^{\od}, g_{2,1}^{\od})(y_1,t) dy_1, $$
	belong to the $ H^4(\R) $ and $  C(0,+\infty; H^3(\R)) $ spaces, respectively.
	
	Consider the Cauchy problem for a linear parabolic equation,
	\begin{equation}\label{equ-eta}
		\begin{cases}
			\p_t \hat{\Psi_1} - \frac{\mut}{\rhot^s} \p_1^2 \hat{\Psi_1} = J + G_2, \qquad x_1 \in \R, t>0, \\
			\hat{\Psi_1}(x_1,0) = \Psi_{01}(x_1), \qquad\qquad x_1\in\R.
		\end{cases}
	\end{equation}
	It follows from the classical parabolic theory that \cref{equ-eta} admits a unique classical solution $ \hat{\Psi_1} \in C(0,T_0; H^3(\R)). $ By comparing \cref{equ-linear}$ _2 $ and \cref{equ-eta}, one has that $ \p_1 \hat{\Psi_1} $ and $ \psi_1^{\od} $ solve the same parabolic equation and share the same initial data.
	Thus, the classical uniqueness theory yields that $  \p_1 \hat{\Psi_1} \equiv \psi_1^{\od}, $ which implies that 
	$$ \hat{\Psi_1}(x_1,t) \equiv \int_{-\infty}^{x_1} \psi_1^\od(y_1,t) dy_1 := \Psi_1(x_1,t) \qquad \forall x_1\in\R, t\geq 0.  $$ 
	Thus, $ \Psi_1 $ belongs to the $ C(0,T_0; H^3(\R)) $ space.
	
	Similarly, for the given $ \hat{\Psi_1} \in C(0,T_0; H^3(\Omega)), $ the Cauchy problem for a linear hyperbolic equation
	\begin{equation}\label{equ-Phi}
		\begin{cases}
			\p_t \hat{\Phi} = G_1 - \p_1 \hat{\Psi}_1, \qquad \ x_1\in\R, t>0, \\
			\hat{\Phi}(x_1,0) = \Phi_0(x_1), \qquad x_1\in\R,
		\end{cases}
	\end{equation} 
	admits a unique classical solution,
	$$ \hat{\Phi}(x_1,t) = \Phi_0(x_1) + \int_0^t (G_1- \p_1 \hat{\Psi}_1)(x_1,\tau) d\tau, $$
	which belongs to the $ C(0,T_0; H^2(\R)) $ space.
	Then by comparing \cref{equ-linear}$ _1 $ and \cref{equ-Phi}, and using the classical uniqueness theory for the linear hyperbolic equations, one can get that $ \phi^{\od} \equiv \p_1 \hat{\Phi} $, which implies that $ \Phi \equiv \hat{\Phi} $. Thus, it holds that $ \Phi \in C(0,T_0; H^2(\R)). $ 
	
	At last, note that
	\begin{align*}
		\norm{\nabla\Phi, \nabla\Psi_1}_{H^3(\R)} = \norm{\phi^\od, \psi_1^\od}_{H^3(\R)} \lesssim \norm{\phi,\psi_1}_{H^3(\Omega)}.
	\end{align*}
	Then one has that $ (\Phi,\Psi_1) \in C(0,T_0; H^4(\R)), $ which completes the proof of \cref{Thm-local}.
\end{proof}

\vspace{.2cm}

\section{Time-asymptotic stability}\label{Sec-stability}

Combining \cref{Prop-apriori} and \cref{Thm-local}, one can use a standard continuation argument to prove that the Cauchy problem \cref{equ-phizeta}, \cref{ic-pertur} admits a global-in-time solution $ (\phi,\zeta) $, which belongs to the $ \mathbb{B}(0,+\infty) $ space.
Then it remains to prove the time-asymptotic behaviors \cref{asymp,exp-pert} to complete the proof of \cref{Thm-pertur}. And the proof of the main result, \cref{Thm}, is completed at the end of this section.

\vspace{.2cm}

Since \cref{Prop-apriori} holds true for $ T=+\infty $, then \cref{est-apriori} yields that
\begin{align}
	\sup_{t>0} \norm{\phi, \zeta}_{W^{1,\infty}(\Omega)} \lesssim \sup_{t>0} \norm{\phi, \zeta}_{H^3(\Omega)} \lesssim \nu_0+\e_0. \label{nu-1}
	 \notag
\end{align}
and 
\begin{align*}
	\int_0^{+\infty} \norm{\phi, \nabla \zeta}_{H^3(\Omega)}^2 dt \lesssim \nu_0 + \e_0.
\end{align*}
%Then the asymptotic behavior \cref{asymp} holds true directly. In fact, 
%since the function, 
%$$ f(t):= \norm{\nabla (\phi,\zeta)}_{L^2(\Omega)}^2(t), $$ belongs to the $ W^{1,1}(\R_+) $ space, then $ f(t) \to 0 $ as $ t\to +\infty. $ Moreover, 
It follows from \cref{Lem-GN} that one can decompose $ (\phi, \zeta) $ as 
\begin{align*}
	(\phi, \zeta) = \sum_{k=1}^3 (\phi^{(k)}, \zeta^{(k)}),
\end{align*}
where each $ (\phi^{(k)}, \zeta^{(k)}) $ satisfies the $ k $-dimensional G-N inequalities. Thus, it holds that
\begin{equation}\label{gn-1-2}
	\begin{aligned}
	\norm{\phi^{(1)},\zeta^{(1)}}_{L^\infty(\Omega))} & \lesssim \norm{\nabla (\phi^{(1)},\zeta^{(1)})}_{L^2(\Omega)}^{\frac{1}{2}} \norm{\phi^{(1)},\zeta^{(1)}}_{L^2(\Omega)}^{\frac{1}{2}} \\
	& \lesssim \norm{\nabla (\phi,\zeta)}_{L^2(\Omega)}^{\frac{1}{2}} \norm{\phi,\zeta}_{L^2(\Omega)}^{\frac{1}{2}}, 
%	\norm{\phi^{(2)},\zeta^{(2)}}_{L^\infty(\Omega))} & \lesssim \norm{\nabla (\phi^{(1)},\zeta^{(1)})}_{L^2(\Omega)} \lesssim \norm{\nabla (\phi,\zeta)}_{L^2(\Omega)},
	\end{aligned}
\end{equation}
and for $k=2,3,$ 
\begin{align*}
	\norm{\phi^{(k)},\zeta^{(k)}}_{L^\infty(\Omega)} & \lesssim \norm{\nabla^2 (\phi^{(k)},\zeta^{(k)})}_{L^2(\Omega)}^{\frac{k}{2(6-k)}} \norm{(\phi^{(k)},\zeta^{(k)})}_{L^6(\Omega)}^{\frac{3(4-k)}{2(6-k)}}, \\
%	& \lesssim \norm{\nabla^2 (\phi, \zeta)}_{L^2(\Omega)}^{\frac{1}{2}} \norm{\phi,\zeta}_{L^6(\Omega)}^{\frac{1}{2}}, \\
	\norm{\phi^{(k)},\zeta^{(k)}}_{L^6(\Omega)} & \lesssim \norm{\nabla (\phi^{(k)},\zeta^{(k)})}_{L^2(\Omega)}^{\frac{k}{3}} \norm{ (\phi^{(k)},\zeta^{(k)})}_{L^2(\Omega)}^{\frac{3-k}{3}},
\end{align*}
which yields that
\begin{align}
	\norm{\phi^{(k)},\zeta^{(k)}}_{L^\infty(\Omega)} \lesssim \norm{(\phi,\zeta)}_{H^2(\Omega)}^{\frac{1}{2}} \norm{\nabla (\phi,\zeta)}_{L^2(\Omega)}^{\frac{1}{2}}. \label{gn-3}
\end{align}
%In fact, it follows from \cref{Lem-GN} that
%\begin{align*}
%	
%\end{align*}
Then combining \cref{gn-1-2,gn-3}, one has that
\begin{align*}
	\norm{\phi,\zeta}_{L^\infty(\R^3)}(t) = \norm{\phi,\zeta}_{L^\infty(\Omega)}(t) \lesssim \norm{\nabla (\phi,\zeta)}_{L^2(\Omega)}^{\frac{1}{2}}(t).
\end{align*}
Since $ \norm{\nabla (\phi,\zeta)}_{L^2(\Omega)}^2(t) $ belongs to the $ W^{1,1}\left((0,+\infty)\right) $ space, one can get that $$  \norm{\nabla (\phi,\zeta)}_{L^2(\Omega)}^2(t) \to 0 \qquad \text{as } \ t \to +\infty. $$ 
Similarly, since $ \norm{\nabla^2 (\phi,\zeta)}_{L^2(\Omega)}^2(t) $ belongs to the $ W^{1,1}\left((0,+\infty)\right) $ space, one can also get that $ \norm{\nabla (\phi,\zeta)}_{L^\infty(\R^3)}(t) \lesssim \norm{\nabla^2 (\phi,\zeta)}_{L^2(\Omega)}^2(t) \to 0 $ as $ t \to +\infty. $ Thus, \cref{asymp} holds true.

\vspace{.2cm}

It remains to show the exponential decay rate of $ (\phi^\md,\psi^\md) $ to complete the proof of \cref{Thm-pertur}. Note that this is the reason why we need the third-order regularity of the solution, $ (\phi,\zeta). $
Recall the notations $ \delta := \abso{\rhob_+-\rhob_-} $ and $ \e := \norm{v_0,\wv_0}_{H^6(\Torus^3)} $ and we denote 
$$  \nu :=  \sup\limits_{t>0} \norm{\phi, \zeta}_{H^3(\Omega)}.  $$

It follows from the Sobolev inequality and \cref{Lem-od-md} and that
\begin{equation}\label{nu-od}
	\sup_{ t>0} \norm{\phi^\od, \zeta^\od}_{W^{2,+\infty}(\R)} \lesssim \sup_{ t>0} \norm{\phi^\od, \zeta^\od}_{H^3(\R)} \lesssim \sup_{ t>0} \norm{\phi, \zeta}_{H^3(\Omega)} = \nu.
\end{equation}
And it holds that
\begin{equation}\label{nu-md-2}
	\sup_{ t>0} \norm{\phi^\md, \zeta^\md}_{H^3(\Omega)} \lesssim \sup_{ t>0} \norm{\phi, \zeta}_{H^3(\Omega)} + \sup_{ t>0} \norm{\phi^\od, \zeta^\od}_{H^3(\R)} \lesssim \nu.
\end{equation}
Then it follows from \cref{Lem-GN} that
\begin{equation}\label{nu-md}
	\sup_{ t>0} \norm{\phi^\md, \zeta^\md}_{W^{1,+\infty}(\Omega)} 
	\lesssim \sup_{ t>0} \norm{\phi^\md, \zeta^\md}_{H^3(\Omega)}
	\lesssim \nu.
\end{equation}

Note that the zero mode of any function is independent of the transverse variables, $ \xp=(x_2,x_3). $ Thus, subtracting \cref{equ-phipsi} by \cref{equ-od} yields that
\begin{equation}\label{equ-md}
	\begin{cases}
		\p_t \phi^\md + \dv \psi^\md = g_1^\md, \\
		\p_t \psi^\md + \sum\limits_{i=1}^3 \p_i \big( \ut_i^s \psi^\md + \mvt^s \zeta_i^\md + \mathbf{r}_{1,i} \big) + \nabla \big(p'(\rhot^s) \phi^\md + r_2 \big) \\
		\qquad\qquad\qquad\quad = \mu \lap \zeta^\md + (\mu+\lambda) \nabla \dv \zeta^\md + \gv_2^\md,
	\end{cases}
\end{equation}
where 
\begin{equation}\label{r-1-2}
	\begin{aligned}
	\rv_{1,i} & =  \big( u_i \mv - \ut_i \mvt \big)^\md - \ut_i^s \psi^\md - \mvt^s \zeta_i^\md \qquad \text{for } \ i = 1, 2, 3, \\
	r_2 & = \big(p(\rho)-p(\rhot)\big)^\md - p'(\rhot^s) \phi^\md.
	\end{aligned}
\end{equation}
By \cref{poincare}, it holds that
\begin{equation}\label{Poincare}
	\norm{ \phi^\md}_{L^2(\Omega)} \lesssim \norm{\nabla \phi^\md}_{L^2(\Omega)} \quad \text{and} \quad \norm{ \psi^\md}_{L^2(\Omega)} \lesssim \norm{\nabla \psi^\md}_{L^2(\Omega)}.
\end{equation} 
And using the fact that $ \psi = (\rhot + \phi)(\uvt+\zeta) - \rhot\uvt = \rhot \zeta + \uvt \phi + \phi \zeta $, one has that
\begin{align}
	\psi^\md = \rhot^s \zeta^\md + \uvt^s \phi^\md +  \rv_3, \label{psi-zeta}
\end{align}
where 
\begin{equation}\label{r-3}
	\rv_3 = [ (\rhot-\rhot^s) \zeta + (\uvt-\uvt^s) \phi ]^\md + \phi^\od \zeta^\md + \zeta^\od \phi^\md + (\phi^\md \zeta^\md)^\md.
\end{equation}
It follows from \cref{r-3,cauchy-ineq} that
\begin{align*}
	\norm{\rv_3}_{H^1(\Omega)} & \lesssim \norm{\rhot-\rhot^s}_{W^{1,+\infty}(\Omega)} \norm{\zeta}_{H^1(\Omega)} + \norm{\uvt-\uvt^s}_{W^{1,+\infty}(\Omega)} \norm{\phi}_{H^1(\Omega)} \\
	& \quad + \norm{\phi^\od}_{W^{1,\infty}(\R)} \norm{\zeta^\md}_{H^1(\Omega)} + \norm{\zeta^\od}_{W^{1,\infty}(\R)} \norm{\phi^\md}_{H^1(\Omega)} \\
	& \quad + \norm{\phi^\md}_{W^{1,\infty}(\Omega)} \norm{\zeta^\md}_{H^1(\Omega)}.
\end{align*}
Then using \cref{ansatz-shock,ansatz-shock-1,nu-od,nu-md}, one can verify that
\begin{equation}\label{est-r-3-1}
	\norm{\rv_3}_{H^1(\Omega)} \lesssim \e \nu e^{-ct} + \nu \norm{\zeta^\md,\phi^\md}_{H^1(\Omega)}.
\end{equation}
And it is similar to prove that
\begin{equation}\label{est-r-3-0}
	\norm{\rv_3}_{L^2(\Omega)} \lesssim \e \nu e^{-ct} + \nu \norm{\zeta^\md,\phi^\md}_{L^2(\Omega)}.
\end{equation}

\vspace{.2cm}

\begin{Lem}\label{Lem-zeta-psi}
	Under the assumptions of \cref{Thm}, there exist $ \delta_0>0, \gamma_0>0 $ and $ \nu_0>0 $ such that if $ \delta \leq \delta_0 $, $ \e \leq \gamma_0 \delta $ and $ \nu \leq \nu_0, $ then
	\begin{align}
		\norm{\zeta^\md}_{H^1(\Omega)} & \lesssim \norm{\nabla \psi^\md}_{L^2(\Omega)} + (\delta +\nu) \norm{\nabla \phi^\md}_{L^2(\Omega)} + \e \nu e^{-ct}, \label{zeta-est} \\
		\norm{\psi^\md}_{L^2(\Omega)} & \lesssim \norm{\zeta^\md}_{L^2(\Omega)} + (\delta +\nu) \norm{\nabla \phi^\md}_{L^2(\Omega)} + \e \nu e^{-ct}, \label{psi-est} \\
		\norm{\nabla \psi^\md}_{L^2(\Omega)} & \lesssim \norm{\nabla \zeta^\md}_{L^2(\Omega)} + (\delta +\nu)   \norm{\nabla\phi^\md}_{L^2(\Omega)} + \e \nu e^{-ct}. \label{psi-nab-est}
	\end{align}
\end{Lem}

\begin{proof}

It follows from \cref{small-us,psi-zeta} that
\begin{align*}
	\norm{\zeta^\md}_{H^1(\Omega)} & \lesssim \norm{\psi^\md}_{H^1(\Omega)} + \delta \norm{\phi^\md}_{H^1(\Omega)} + \norm{\rv_3}_{H^1(\Omega)},
\end{align*}
which, together with \cref{est-r-3-1}, yields that
\begin{align*}
	\norm{\zeta^\md}_{H^1(\Omega)} \lesssim \norm{\psi^\md}_{H^1(\Omega)} + (\delta+\nu) \norm{\phi^\md}_{H^1(\Omega)} + \e \nu e^{-ct}.
\end{align*}
%Then using \cref{nu-od,nu-md,nu}, one has that
%\begin{align*}
%	\norm{\zeta^\md}_{H^1(\Omega)} & \lesssim 
%\end{align*}
Then \cref{zeta-est} holds true by using \cref{Poincare}. 

On the other hand, by \cref{est-r-3-0,est-r-3-1,psi-zeta,Poincare}, it is direct to prove \cref{psi-est} and 
\begin{align*}
	\norm{\nabla \psi^\md}_{L^2(\Omega)} & \lesssim \norm{\nabla \zeta^\md}_{L^2(\Omega)} + (\delta+\nu) \big(\norm{\zeta^\md}_{L^2(\Omega)} +  \norm{\phi^\md}_{H^1(\Omega)}\big) + \e \nu e^{-ct}.
\end{align*}
This, together with \cref{zeta-est}, yields \cref{psi-nab-est}.

\end{proof}

\begin{Lem}\label{Lem-r}
	Under the assumptions of \cref{Thm}, there exist $ \delta_0>0, \gamma_0>0 $ and $ \nu_0>0 $ such that if $ \delta \leq \delta_0 $, $ \e \leq \gamma_0 \delta $ and $ \nu \leq \nu_0, $ then
	\begin{align}
		\norm{\rv_{1,i}}_{H^1(\Omega)} & \lesssim \nu \norm{\nabla \psi^\md}_{L^2(\Omega)} + (\delta +\nu) \norm{\nabla \phi^\md}_{L^2(\Omega)} + \e \nu e^{-ct} \qquad \text{for} \ i=1,2,3, \label{r-1-est} \\
%		\norm{\nabla^2 \rv_{1,i}}_{L^2(\Omega)} & \lesssim \nu \norm{\nabla^2 \psi^\md}_{L^2(\Omega)} + (\delta +\nu)  \norm{\nabla\phi^\md}_{H^1(\Omega)} + \e \nu e^{-ct}  + \nu \norm{\nabla^3 \zeta^\md}_{L^2(\Omega)}, \\
		\norm{r_2}_{H^1(\Omega)} & \lesssim \nu \norm{\nabla \phi^\md}_{L^2(\Omega)} + \e \nu e^{-ct}. \label{r-2-est}
%		\norm{\nabla^2 r_2^\md}_{L^2(\Omega)} & \lesssim \nu \norm{\nabla \phi^\md}_{H^1(\Omega)} + \e \nu e^{-ct}.
	\end{align}
\end{Lem}
\begin{proof}
	Note that it follows from \cref{r-1-2} that
	\begin{align*}
		\rv_{1,i} & = \big( u_i \mv - \ut_i \mvt - \ut^s_i \psi - \mvt^s \zeta_i \big)^\md \\
		& = \big[ (\ut_i - \ut_i^s) \psi + (\mvt-\mvt^s) \zeta_i + \zeta_i \psi \big]^\md \\
		& =  \big[(\ut_i - \ut_i^s) \psi + (\mvt-\mvt^s) \zeta_i \big]^\md + \zeta_i^\od \psi^\md + \zeta_i^\md \psi^\od + (\zeta_i^\md \psi^\md)^\md.
	\end{align*}
	Thus, using \cref{cauchy-ineq,nu-md,nu-od}, it holds that
	\begin{align*}
		\norm{\rv_{1,i}}_{H^1(\Omega)} & \lesssim \e e^{-ct} \norm{\psi, \zeta}_{H^1(\Omega)} + \norm{\zeta^\od, \psi^\od, \zeta^\md}_{W^{1,\infty}(\Omega)} \norm{\psi^\md,\zeta^\md}_{H^1(\Omega)} \\
		& \lesssim \e\nu e^{-ct} + \nu \norm{\psi^\md,\zeta^\md}_{H^1(\Omega)},
	\end{align*}
	which, together with \cref{zeta-est,Poincare}, yields \cref{r-1-est}.
	
	Similarly, the term $ r_2 $ in \cref{r-1-2} satisfies that
	\begin{align*}
		r_2 & = \big[ p(\rho) - p(\rhot) - p'(\rhot^s) \phi \big]^\md \\
		& = \underbrace{\Big[ \int_0^1 \big( p'(\rhot + \theta \phi) - p'(\rhot^s + \theta \phi) \big) d\theta \phi \Big]^\md}_{I_1} + \Big[ \int_0^1 \big( p'(\rhot^s + \theta\phi) - p'(\rhot^s) \big) d\theta \phi \Big]^\md \\
		& = I_1 + \underbrace{\Big[ \int_0^1 \big( p'(\rhot^s + \theta\phi) - p'(\rhot^s+ \theta \phi^\od) \big) d\theta \phi \Big]^\md}_{I_2} \\
		& \qquad \ + \underbrace{ \int_0^1 \big( p'(\rhot^s + \theta\phi^\od) - p'(\rhot^s) \big) d\theta \phi^\md}_{I_3}.
	\end{align*}
	It follows from \cref{cauchy-ineq} that
	\begin{align*}
		\norm{I_1}_{H^1(\Omega)} & \lesssim \norm{\rhot - \rhot^s}_{W^{1,\infty}(\Omega)} \norm{\phi}_{H^1(\Omega)} \lesssim \e \nu e^{-ct}.
	\end{align*}
	Note that 
	\begin{align*}
		I_2 = \Big[ \int_0^1 \int_0^1 p''(\rhot^s + \theta\phi^\od + r\theta \phi^\md) dr \theta d\theta \phi^\md \phi \Big]^\md.
	\end{align*}
	Thus, using \cref{small-phizeta,cauchy-ineq}, one has that
	\begin{align*}
		\norm{I_2}_{H^1(\Omega)} &  \lesssim \norm{\phi}_{W^{1,\infty}(\Omega)} \norm{\phi^\md}_{H^1(\Omega)} \lesssim \nu \norm{\phi^\md}_{H^1(\Omega)}.
	\end{align*}
	Similarly, it follows from \cref{cauchy-ineq,nu-od} that
	\begin{align*}
		\norm{I_3}_{H^1(\Omega)} \lesssim \norm{ \phi^\od}_{W^{1,\infty}(\R)} \norm{\phi^\md}_{H^1(\Omega)} \lesssim \nu \norm{\phi^\md}_{H^1(\Omega)}.
	\end{align*}
	Then with the estimate of each $ I_i $, \cref{r-2-est} can follow from \cref{Poincare} directly.

\end{proof}

Hereafter, we let $ a_i $ for $i=1,2,3,\cdots $ denote some positive generic constants, which are independent of $ \delta, \e, \nu, \gamma_0 $ and $t$.

\begin{Lem}\label{Lem-exp-1}
	Under the assumptions of \cref{Thm}, there exist $ \delta_0>0, \gamma_0>0 $ and $ \nu_0>0 $ such that if $ \delta \leq \delta_0 $, $ \e \leq \gamma_0 \delta $ and $ \nu \leq \nu_0, $ then
	\begin{equation}\label{est-md-1}
		\begin{aligned}
		& \frac{d}{dt} \Big[ \int_\Omega  \big( \abso{p'(\rhot^s)} \abso{\phi^\md}^2 + \abso{\psi^\md}^2 \big) dx \Big] + a_1 \norm{\nabla \psi^\md}_{L^2(\Omega)}^2 \\
		& \qquad \lesssim (\delta+\nu) \norm{\nabla\phi^\md}^2_{L^2(\Omega)} + \e \nu e^{-ct}.
	\end{aligned}
	\end{equation}
\end{Lem}
\begin{proof}
Multiplying $ p'(\rhot^s) \phi^\md $ on \cref{equ-md}$ _1 $ and $ \psi^\md $ on \cref{equ-md}$ _2 $ and adding the resulting equations up, one has that
\begin{equation}\label{equ-exp-1}
	\p_t \Big[\frac{1}{2} \big( \abso{p'(\rhot^s)} \abso{\phi^\md}^2 + \abso{\psi^\md}^2 \big) \Big] + I_4 = \dv (\cdots) + \sum_{i=5}^7 I_i,
\end{equation}
where $ \dv (\cdots) $ denotes some terms which vanishes after integration over $ \Omega $, and
\begin{align*}
	I_4 & = \sum_{i=1}^3 \mu \p_i \zeta^\md \cdot \p_i \psi^\md  + (\mu+\lambda) \dv \zeta^\md \dv \psi^\md, \\
	I_5 & = \frac{1}{2} p''(\rhot^s) \p_t \rhot^s \abso{\phi^\md}^2 + \sum_{i=1}^3 \big(\ut_i^s \psi^\md + \mvt^s \zeta_i^\md \big)\cdot \p_i \psi^\md, \\
	I_6 & = \sum_{i=1}^3 \rv_{1,i} \cdot \p_i\psi^\md + r_2 \dv \psi^\md, \\
	I_7 & = p'(\rhot^s) \phi^\md g_1^\md + \psi^\md \cdot \gv_2^\md.
\end{align*}
It follows from \cref{psi-zeta} that
\begin{align*}
	I_4 & = \rhot^s \big(\mu \abso{\nabla\zeta^\md}^2 + (\mu+\lambda) \abso{\dv \zeta^\md}^2 \big) \\
	& \quad + O(1) \delta \abso{\nabla\zeta^\md}  \big( \abso{\zeta^\md} + \abso{\phi^\md} + \abso{\nabla\phi^\md} \big) + O(1) \abso{\nabla\zeta^\md} \abso{\nabla\rv_3}.
\end{align*}
This, together with \cref{Poincare,est-r-3-1,zeta-est,psi-nab-est}, 
%yields that 
%\begin{align*}
%	\int_\Omega I_4 dx & \gtrsim \norm{\nabla \zeta^\md}_{L^2(\Omega)}^2 - (\delta+\nu) \norm{\phi^\md,\zeta^\md}_{H^1(\Omega)}^2 - \e \nu e^{-ct}.
%\end{align*}
yields that
\begin{equation}\label{I-1-1}
	\int_\Omega I_4 dx \gtrsim \norm{\nabla \psi^\md}_{L^2(\Omega)}^2 - (\delta+\nu) \norm{\nabla \phi^\md}_{L^2(\Omega)}^2 - \e \nu e^{-ct}.
\end{equation}
It follows from \cref{small-us,Poincare} and Lemmas \ref{Lem-F}, \ref{Lem-zeta-psi} and \ref{Lem-r} that
\begin{equation}\label{I-2-1}
	\begin{aligned}
		\norm{I_5}_{L^1(\Omega)} 
		& \lesssim \delta \norm{\phi^\md}_{L^2(\Omega)} + \delta \norm{\nabla\psi^\md}_{L^2(\Omega)} \big(\norm{\psi^\md}_{L^2(\Omega)} + \norm{\zeta^\md}_{L^2(\Omega)} \big) \\
		%	& \lesssim \delta \norm{\nabla\psi^\md}_{L^2(\Omega)} \big( \norm{\nabla \psi^\md}_{L^2(\Omega)} + (\delta+\nu) \norm{\nabla \phi^\md}_{L^2(\Omega)} + \e \nu e^{-ct} \big)  \\
		& \lesssim \delta \norm{\nabla\psi^\md}_{L^2(\Omega)}^2 + (\delta+\nu) \norm{\nabla \phi^\md}_{L^2(\Omega)}^2 + \e \nu e^{-ct}, \\
		\norm{I_6}_{L^1(\Omega)} & \lesssim \norm{\nabla\psi^\md}_{L^2(\Omega)} \big[ \nu \norm{\nabla\psi^\md}_{L^2(\Omega)} + (\delta+\nu) \norm{\nabla \phi^\md}_{L^2(\Omega)} + \e \nu e^{-ct} \big] \\
		& \lesssim (\delta+\nu) \big(\norm{\nabla\psi^\md}_{L^2(\Omega)}^2 +  \norm{\nabla \phi^\md}_{L^2(\Omega)}^2\big) + \e \nu e^{-ct}, \\
		\norm{I_7}_{L^1(\Omega)} & \lesssim \e e^{-ct} \big( \norm{\phi^\md}_{L^2(\Omega)} + \norm{\psi^\md}_{L^2(\Omega)} \big) \lesssim  \e \nu e^{-ct}.
	\end{aligned}
\end{equation}
Combining \cref{I-1-1,I-2-1}, one can integrate \cref{equ-exp-1} on $ \Omega $ to obtain \cref{est-md-1}.

\end{proof}

\begin{Lem}\label{Lem-exp-2}
	Under the assumptions of \cref{Thm}, there exist $ \delta_0>0, \gamma_0>0 $ and $ \nu_0>0 $ such that if $ \delta \leq \delta_0 $, $ \e \leq \gamma_0 \delta $ and $ \nu \leq \nu_0, $ then
	\begin{equation}\label{est-md-2}
		\begin{aligned}
		& \frac{d}{dt} \Big[\int_\Omega \Big(\frac{\mut}{2\rhot^s} \abso{\nabla \phi^\md}^2 + \psi^\md\cdot \nabla\phi^\md \Big) dx \Big] + a_2 \norm{\nabla \phi^\md}_{L^2(\Omega)}^2 \\
		& \qquad \lesssim \norm{\nabla \psi^\md}_{L^2(\Omega)}^2 + \nu \norm{\nabla^2 \zeta^\md}_{L^2(\Omega)}^2 + \e \nu e^{-ct}.
	\end{aligned}
	\end{equation}
\end{Lem}

\begin{proof}

Multiplying $ \frac{\mut}{\rhot^s} \nabla\phi^\md $ on the gradient $ ``\nabla" $ of  \cref{equ-md}$ _1 $ and $ \nabla\phi^\md $ on \cref{equ-md}$ _2 $, respectively, then adding the resulting equations together yields that
\begin{align}
	& \p_t \Big(\frac{\mut}{2\rhot^s} \abso{\nabla \phi^\md}^2 + \psi^\md\cdot \nabla\phi^\md \Big)  + p'(\rhot^s) \abso{\nabla\phi^\md}^2 =  \dv(\cdots) + \sum_{i=8}^{11} I_i, \label{equ-exp-2}
\end{align}
where $ (\cdots) = \p_t \phi^\md\psi^\md + \mu \nabla\phi^\md \times \text{curl} \zeta^\md $ and
\begin{align*}
	I_8 & := -\frac{\mut}{\rhot^s} \nabla\dv \psi^\md \cdot \nabla \phi^\md + \mut \nabla\dv \zeta^\md \cdot\nabla\phi^\md, \\
	I_9 & := - \frac{\mut\p_t \rhot^s}{2\abso{\rhot^s}^2} \abso{\nabla \phi^\md}^2 + \abso{\dv \psi^\md}^2 - p''(\rhot^s)\phi^\md  \nabla\rhot^s \cdot \nabla\phi^\md \\
	& \quad \ - \sum_{i=1}^{3} \nabla\phi^\md \cdot \p_i \big( \ut_i^s \psi^\md + \mvt^s \zeta_i^\md \big), \\
	I_{10} & := - \sum_{i=1}^{3} \nabla\phi^\md \cdot \p_i \rv_{1,i} - \nabla \phi^\md \cdot \nabla r_2,  \\
	I_{11} & := \frac{\mut}{\rhot^s}\nabla g_1^\md \cdot \nabla\phi^\md + \nabla\phi^\md \cdot \gv_2^\md - \dv \psi^\md g_1^\md.
\end{align*}
Here we have used the fact that
\begin{align*}
	\nabla\phi^\md \cdot (\lap \zeta^\md - \nabla\dv\zeta^\md) = \dv \big( \nabla\phi^\md \times \text{curl} \zeta^\md \big).
\end{align*}
By using \cref{psi-zeta,r-3}, one can get that
\begin{align*}
	I_8 = I_{8,1} + I_{8,2} + I_{8,3},
\end{align*}
where 
\begin{align*}
	I_{8,1} & = -\frac{\mut}{\rhot^s} \sum_{i=1}^3 \big(\ut_i^s+\zeta_i^\od\big)  \nabla\phi^\md \cdot \nabla \p_i \phi^\md, \\
	I_{8,2} & = -\frac{\mut}{\rhot^s} \nabla\phi^\md \cdot \big(\nabla^2 \phi^\md \zeta^\md \big)^\md,
\end{align*}
and $ I_{8,3} $ denotes the summation of the remaining terms, satisfying that 
\begin{align}
	\norm{I_{8,3}}_{L^1(\Omega)} & \lesssim \norm{\nabla\phi^\md}_{L^2(\Omega)} \Big( \delta \norm{\zeta^\md,\phi^\md}_{H^1(\Omega)} + \e e^{-ct} \norm{\phi,\zeta}_{H^2(\Omega)} \notag \\
	& \quad + \norm{\phi^\od}_{W^{2,\infty}(\R)} \norm{\zeta^\md}_{H^2(\Omega)} + \norm{\p_1\zeta^\od}_{W^{1,\infty}(\R)} \norm{\phi^\md}_{H^1(\Omega)} \notag \\
	& \quad + \norm{\phi^\md}_{L^\infty(\Omega)} \norm{\nabla^2 \zeta^\md}_{L^2(\Omega)} + \norm{\nabla\phi^\md}_{L^2(\Omega)} \norm{\nabla \zeta^\md}_{L^\infty(\Omega)} \Big) \notag \\
	& \lesssim (\delta+\nu) \big(\norm{\nabla\phi^\md}_{L^2(\Omega)}^2 + \norm{\nabla\psi^\md}_{L^2(\Omega)}^2 \big) + \nu \norm{\nabla^2 \zeta^\md}_{L^2(\Omega)}^2 + \e \nu e^{-ct}. \label{I-8-3}
\end{align}
Here we have used \cref{nu-od,nu-md,Poincare,zeta-est}.

Note that
\begin{align*}
	I_{8,1} = \dv \Big( - \frac{\mut}{2\rhot^s} \big( \uvt^s + \zeta^\od \big) \abso{\nabla\phi^\md}^2 \Big) +  \dv \Big( \frac{\mut}{2\rhot^s} \big( \uvt^s + \zeta^\od \big) \Big) \abso{\nabla\phi^\md}^2,
\end{align*}
which, together with \cref{small-us,nu-od}, yields that 
\begin{equation}\label{I-8-1}
	\abso{\int_{\Omega} I_{8,1} dx } \lesssim (\delta+\nu) \norm{\nabla\phi^\md}_{L^2(\Omega)}^2.
\end{equation}
It follows from \cref{Lem-GN}, \cref{zeta-est,nu-md-2} that
\begin{align}
	\norm{I_{8,2}}_{L^1(\Omega)} & \lesssim  \norm{\nabla\phi^\md}_{L^2(\Omega)}  \norm{\zeta^\md}_{L^\infty(\Omega)} \sup_{ t>0} \norm{\nabla^2 \phi^\md}_{L^2(\Omega)} \notag \\
	& \lesssim \nu \norm{\nabla\phi^\md}_{L^2(\Omega)}  \norm{\zeta^\md}_{H^2(\Omega)} \notag \\
	& \lesssim \nu \norm{\nabla\phi^\md}_{L^2(\Omega)}^2 + \nu \norm{\nabla\psi^\md}_{L^2(\Omega)}^2 + \nu \norm{\nabla^2 \zeta^\md}_{L^2(\Omega)}^2 + \e \nu e^{-ct}. \label{I-8-2}
\end{align}
Collecting \cref{I-8-1,I-8-2,I-8-3} yields that
\begin{equation}\label{I-8-exp}
	\begin{aligned}
	\abso{\int_{\Omega} I_8 dx } & \lesssim (\delta+\nu) \norm{\nabla \phi^\md}^2_{L^2(\Omega)} +  \norm{\nabla \psi^\md}^2_{L^2(\Omega)} + \nu \norm{\nabla^2 \zeta^\md}^2_{L^2(\Omega)} + \e \nu e^{-ct}.
\end{aligned}
\end{equation}
Using \cref{Poincare} and Lemmas \ref{Lem-F}, \ref{Lem-zeta-psi} and \ref{Lem-r}, one can verify directly that
\begin{align}
	\sum_{i=9}^{11} \norm{I_i}_{L^1(\Omega)} & \lesssim (\delta+\nu+ \e) \norm{\nabla\phi^\md}_{L^2(\Omega)}^2 + \norm{\nabla\psi^\md}_{L^2(\Omega)}^2 + \e \nu e^{-ct}. \label{I-2-2}
\end{align}
Thus, combining \cref{I-8-exp,I-2-2}, one can integrate \cref{equ-exp-2} over $ \Omega $ to obtain \cref{est-md-2}.
\end{proof}

\vspace{.2cm}

\begin{Lem}\label{Lem-exp-3}
Under the assumptions of \cref{Thm}, there exist $ \delta_0>0, \gamma_0>0 $ and $ \nu_0>0 $ such that if $ \delta \leq \delta_0 $, $ \e \leq \gamma_0 \delta $ and $ \nu \leq \nu_0, $ then
\begin{equation}\label{est-md-3}
	\frac{d}{dt} \Big( \int_\Omega \rhot^s \abso{\nabla\zeta^\md}^2 dx \Big) + a_3 \norm{\nabla^2 \zeta^\md}_{L^2(\Omega)}^2 \lesssim \norm{\nabla\phi^\md, \nabla\psi^\md}_{L^2(\Omega)}^2 + \e e^{-ct}.
\end{equation}
\end{Lem}

\begin{proof}

Multiply $-\lap \zeta^\md$ on \cref{equ-md}$_2$ yields that
\begin{equation}\label{equ-exp-3}
	\begin{aligned}
	& - \p_t \psi^\md \cdot \lap \zeta^\md + \mu \abso{\lap \zeta^\md}^2 + (\mu+\lambda) \abso{\nabla\dv\zeta^\md}^2 \\
	& \quad = \dv (\cdots) + \lap \zeta^\md \cdot \Big[ \sum\limits_{i=1}^3 \p_i \big( \ut_i^s \psi^\md + \mvt^s \zeta_i^\md + \mathbf{r}_{1,i} \big) + \nabla \big(p'(\rhot^s) \phi^\md + r_2 \big) + \gv_2^\md \Big].
\end{aligned}
\end{equation}
where $ (\cdots) = \dv\zeta^\md \big(\lap \zeta^\md - \nabla \dv\zeta^\md \big). $
Thus, by using \cref{Poincare} and Lemmas \ref{Lem-zeta-psi} and \ref{Lem-r}, one can integrate \cref{equ-exp-3} over $\Omega$ to obtain that
\begin{equation}\label{ineq-3}
	\begin{aligned}
		& - \int_\Omega \p_t \psi^\md \cdot \lap \zeta^\md dx  + a_3 \norm{\nabla^2 \zeta^\md}_{L^2(\Omega)}^2 \lesssim \norm{\nabla \psi^\md,\nabla \phi^\md}_{L^2(\Omega)}^2 + \e e^{-ct}
	\end{aligned}
\end{equation}
for some constant $ a_3>0. $
Using \cref{psi-zeta,r-3}, one has that
\begin{equation}\label{equ-2}
	\begin{aligned}
	& - \p_t \psi^\md \cdot \lap \zeta^\md \\
	& \quad = \p_t \Big( \frac{1}{2} \rhot^s \abso{\nabla \zeta^\md}^2 \Big) - \dv \big( \rhot^s \nabla\zeta^\md \p_t \zeta^\md \big) + \nabla \rhot^s \cdot \nabla\zeta^\md \p_t \zeta^\md - \frac{1}{2} \p_t \rhot^s \abso{\nabla\zeta^\md}^2  \\
	& \qquad + O(1) \abso{\lap \zeta^\md} \big[ (\delta+\nu) \big(\abso{\zeta^\md} + \abso{\phi^\md} + \abso{\p_t \phi^\md}  + \abso{\p_t \zeta^\md} \big) \\
	& \qquad\qquad\qquad\qquad\quad + \e e^{-ct} \big( \abso{\zeta} + \abso{\phi} + \abso{\p_t\zeta} + \abso{\p_t \phi} \big) \\
	& \qquad\qquad\qquad\qquad\quad + \abso{\p_t \phi^\od} \abso{\zeta^\md} + \abso{\p_t \zeta^\od} \abso{\phi^\md} \big]. 
\end{aligned}
\end{equation}
By \cref{equ-phizeta} and \cref{Lem-od-md}, one can prove that
\begin{align}
	\norm{\p_t \phi^\od}_{L^\infty(\R)} & \lesssim \norm{\p_t \phi^\od}_{H^1(\R)} \lesssim \norm{\p_t \phi}_{H^1(\Omega)}  \lesssim \norm{\zeta}_{H^2(\Omega)} + \norm{\phi}_{H^2(\Omega)} + \norm{g_1}_{H^1(\Omega)} \notag \\
	& \lesssim \nu + \e,  \label{small-phi-od}
\end{align}
and
\begin{align}
	\norm{\p_t \zeta^\od}_{L^\infty(\R)} & \lesssim \norm{\p_t \zeta^\od}_{H^1(\R)} \lesssim \norm{\p_t \zeta}_{H^1(\Omega)}  \lesssim \norm{\zeta}_{H^3(\Omega)} + \norm{\phi}_{H^2(\Omega)} + \norm{\gv_2, g_1}_{H^1(\Omega)} \notag \\
	& \lesssim \nu + \e. \label{small-zeta-od}
\end{align}
Moreover, it follows from \cref{equ-md} and Lemmas \ref{Lem-zeta-psi} and \ref{Lem-r} that
\begin{equation}\label{pt-phi-psi}
	\begin{aligned}
	\norm{\p_t \phi^\md}_{L^2(\Omega)} & \lesssim \norm{\nabla\psi^\md}_{L^2(\Omega)} + \e e^{-ct}, \\
	\norm{\p_t \psi^\md}_{L^2(\Omega)} & \lesssim \norm{\nabla\psi^\md, \nabla \phi^\md}_{L^2(\Omega)} + \e e^{-ct} + \norm{\nabla^2 \zeta^\md}_{L^2(\Omega)}.
	\end{aligned}
\end{equation}
Using \cref{psi-zeta,r-3}, one has that
\begin{align*}
	\norm{\p_t \zeta^\md}_{L^2(\Omega)} & \lesssim \norm{\p_t \psi^\md}_{L^2(\Omega)} + \delta\norm{\psi^\md}_{L^2(\Omega)}  + \delta \norm{\phi^\md, \p_t \phi^\md}_{L^2(\Omega)}  \\
	& \quad + \e e^{-ct} \norm{\zeta,\phi,\p_t \zeta, \p_t \phi}_{L^2(\Omega)} + \nu \norm{ \p_t \zeta^\md, \p_t \phi^\md}_{L^2(\Omega)} \\
	& \quad + \norm{\p_t \phi^\od,\p_t \zeta^\od}_{L^\infty(\R)} \norm{\zeta^\md, \phi^\md}_{L^2(\Omega)},
\end{align*}
which, together with \cref{small-phi-od,small-zeta-od,pt-phi-psi}, yields that
\begin{equation}\label{ineq-2}
	\norm{\p_t \zeta^\md}_{L^2(\Omega)} \lesssim \norm{\nabla\psi^\md, \nabla\phi^\md}_{L^2(\Omega)} + \e e^{-ct} + \norm{\nabla^2 \zeta^\md}_{L^2(\Omega)}.
\end{equation}
Using \cref{small-phi-od,small-zeta-od,pt-phi-psi,zeta-est,Poincare}, one can integrate \cref{equ-2} to get that
\begin{align*}
	& \int_{\Omega} \lap \zeta^\md \cdot \p_t \psi^\md dx + \frac{d}{dt} \int_\Omega \Big( \frac{1}{2} \rhot^s \abso{\nabla\zeta^\md}^2 \Big) dx \\ 
	& \qquad \lesssim \norm{\nabla\phi^\md,\nabla\psi^\md}_{L^2(\Omega)}^2 +  (\delta +\nu) \norm{\nabla^2 \zeta^\md}_{L^2(\Omega)}^2 + \e e^{-ct}.
\end{align*}
This, together with \cref{ineq-3}, can finish the proof.

\end{proof}

Let $ A_1 $ and $ A_2 $ be two positive constants to be determined. 
Multiplying the constants $ A_1 $ and $ A_2 $ on \cref{est-md-1} and $ \cref{est-md-2} $, respectively, and then summing the resulting two inequalities and \cref{est-md-3} together, one can obtain that
\begin{align*}
	E'(t) + H(t) & \leq C \ \Big[ \big((\delta+\nu) A_1 + 1\big) \norm{\nabla\phi^\md}_{L^2(\Omega)}^2 + (A_2+1) \norm{\nabla\psi^\md}_{L^2(\Omega)}^2  \\
	& \qquad \quad + \nu A_2 \norm{\nabla^2\zeta^\md}_{L^2(\Omega)}^2 + \e (A_1+A_2+1) e^{-ct} \Big],
\end{align*}
where
\begin{equation*}
	\begin{aligned}
		E(t) & := A_1 \int_\Omega \big( \abso{p'(\rhot^s)} \abso{\phi^\md}^2 + \abso{\psi^\md}^2 \big) dx \\
		& \qquad  + A_2 \int_\Omega \Big(\frac{\mut}{2\rhot^s} \abso{\nabla \phi^\md}^2 + \psi^\md\cdot \nabla\phi^\md \Big) dx + \int_\Omega \rhot^s \abso{\nabla\zeta^\md}^2 dx,
	\end{aligned}
\end{equation*}
and 
\begin{align*}
	H(t) := a_1 A_1 \norm{\nabla\psi^\md}_{L^2(\Omega)}^2 + a_2 A_2 \norm{\nabla\phi^\md}_{L^2(\Omega)}^2 + a_3 \norm{\nabla^2 \zeta^\md}_{L^2(\Omega)}^2.
\end{align*}
If $ \delta $ and $ \nu $ are small, then one can choose some generic constants $ A_1 > A_2 \geq 1 $ such that
\begin{equation}\label{A}
	\begin{cases}
		\frac{1}{2} a_1 A_1 \geq C(A_2+1), \\
		\frac{1}{2} a_2 A_2 \geq C \big((\delta+\nu) A_1 + 1 \big), \\
		\frac{1}{2} a_3 \geq (\delta+\nu) A_2,
	\end{cases}
\end{equation}
which yields that
\begin{equation}\label{ineq-4}
	E'(t) + \frac{1}{2} H(t) \leq C \e (A_1+A_2+1) e^{-ct}.
\end{equation}
It follows from \cref{zeta-est,Poincare} that
\begin{align*}
	E(t) 
%	& \lesssim A_1 \big(\norm{\phi^\md}_{H^1(\Omega)}^2 + \norm{\psi^\md}_{H^1(\Omega)}^2\big) + \e \nu e^{-ct} \\
	& \lesssim A_1 \big(\norm{\nabla \phi^\md}_{L^2(\Omega)}^2 +  \norm{\nabla\psi^\md}_{L^2(\Omega)}^2 \big)  + \e \nu e^{-ct} \\
	& \lesssim A_1 H(t)  + \e \nu e^{-ct}.
\end{align*}
Then by \cref{ineq-4}, there exist some generic constants $ c_1>0 $ and $ C_2 >0 $ such that
\begin{align*}
	E'(t) + c_1 E(t) \leq C_2 \e e^{-ct}.
\end{align*}
Without loss of generality, we can let $ c_1 < c $. Then it holds that
\begin{equation}\label{ineq-5}
	E(t) \lesssim (\e+\nu) e^{- c_1 t} \qquad \forall t>0.
\end{equation}
Note that one can let $ A_1 $ be suitably larger than $ A_2 $ (which is compatible with \cref{A}) such that 
\begin{align*}
	E(t) & \gtrsim A_1 \big(\norm{\phi^\md}_{L^2(\Omega)}^2 + \norm{\psi^\md}_{L^2(\Omega)}^2\big) + A_2 \norm{\nabla \phi^\md}_{L^2(\Omega)}^2 + \norm{\nabla\zeta^\md}_{L^2(\Omega)}^2.
\end{align*}
This, together with \cref{psi-nab-est,Poincare}, yields that
\begin{equation}\label{ineq-7}
	E(t) \gtrsim \norm{\phi^\md}_{H^1(\Omega)}^2 + \norm{\psi^\md}_{H^1(\Omega)}^2 - \e \nu e^{-ct}.
\end{equation}
%And one can use \cref{psi-zeta} to verify that 
%\begin{align*}
%	\norm{\zeta^\md}_{L^2(\Omega)} \lesssim \norm{\psi^\md}_{L^2(\Omega)} + (\delta+\nu)\norm{\phi^\md}_{L^2(\Omega)} + \e \nu e^{-ct}.
%\end{align*}
Then using \cref{zeta-est,Poincare}, one has that
\begin{align*}
	E(t) \gtrsim  \norm{\phi^\md}_{H^1(\Omega)}^2 + \norm{\zeta^\md}_{H^1(\Omega)}^2 - \e \nu e^{-ct},
\end{align*}
which yields that
\begin{equation}\label{ineq-6}
	\norm{\big(\phi^\md,\zeta^\md\big)}_{H^1(\Omega)}^2 \lesssim (\e+\nu) e^{-c_1 t}.
\end{equation}
Similar to \cref{gn-1-2,gn-3}, one can use \cref{Lem-GN}, \cref{ineq-6,nu-md-2} to get that
\begin{align*}
	\norm{\phi^\md, \zeta^\md}_{L^\infty(\Omega)} \lesssim \nu^{\frac{1}{2}} \norm{\nabla\big(\phi^\md, \zeta^\md\big)}_{L^2(\Omega)}^{\frac{1}{2}} \lesssim (\e+\nu) e^{- \frac{c_1}{2} t}  \qquad \forall t>0.
\end{align*}
Thus, \cref{exp-pert} holds true and the proof of \cref{Thm-pertur} is finished.

\vspace{.5cm}

\textbf{Proof of \cref{Thm}.}
Once \cref{Thm-pertur} is proved, it is straightforward to obtain \cref{Thm} except for \cref{behavior,exp-Thm}. 
In fact, it follows from \cref{ansatz} and Lemmas \ref{Lem-per}, \ref{Lem-shift} and \ref{Lem-shock} that
\begin{align*}
	& \norm{(\rho,\uv)(x,t)-(\rho^s, u_1^s,0,0)(x_1-st-X_\infty)}_{W^{1,\infty}(\R^3)} \\
	& \quad \lesssim \norm{(\phi,\zeta)}_{W^{1,\infty}(\Omega)} + \norm{\rhot(x,t)-\rho^s(x_1-st-X(t))}_{W^{1,\infty}(\Omega)} \\
	& \qquad  + \norm{\uvt(x,t)-\uv^s(x_1-st-Y(t))}_{W^{1,\infty}(\Omega)} + \delta^2 \big(\abso{X(t)-X_\infty}+\abso{Y(t)-X_\infty}\big) \\
	& \quad \lesssim \norm{(\phi,\zeta)}_{W^{1,\infty}(\Omega)} + \delta^{-1} \norm{(v_\pm,\wv_\pm)}_{W^{1,\infty}(\Torus^3)} + \delta \e e^{-ct} \\
	& \quad \lesssim \norm{(\phi,\zeta)}_{W^{1,\infty}(\Omega)} + \gamma_0 e^{-ct}.
\end{align*}
This, together with \cref{asymp}, yields \cref{behavior} immediately.

To prove \cref{exp-Thm}, we firstly note that
\begin{align*}
	(\rho^\md, \uv^\md) = (\rhot^\md, \uvt^\md) + (\phi^\md, \zeta^\md).
\end{align*}
It follows from \cref{ansatz} that
\begin{align*}
	\rhot^\md = \rho^\md_- (1-\eta_{st+X(t)}) + \rho_+^\md \eta_{st+X(t)},
\end{align*}
which satisfies that
\begin{equation}\label{exp-rho}
	\norm{\rhot^\md}_{L^\infty(\Omega)} \lesssim \norm{\rho_-^\md, \rho_+^\md}_{L^\infty(\Omega)} \lesssim \e e^{-ct}.
\end{equation}
It is similar to use \cref{ansatz} to prove that
\begin{equation}\label{exp-m}
	\norm{\mvt^\md}_{L^\infty(\Omega)} \lesssim \norm{\mv_-^\md, \mv_+^\md}_{L^\infty(\Omega)} \lesssim \e e^{-ct}.	
\end{equation}
Since it holds that 
\begin{align*}
	\mvt^\md = (\rhot \uvt)^\md = \rhot \uvt^\md + \rhot^\md \uvt^\od - \int_{\Torus^2} \rhot^\md \uvt^\md d\xp,
\end{align*}
then it follows from \cref{exp-rho,exp-m} that
\begin{align*}
	\norm{\uvt^\md}_{L^\infty(\Omega)} & \lesssim \norm{\mvt^\md}_{L^\infty(\Omega)} + \norm{\rhot^\md}_{L^\infty(\Omega)} \big(\norm{\uvt^\od}_{L^\infty(\R)} + \norm{\uvt^\md}_{L^\infty(\Omega)} \big) \\
	& \lesssim \e e^{-ct}.
\end{align*}
Thus, one can obtain that
\begin{align*}
	\norm{(\rho^\md, \uv^\md)}_{L^\infty(\Omega)} & \lesssim \norm{(\rhot^\md, \uvt^\md)}_{L^\infty(\Omega)} + \norm{(\phi^\md, \zeta^\md)}_{L^\infty(\Omega)} \\
	& \lesssim \e e^{-ct} + (\e+\nu) e^{-\frac{c_1}{2} t} \\
	& \lesssim  (\e+\nu) e^{-\frac{c_1}{2} t}.
\end{align*}
The proof of \cref{Thm} is finished.

\vspace{.5cm}

\appendix

\section{Shift curves}\label{App-shift}

This section is devoted to the proof of \cref{Lem-shift}.

As the periodic solutions $ (\rho_\pm, \mv_\pm) $ belong to the $ W^{4,\infty}(\Torus^3) $ space, then each $ \mathcal{L}_i(d,t) $ in \cref{L} is $ C^4 $ and $ C^2 $ with respect to  $ d $ and $ t $, respectively.
Thus, the existence and uniqueness of the $ C^3 $ solution of the problem \cref{ode-XY}, \cref{ic-XY} can be derived from the Cauchy-Lipschitz theorem. 
It follows from \cref{L,decay-vw} that
\begin{align*}
	\mathcal{L}_1(d,t) & = \jump{m_1} + O(1) \e e^{-ct}, \\
	\mathcal{L}_2(d,t) & = \jump{\rho}  + O(1) \e e^{-ct}, \\
	\mathcal{L}_3(d,t) & = \jump{u_1 m_1} + \jump{p(\rho)}  + O(1) \e e^{-ct},
\end{align*}
which yields \cref{decay-XYp} immediately. 
%And it is similar to prove \cref{decay-XYp} for $ k=2 $ and $ 3. $
Then it remains to prove \cref{decay-XY}, and we only compute the limit of $ Y(t) $, since the other one is similar. 

Following \cite{XYYindi,HY1shock,Yuan2021}, for fixed $ r\in [0,1], $ $ t>0 $ and integer $ N>0 $, we define the compact subset of $ \Omega\times\R_+ $,
\begin{align*}
	\Omega_{(r,t)}^N & = \{ (x, \tau): \Gamma_{r,-}^N(\tau) < x_1 < \Gamma_{r,+}^N(\tau), \xp \in \Torus^2, 0<\tau<t \} \\
	\text{with} \quad \Gamma_{r,\pm}^N(\tau)  & := r+ s\tau + Y(\tau) \pm N.
\end{align*}
Then integrating the second equation of \cref{equ-ansatz}, i.e.
\begin{align*}
	\p_t \mt_1 + \sum_{i=1}^{3} \p_i (\ut_i \mt_1) + \p_1 p(\rhot) - \mu \lap \ut_1 - (\mu+\lambda) \p_1 \dv \ut = -\sum_{i=1}^{3} \p_i F_{3,i1} - f_{4,1},
\end{align*}
over the domain $ \Omega_{(r,t)}^N $ yields that
\begin{equation}\label{equ-app}
	\begin{aligned}
	& \underbrace{\int_{\Gamma_{r,-}^N(0)}^{\Gamma_{r,+}^N(0)} \int_{\Torus^2} \mt_1(x,0) d\xp dx_1 - \int_{\Gamma_{r,-}^N(t)}^{\Gamma_{r,+}^N(t)} \int_{\Torus^2} \mt_1(x,t) d\xp dx_1}_{I_1} \\
	& \qquad + \underbrace{\int_0^t \int_{\Torus^2} \Big[ (s+ Y') \mt_1 - \ut_1 \mt_1 - p(\rhot) + \mut \p_1 \ut_1 \Big]\big(\Gamma_{r,+}^N(\tau), \xp, \tau\big) d\xp d\tau}_{I_2} \\
	& \qquad - \underbrace{\int_0^t \int_{\Torus^2} \Big[ (s+ Y') \mt_1 - \ut_1 \mt_1 - p(\rhot) + \mut \p_1 \ut_1 \Big]\big(\Gamma_{r,-}^N(\tau), \xp, \tau\big) d\xp d\tau}_{I_3} \\
	& \quad = - \int_{\Omega_{(r,t)}^N} \big( \p_1 F_{3,11} + f_{4,1} \big) dxd\tau.
\end{aligned}
\end{equation}
Here we have used the fact that the ansatz is periodic in $ \xp \in \Torus^2. $ Then one can use \cref{Lem-per} and the dominated convergence theorem to prove that
\begin{align}
	\lim_{N\to+\infty} \int_0^1 I_1 dr & = \lim_{N\to+\infty} \int_0^1 \Bigg\{ - \int_{r-N}^{r} \int_{\Torus^2} (w_{1+}-w_{1-})(x_1+st+Y(t), \xp,t) \eta(x_1) dx \notag \\
	& \qquad \qquad\qquad + \int_{r}^{r+N} \int_{\Torus^2} (w_{1+}-w_{1-})(x_1+st+Y(t), \xp,t) (1-\eta) dx \Bigg\} dr \notag \\
	& = O(1) \e\delta^{-1} e^{-ct}. \label{I-1-app}
\end{align}
Note that
\begin{align*}
	\abso{\rhot-\rho_+} \lesssim \delta (1-\eta_{st+Y}), \qquad \abso{\mt_1 - m_{1+}} \lesssim \delta (1-\eta_{st+Y}).
\end{align*}
Then it follows from \cref{decay-vw} that
\begin{align}
	\lim_{N\to+\infty} \int_0^1 I_2 dr & = \int_0^t \int_{\Torus^3} \big[ (s+ Y') m_{1+} - u_{1+} m_{1+} - p(\rho_+) \big](\Gamma_{r,+}^N(\tau),\xp,\tau) dx d\tau \notag \\
	& =  \mb_{1+} \big(st + Y(t) - Y_0 \big) - \big(\ub_{1+} \mb_{1+} + p(\rhob_+)\big) t \notag \\
	& \quad  - \int_0^t \int_{\Torus^3} \big(u_{1+} m_{1+} - \ub_{1+} \mb_{1+} + p(\rho_+) - p(\rhob_+) \big) dx d\tau \notag \\
	& = \mb_{1+} \big(st + Y(t) - Y_0\big) - \big(\ub_{1+} \mb_{1+} + p(\rhob_+)\big) t \notag \\
	& \quad - \int_0^{+\infty} \int_{\Torus^3} \big(u_{1+} m_{1+} - \ub_{1+} \mb_{1+} + p(\rho_+) - p(\rhob_+) \big) dx d\tau + O(1) \e e^{-ct}.   \label{I-2-app}
\end{align}
Similarly, it holds that
\begin{align}
	\lim_{N\to+\infty} \int_0^1 I_3 dr & = \mb_{1-} \big(st + Y(t) - Y_0\big) - \big(\ub_{1-} \mb_{1-} + p(\rhob_-)\big) t \notag \\
	& \quad - \int_0^{+\infty} \int_{\Torus^3} \big(u_{1-} m_{1-} - \ub_{1-} \mb_{1-} + p(\rho_-) - p(\rhob_-) \big) dx d\tau + O(1) \e e^{-ct}. \label{I-3-app}
\end{align}
Since $ F_{3,11} $ vanishes as $ \abso{x_1} \to 0 $ and $ f_{4,1} $ has zero mass on $ \Omega, $ one can get that
\begin{align*}
	\lim_{N\to+\infty} \int_0^1 \int_{\Omega_{(r,t)}^N} \big( \p_1 F_{3,11} + f_{4,1} \big) dxd\tau dr = 0.
\end{align*}
This, together with \cref{I-1-app,I-2-app,I-3-app}, can complete the proof.

\section{Errors of ansatz}\label{App-g}
This section is devoted to the proof of \cref{Lem-F}.

We prove only the $ L^2 $-norms of $ G_2 $ and $ \gv_2, $ since the other estimates are similar.
For convenience, in this section we write $ a \simeq b $ and $ a \cong b $ when $ \norm{a-b}_{L^2(\Omega)} \lesssim \e \delta^{-\frac{1}{2}} e^{-ct} $ and $ \norm{a-b}_{L^2(\Omega)} \lesssim \e \delta^{1/2} e^{-ct} $ hold, respectively. Set
\begin{equation*}
	\etat := \eta_{st+X_\infty} = \eta(x_1-st-X_\infty).
\end{equation*}
Then by \cref{shock-ts}, it holds that
\begin{align*}
	(\rhot^s, \mvt_1^s) & = (\rhob_-,\mvb_-) (1-\etat) + (\rhob_+,\mvb_+) \etat,
\end{align*}
which yields that
\begin{equation}\label{uts}
	\uvt^s = \frac{\mt_1^s}{\rhot^s} = \uvb_- (1-\etat) + \uvb_+ \etat + \frac{1}{\rhot^s} \jump{\rho} \jump{\uv}  \etat (1-\etat).
\end{equation}
It follows from Lemmas \ref{Lem-shift} and \ref{Lem-shock} that
\begin{align*}
	& \norm{\etat (1-\etat)}_{L^2(\R)} \lesssim \delta^{-\frac{1}{2}}, \quad \norm{\etat'}_{L^2(\R)} \lesssim \delta^{\frac{1}{2}}, \quad \norm{\etat''}_{L^2(\R)} \lesssim \delta^{\frac{3}{2}},
\end{align*}
and 
\begin{align*}
	\eta_{st+X} \simeq \eta_{st+Y} \simeq \etat.
\end{align*}
Thus, using \cref{ansatz}, it holds that
\begin{equation}\label{rho-mt}
	\rhot \cong \rho_- (1-\etat) + \rho_+ \etat \quad \text{and} \quad \mvt \cong \mv_- (1-\etat) + \mv_+ \etat.
\end{equation}
And \cref{small-u} yields that
\begin{align}
	\uvt = \frac{\mvt}{\rhot} & \cong \Big( \frac{1}{\rhot} - \frac{1}{\rho_-}\Big) \mv_- (1-\etat) + \Big( \frac{1}{\rhot} - \frac{1}{\rho_+}\Big) \mv_+ \etat + \uv_- (1-\etat) + \uv_+ \etat \notag \\
	& \cong \frac{\rhob_+ - \rhob_-}{\rhot^s} \etat (1-\etat) (\uvb_+ - \uvb_-) + \uv_- (1-\etat) + \uv_+ \etat \notag \\
	& = \uvt^s + \zv_- (1-\etat) + \zv_+ \etat, \label{ut}
\end{align}
where we have used \cref{uts} and the fact that
\begin{align*}
	\frac{\rho_+ - \rho_-}{\rhot} \etat (1-\etat) (\uv_+ - \uv_-) & \cong \frac{\rhob_+ - \rhob_-}{\rhot^s} \etat (1-\etat) (\uv_+ - \uv_-) \\
	& \cong \frac{\rhob_+ - \rhob_-}{\rhot^s} \etat (1-\etat) (\uvb_+ - \uvb_-).
\end{align*}
Similarly, one can prove that for $ i=1,2,3, $
\begin{align*}
	\p_i \rhot \cong \p_i \rho_- (1-\etat) + \p_i \rho_+ \etat, \qquad \p_i \mvt \cong \p_i \mv_- (1-\etat) + \p_i \mv_+ \etat,
\end{align*}
and
\begin{equation}\label{utd}
	\p_i \uvt \cong \p_i \uvt^s + \p_i \zv_- (1-\etat) + \p_i \zv_+ \etat.
\end{equation}

\vspace{.2cm}

1) We first estimate $ G_2 = F_{3,11}+F_{4,1}, $ where $ F_{3,11} $ and $ F_{4,1} :=\int_{-\infty}^{x_1} f_{4,1}^{\od} dy_1 $ are given by \cref{F} and \cref{G}, respectively. 

By the definition of $ \Fv_{3,i} $  in \cref{F}, we write that $ \Fv_{3,i} = \sum_{k=1}^4 I_k^{(i)}, $
%\begin{equation}\label{F-3-11}
%	\begin{aligned}
%		\Fv_{3,i} & \cong u_{i-} \mv_- (1-\etat) +  u_{i+} \mv_+ \etat - \ut_i \mvt \\
%		& \quad + \big[ p(\rho_-) (1-\etat) + p(\rho_+) \etat - p(\rhot)  \big] \E_i \\
%		& \quad  - \mu  \big[ \p_i \uv_{-} (1-\etat) + \p_i \uv_+ \etat - \p_i \uvt \big] \\
%		& \quad - (\mu+\lambda) \big[ \dv \uv_- (1-\etat) + \dv \uv_+ \etat - \dv \uvt \big] \E_i \\
%		& := \sum_{k=1}^4 I_k.
%	\end{aligned}
%\end{equation}
where each $ I_k^{(i)} $ denotes the $ k $-th line of the right-hand side of the definition. 
Using \cref{Lem-per}, \cref{rho-mt,ut}, one has that
\begin{align}
	I_1^{(i)} & \cong (u_{i-}-\ut_i) \mv_- (1-\etat) + (u_{i+}-\ut_i) \mv_+ \etat \notag \\
	& \cong (u_{i-} - \ut_i^s - z_{i-} ) \mv_- (1-\etat) + (u_{i+} - \ut_i^s - z_{i+} ) \mv_+ \etat \notag \\
%	& = (\ub_{i-} - \ut_i^s) \mv_- (1-\etat) + (\ub_{i+} - \ut_i^s) \mv_+ \etat \notag \\
	& \cong (\ub_{i-} - \ut_i^s) \mvb_- (1-\etat) + ( \ub_{i+} - \ut_i^s ) \mvb_+ \etat \notag \\
	& = \delta_{1i} \big[ \ub_{1-} \mb_{1-} (1-\etat) + \ub_{1+} \mb_{1+} \etat - \ut_1^s \mt_1^s \big] \E_1. \label{l-1}
\end{align} 
For the pressure law, define
\begin{align*}
	\sigma(\rho,\rhot) := \int_0^1 p'(\rhot + r (\rho-\rhot)) dr.
\end{align*}
Then it follows from \cref{Lem-per} and \cref{rho-mt} that
\begin{align}
	I_2^{(i)} & \cong \big[\sigma(\rho_-, \rhot) (\rho_- -\rhot) (1-\etat) + \sigma(\rho_+, \rhot) (\rho_+ -\rhot) \etat\big] \E_i \notag \\
	& \cong \big[\sigma(\rhob_-, \rhot^s) (\rhob_- - \rhot^s) (1-\etat) + \sigma(\rhob_+, \rhot^s) (\rhob_+ - \rhot^s) \etat\big] \E_i \notag \\
	& = \big[p(\rhob_-) (1-\etat) + p(\rhob_+) \etat - p(\rhot^s)\big] \E_i, \label{l-2}
\end{align}
where the fact that $ \abso{\sigma(\rho_-, \rhot) - \sigma(\rho_+, \rhot)} \lesssim \delta $ is used.
And using \cref{ut,utd}, one can get that
\begin{equation}\label{l-3-4-5}
	\begin{aligned}
	I_3^{(i)} & \cong - \mu  \big[ \p_i \zv_- (1-\etat) + \p_i \zv_+ \etat - \p_i \uvt \big] \notag \\
	& \cong \delta_{1i} \mu \big[  \p_1 \ut_1^s \E_1 + (\zv_+ - \zv_- ) \p_i\etat \big] \\
	& \cong \delta_{1i} \mu \p_1 \ut_1^s \E_1, \\
	I_4^{(i)} & \cong - (\mu+\lambda)  \big[\dv \zv_- (1-\etat) + \dv \zv_+ \etat - \dv \uvt \big] \E_i \\
	& \cong (\mu+\lambda) \p_1 \ut_1^s \E_i.
	\end{aligned}
\end{equation}
Collecting \cref{l-1,l-2,l-3-4-5} yields that
\begin{equation}\label{F-3i}
	\Fv_{3,i} \cong A_-^{(i)} (1-\etat) + A_+^{(i)} \etat - B^{(i)},
\end{equation}
where 
\begin{align*}
	A_-^{(i)} & = \delta_{1i} \ub_{1-} \mb_{1-}  \E_1 + p(\rhob_-) \E_i, \\
	A_+^{(i)} & = \delta_{1i} \ub_{1+} \mb_{1+} \E_1 + p(\rhob_+) \E_i, \\
	B^{(i)} & = \delta_{1i} \ut_1^s \mt_1^s \E_1 + p(\rhot^s) \E_i  - \p_1 \ut_1^s \big[ \delta_{1i} \mu \E_1 + (\mu+\lambda) \E_i\big].
\end{align*}
Thus, one has that
\begin{equation}\label{F-311}
	F_{3,11} \cong a_{3,-} (1-\etat) + a_{3,+} \etat - b_3, 
\end{equation}
where $ a_{3,-} = \ub_{1-} \mb_{1-} + p(\rhob_-) $, $ a_{3,+} = \ub_{1+} \mb_{1+} + p(\rhob_+) $ and $ b_3 = \ut_1^s \mt_1^s + p(\rhot^s)  - \mut  \p_1 \ut_1^s $.

\vspace{.2cm}

For $ F_{4,1}, $ it follows from \cref{F,rho-mt,ut} that
\begin{equation}\label{F-4}
	F_{4,1} = \int_{-\infty}^{x_1} \int_{\Torus^2} f_{4,1}(y_1, \xp,t) d\xp dy_1 = a_4 \eta_{st+Y} + R_-,	\quad x_1 \in\R,t>0,
\end{equation}
where 
\begin{align*}
	a_4 = s \jump{m_1} - \jump{u_1m_1} - \jump{p}
\end{align*}
and the remainder $ R_- $ satisfies that
\begin{align*}
	\Big(\int_{-\infty}^{st} \abso{R_-}^2 dx_1\Big)^{\frac{1}{2}} \lesssim \e e^{-ct} \Big(\int_{-\infty}^{st} \abso{\eta_{st+Y}}^2 dx_1\Big)^{\frac{1}{2}} \lesssim \e \delta^{-1/2} e^{-ct}.
\end{align*}
Meanwhile, by \cref{zero-mass-f}, one has that
\begin{equation}\label{F-4-1}
	F_{4,1} = -\int_{x_1}^{+\infty} \int_{\Torus^2} f_{4,1}(y_1, \xp,t) d\xp dy_1  = - a_4 (1-\eta_{st+Y}) + R_+, \quad x_1 \in\R,t>0,
\end{equation}
where the remainder $ R_+ $ satisfies that
\begin{align*}
	\Big(\int_{st}^{+\infty} \abso{R_+}^2 dx_1\Big)^{\frac{1}{2}} \lesssim \e e^{-ct} \Big(\int_{st}^{+\infty} \abso{1-\eta_{st+Y}}^2 dx_1\Big)^{\frac{1}{2}} \lesssim \e \delta^{-1/2} e^{-ct}.
\end{align*}
Thus, one has that
\begin{align*}
	F_{4,1} \cong a_4 \etat + R_- \cong - a_4 (1-\etat) + R_+.
\end{align*}
It follows from \cref{G}$ _2, $ \cref{F-311,F-4} that
\begin{align}
	\Big(\int_{-\infty}^{st} \abso{G_2}^2 dx_1\Big)^{\frac{1}{2}} & \lesssim \Big(\int_{-\infty}^{st} \abso{F_{3,11} + F_{4,1}}^2 dx_1\Big)^{\frac{1}{2}} \notag \\
	& \lesssim \Big(\int_{-\infty}^{st} \abso{(a_{3,-}-b_3)+ (a_{3,+}-a_{3,-}+a_4)\etat }^2 dx_1\Big)^{\frac{1}{2}} + \e \delta^{-1/2} e^{-ct} \notag \\
	& \lesssim \e \delta^{-1/2} e^{-ct}, \label{G-2-minus}
\end{align}
where we have used the fact that
\begin{align*}
	& (a_{3,-}-b_3)+ (a_{3,+}-a_{3,-}+a_4)\etat \\
	& \qquad = \ub_{1-} \mb_{1-} - \ut_1^s \mt_1^s + p(\rhob_-) - p(\rhot^s) + \mut (\ut_1^s)' + s ( \mt_1^s - \mb_{1-}) \\
	& \qquad = 0.
\end{align*}
Similarly, one can use \cref{F-311,F-4-1} to prove that
\begin{align*}
	\Big(\int_{st}^{+\infty} \abso{G_2}^2 dx_1\Big)^{\frac{1}{2}}  & \lesssim \Big(\int_{st}^{+\infty} \abso{(a_{3,+}-b_3) - (a_{3,-} + a_{3,+} - a_4)(1-\etat) }^2 dx_1\Big)^{\frac{1}{2}} + \e \delta^{-1/2} e^{-ct} \notag \\ 
	& \lesssim \e \delta^{-1/2} e^{-ct},
\end{align*}
which, together with \cref{G-2-minus}, yields that $ G_2 \simeq 0. $

\vspace{.3cm}

2) Then we estimate $ \gv_2 $ in \cref{equ-phipsi}$ _2 $.
Similar to \cref{F-3i}, one can prove that
\begin{align*}
	\p_i \Fv_{3,i} & \cong \p_i \big( A_-^{(i)} (1-\etat) + A_+^{(i)} \etat - B^{(i)} \big) \\
	& = \delta_{1i} \big(A_+^{(i)} - A_-^{(i)} \big) \p_1 \etat  - \p_i B^{(i)} \\
	& = \delta_{1i} \big[ \big(\jump{u_1 m_1} + \jump{p} \big) \etat' - (\ut_1^s \mt_1^s)' - \big(p(\rhot^s)\big)' + \mut (\ut_1^s)'' \big] \E_1 \\
	& = 0,
\end{align*} 
where we have used \cref{RH}$ _2 $ and \cref{ode-shock}$ _2 $.
Moreover, it follows from \cref{RH}$ _2 $ and Lemmas \ref{Lem-per} and \ref{Lem-shift} that
\begin{align*}
	\fv_4 \cong \big(\jump{m_1} s - \jump{u_1 m_1} - \jump{p}\big) \etat' \E_1 = 0.
\end{align*}
Thus, one has that $ \gv_2 \cong 0. $

\vspace{.4cm}

\textbf{Conflict of interest statement.} The author has no conflict of interest to declare.

%%%%%%%%%%%%%%%%%%%%%%%%%%%%%%%%%%%%%%%%%%%
%%%%% Bib
%%%%%%%%%%%%%%%%%%%%%%%%%%%%%%%%%%%%%%%%%%%

\providecommand{\bysame}{\leavevmode\hbox to3em{\hrulefill}\thinspace}
\providecommand{\MR}{\relax\ifhmode\unskip\space\fi MR }
\providecommand{\MRhref}[2]{%
	\href{http://www.ams.org/mathscinet-getitem?mr=#1}{#2}
}
\providecommand{\href}[2]{#2}

%\bibliographystyle{amsplain}
%\bibliography{bibli}

\begin{thebibliography}{1}
	
%	\bibitem{Alinhac}
%	S. Alinhac, Existence d'ondes de rar\'efaction pour des syst\`emes quasi-lin\'eaires hyperboliques multidimensionnels,  \emph{Comm. Partial Differential Equations} \textbf{14}: no.2, 173--230 (1989)
	
%	\bibitem{BK2016} Barker, B., Zumbrun, K.: Numerical proof of stability of viscous shock profiles. \emph{Math. Models Methods Appl. Sci.} \textbf{26}(13), 2451--2469 (2016)
	
	\bibitem{Chio2014}
	Chiodaroli, E.: A counterexample to well-posedness of entropy solutions to the compressible Euler system.  \emph{J. Hyperbolic Differ. Equ.} \textbf{11}(3), 493--519 (2014)
	
	\bibitem{CDK2015}
	Chiodaroli, E., De Lellis C. and Kreml O.: Global ill-posedness of the isentropic system of gas dynamics. \emph{Comm. Pure Appl. Math.} \textbf{68}(7), 1157--1190 (2015)
	
%	\bibitem{Coulomb2004}
%	J. F. Coulombel and P. Secchi, The stability of compressible vortex sheets in two space dimensions, \emph{Indiana Univ. Math. J.} \textbf{53}: no. 4, 941--1012 (2004)
	
%	\bibitem{Coulomb2008}
%	J. F. Coulombel and P. Secchi, Nonlinear compressible vortex sheets in two space dimensions.  \emph{Ann. Sci. \'Ec. Norm. Sup\'er (4)} \textbf{4}: no. 1, 85--139 (2008)

	\bibitem{Dafe1995} Dafermos, C. M.: Large time behavior of periodic solutions of hyperbolic systems of conservation laws. \emph{J. Differential Equations} \textbf{121}(1), 183--202  (1995)
	
	\bibitem{Dafe2013} \bysame: Long time behavior of periodic solutions to scalar conservation laws in several space dimensions. \emph{SIAM J. Math. Anal.} \textbf{45}(4), 2064--2070 (2013)
%	
%	\bibitem{DS2010}
%	C. De Lellis and L. Sz\'ekelyhidi,
%	On admissibility criteria for weak solutions of the Euler equations.
%	\emph{Arch. Ration. Mech. Anal.} \textbf{195}: no. 1, 225--260 (2010)
	
%	\bibitem{DS2009}
%	 C. De Lellis and L. Sz\'ekelyhidi, The Euler equations as a differential inclusion, \emph{Ann. of Math. (2)}, \textbf{170}: no. 3, 1417--1436 (2009)
	
	\bibitem{FS2010}
	Freistühler, H., Szmolyan, P.: Spectral stability of small-amplitude viscous shock waves in several space dimensions. \emph{Arch. Ration. Mech. Anal.}, \textbf{195}(2), 353--373 (2010)

	\bibitem{Glimm1965} Glimm, J.: Solutions in the large for nonlinear hyperbolic systems of equations. \emph{Comm. Pure Appl. Math.} \textbf{18}, 697--715  (1965)
	
	\bibitem{Glimm1970}
	Glimm, J. and  Lax, P. D.: Decay of solutions of systems of nonlinear hyperbolic conservation laws, Memoirs of the American Mathematical Society, no. 101 \emph{American Mathematical Society, Providence, R.I.} (1970)
	
	\bibitem{G1986} Goodman, J.: Nonlinear asymptotic stability of viscous shock profiles for conservation laws, \emph{Arch. Ration. Mech. Anal.} \textbf{95}(4), 325--344 (1986)
	
	\bibitem{GM1999} Goodman, J., Miller, J. R.: Long-time behavior of scalar viscous shock fronts in two dimensions. \emph{J. Dynam. Differential Equations} \textbf{11}(2), 255--277  (1999)
	
	\bibitem{HZ2000}  Hoff, D., Zumbrun, K.: Asymptotic behavior of multidimensional scalar viscous shock fronts, \emph{Indiana Univ. Math. J.} \textbf{49}(2), 427--474 (2000)
	
%	\bibitem{HZ2002} \bysame: Pointwise Green's function bounds for multidimensional scalar viscous shock fronts.
%	\emph{J. Differential Equations} \textbf{183}(2), 368--408 (2002) 
	
%	\bibitem{HoM} H. Hokari and A. Matsumura, Asymptotics toward one-dimensional rarefaction wave for the solution of two-dimensional compressible Euler equation with an artificial viscosity, \emph{Asymptot. Anal.} \textbf{15}: no. 3-4, 283--298 (1997)
	
%	\bibitem{HRZ2006} Howard, P., Raoofi, M., Zumbrun, K.: Sharp pointwise bounds for perturbed viscous shock waves. \emph{J. Hyperbolic Differ. Equ.} \textbf{3}(2), 297--373 (2006)
	
%	\bibitem{HLM} F.  Huang, J. Li and A. Matsumura, Asymptotic stability of combination of viscous contact wave with rarefaction waves for one-dimensional compressible Navier-Stokes system, \emph{Arch. Ration. Mech. Anal.} \textbf{197}: no. 1, 89--116 (2010)
	
%	\bibitem{HM2009} Huang F., Matsumura A.:  Stability of a composite wave of two viscous shock waves for the full compressible Navier-Stokes equation, \emph{Comm. Math. Phys.} \textbf{289}(3), 841--861 (2009)
	
%	\bibitem{HMS} F.  Huang, A. Matsumura and X. Shi, On the stability of contact discontinuity for compressible Navier-Stokes equations with free boundary, \emph{Osaka J. Math.} \textbf{41}: no. 1, 193--210 (2004)
	
%	\bibitem{HMX} F.  Huang, A. Matsumura and Z. Xin, Stability of contact discontinuities for the 1D compressible Navier-Stokes equations, \emph{Arch. Ration. Mech. Anal.}  \textbf{179}: no. 1, 55-77  (2006)
%	
%	\bibitem{HXY} F.  Huang, Z. Xin and T. Yang, Contact discontinuity with general perturbations for gas motions, \emph{Adv. Math.} \textbf{219}: no.4, 1246--1297  (2008)
%	
%	\bibitem{HY} F.  Huang and T. Yang, Stability of contact discontinuity for the Boltzmann equation, \emph{J. Differential Equations} \textbf{229}: no. 2, 698--742  (2006)

	\bibitem{HXY2022}
	Huang, F., Xu, L., Yuan, Q.: Asymptotic stability of planar rarefaction waves under periodic perturbations for 3-d Navier-Stokes equations. \emph{Adv. Math.} \textbf{404}, Paper No. 108452, 27 pp. (2022)
	
	\bibitem{HY1shock} 
	Huang, F., Yuan, Q.: Stability of large-amplitude viscous shock under periodic perturbation for 1D isentropic Navier-Stokes equations. \emph{Comm. Math. Phys.} \textbf{387}(3), 1655--1679 (2021)
	
	\bibitem{HY2020}
	\bysame: Stability of planar rarefaction waves for scalar viscous conservation law under periodic perturbations. \emph{Methods Appl. Anal.} \textbf{28}(3), 337--353 (2021)
	
	
%	\bibitem{HZ}  F. Huang and H. J. Zhao, On the global stability of contact discontinuity for compressible Navier-Stokes equations, Rend. \emph{Sem. Mat. Univ. Padova.} \textbf{109}: 283--305 (2003)
	
		
%	\bibitem{Ito1996} K. Ito, Asymptotic decay toward the planar rarefaction waves of solutions for viscous conservation laws in several space dimensions, \emph{Math. Models Methods Appl. Sci.} \textbf{6}: no. 3, 315--338 (1996)
	
%	\bibitem{HLZ2010} Humpherys, J., Lafitte, O.,  Zumbrun, K.: Stability of isentropic Navier-Stokes shocks in the high-Mach number limit. \emph{Comm. Math. Phys.} \textbf{293}(1), 1--36  (2010)
	
	\bibitem{HLZ2017} Humpherys, J., Lyng, G.,  Zumbrun, K.: Multidimensional stability of large-amplitude Navier-Stokes shocks. \emph{Arch. Ration. Mech. Anal.} \textbf{226}(3), 923--973  (2017) 
	
	\bibitem{Ilin1960}  Il'in, A. M.,  Ole\v{\i}nik, O. A.: Asymptotic behavior of solutions of the Cauchy problem for some quasi-linear equations for large values of the time, \emph{Mat. Sb. (N.S.)} \textbf{51}(93) 191--216 (1960)
	
%	\bibitem{KNM2004} S. Kawashima, S. Nishibata and M. Nishikawa, $L^p$ energy method for multi-dimensional viscous conservation laws and application to the stability of planar waves, \emph{J. Hyperbolic Differ. Equations} \textbf{1}: no. 3, 581--603 (2004)

 	\bibitem{KVW2019} Kang, M.-J.,  Vasseur, A. F.,  Wang, Y.: L2-contraction of large planar shock waves for multi-dimensional scalar viscous conservation laws. \emph{J. Differential Equations} \textbf{267}(5), 2737--2791 (2019)
	
%	\bibitem{Kawa2002}  Kawashita M.: On global solutions of Cauchy problems for compressible Navier-Stokes equations. Nonlinear Analysis, Theory, Methods and Applications \textbf{48}(8), 1087--1105 (2002)
	
%	\bibitem{KKMM2020}
%	 Klingenberg, C., Kreml, O.,  M\'acha, V., Markfelder, S.: Shocks make the Riemann problem for the full Euler system in multiple space dimensions ill-posed, \emph{Nonlinearity} \textbf{33}(12), 6517--6540 (2020)
	
	\bibitem{Lax1957}
	 Lax, P. D.: Hyperbolic systems of conservation laws ii,  \emph{Comm. Pure Appl. Math} \textbf{10}(4), 537--566 (1957)
	
%	\bibitem{LW} L. Li and Y, Wang, Stability of planar rarefaction wave to two-dimensional compressible Navier-Stokes equations, \emph{SIAM J. Math. Anal.} \textbf{50}: no. 5, 4937--4963 (2018)
%	
%	\bibitem{LWW}L. Li, T. Wang and Y, Wang, Stability of planar rarefaction wave to 3D full compressible Navier-Stokes equations, \emph{Arch. Ration. Mech. Anal.} \textbf{230}: no. 3, 911--937 (2018)
	
	\bibitem{Liu1977} Liu, T.-P.: Linear and nonlinear large-time behavior of solutions of general systems of hyperbolic conservation laws, \emph{Commun. Pure Appl. Math.} \textbf{30}(6), 767--796 (1977).
	
%	\bibitem{Liu1} T.-P. Liu, Nonlinear stability of shock waves for viscous conservation laws, \emph{Mem. Amer. Math. Soc.} \textbf{56} (1985)
	
	\bibitem{Liu1997} \bysame: Pointwise convergence to shock waves for viscous conservation laws, \emph{Comm. Pure Appl. Math.} \textbf{50}(11), 1113--1182  (1997)
	
%	\bibitem{Liu-Xin} T.-P. Liu, Z. Xin, Pointwise decay to contact discontinuities for systems of viscous conservation laws,  \emph{Asian J. Math.} \textbf{1}: no. 1, 34--84 (1997)

%	\bibitem{LY2018} Liu, T.-P.,  Yu, S.-H.: Multi-dimensional wave propagation over a Burgers shock profile. \emph{Arch. Ration. Mech. Anal.} \textbf{229}(1), 231--337 (2018)
	
	\bibitem{LZ2015} Liu, T.-P., Zeng, Y., Shock waves in conservation laws with physical viscosity. \emph{Mem. Amer. Math. Soc.} \textbf{234}(1105), vi+168 pp  (2015)
	
	\bibitem{Majda1983s}   Majda, A.: The stability of multi-dimensional shock fronts, \emph{Mem. Amer. Math. Soc.} \textbf{41} (1983)
	
%	\bibitem{Majda1983e}  A. Majda, The existence of multi-dimensional shock fronts, \emph{Mem. Amer. Math. Soc.} \textbf{43} (1983)
	
	\bibitem{MR1984}  Majda, A., Rosales, R.: Resonantly interacting weakly nonlinear hyperbolic waves. I. A single space variable, \emph{Stud. Appl. Math.} \textbf{71}(2), 149--179 (1984)
	
	\bibitem{HMR1986} Hunter, J. K., Majda, A., Rosales, R.: Resonantly interacting, weakly nonlinear hyperbolic waves. II. Several space variables. \emph{Stud. Appl. Math.} \textbf{75}(3), 187--226  (1986)
	
	\bibitem{MZ2004}  Mascia, C.,  Zumbrun, K.: Stability of large-amplitude viscous shock profiles of hyperbolic-parabolic systems. \emph{Arch. Ration. Mech. Anal.} \textbf{172}(1), 93--131 (2004)
	
%	\bibitem{MZ2004cpam} \bysame: Stability of small-amplitude shock profiles of symmetric hyperbolic-parabolic systems. \emph{Comm. Pure Appl. Math.} \textbf{57}(7), 841--876 (2004)
	
	\bibitem{MN1980}  Matsumura, A.,  Nishida, T.: The initial value problem for the equations of motion of viscous and heat-conductive gases, J. Math. Kyoto Univ. \textbf{20}(1), 67--104  (1980)
	
	\bibitem{MN}  Matsumura, A.,  Nishihara, K.: On the stability of travelling wave solutions of a one-dimensional model system for compressible viscous gas, \emph{Japan J. Appl. Math.} \textbf{2}(1), 17--25 (1985)
	
%	\bibitem{Matsumura1986}
%	A. Matsumura and K. Nishihara, Asymptotics toward the rarefaction waves of the solutions of a one-dimensional model system for compressible viscous gas, \emph{Japan J. Appl. Math.} \textbf{3}: no. 1, 1--13 (1986)
	
%	\bibitem{MN3} A. Matsumura and K. Nishihara, Global stability of the rarefaction wave of a one dimensional model system for compressible viscous gas, \emph{Comm. Math. Phys.} \textbf{144}: no. 2, 325--335 (1992)
	
	\bibitem{MK2018}
	 Markfelder, S., Klingenberg, C.: The Riemann problem for the multidimensional isentropic system of gas dynamics is ill-posed if it contains a shock \emph{Arch. Ration. Mech. Anal.} \textbf{227}(3), 967--994 (2018)
	 
	 \bibitem{QX2015} Qu, P., Xin, Z.: Long time existence of entropy solutions to the one-dimensional non-isentropic Euler equations with periodic initial data. \emph{Arch. Ration. Mech. Anal.} \textbf{216}(1), 221--259 (2015)
	 
	 \bibitem{TY1996} Temple, B., Young, R.: The large time stability of sound waves. \emph{Comm. Math. Phys.} \textbf{179}(2), 417--466 (1996)
	 
	\bibitem{TY2023} \bysame: The Nonlinear Theory of Sound.  \emph{ArXiv:2305.15623}, 76 pages (2023)
	 
	 \bibitem{Smoller1994} Smoller, J.: Shock waves and reaction-diffusion equations. Second edition. Grundlehren der mathematischen Wissenschaften [Fundamental Principles of Mathematical Sciences], \textbf{258}. \emph{Springer-Verlag, New York}, xxiv+632 pp (1994)
	
%	\bibitem{NN2000}
%	M. Nishikawa and K. Nishihara, Asymptotics toward the planar rarefaction wave for viscous conservation law in two space dimensions, \emph{Trans. Amer. Math. Soc.} \textbf{352}: no. 3, 1203--1215 (2000)
	
%	\bibitem{NYZ} K. Nishihara, T. Yang and H. J. Zhao, Nonlinear stability of strong rarefaction waves for compressible Navier-Stokes equations, \emph{SIAM J. Math. Anal.} \textbf{35}: no. 6, 1561--1597  (2004)

%	\bibitem{Satt1976} D. H. Sattinger.	On the stability of waves of nonlinear parabolic systems. \emph{Advances in Math.} \textbf{22}: no. 3, 312?355 (1976)
	
	\bibitem{SX}  Szepessy, A.,   Xin, Z.: Nonlinear stability of viscous shock waves, \emph{Arch. Ration. Mech. Anal.} \textbf{122}(1), 53--103  (1993)
	
	\bibitem{WW2022} Wang, T., Wang, Y.:  Nonlinear stability of planar viscous shock wave to three-dimensional compressible Navier-Stokes equations, to appear in \emph{J. Eur. Math. Soc.}, \emph{arXiv:2204.09428} (2022)
	
%	\bibitem{WW2015} T. Wang, Y. Wang. Stability of superposition of two viscous shock waves for the Boltzmann equation,	\emph{SIAM J. Math. Anal.} \textbf{47}: no. 2, 1070--1120  (2015)
	
%	\bibitem{Xin1990}
%	 Xin Z.: Asymptotic stability of planar rarefaction waves for viscous conservation laws in several dimensions, \emph{Trans. Amer. Math. Soc.} \textbf{319}(2), 805--820  (1990)
	
	\bibitem{XYYindi}
	 Xin, Z.,  Yuan, Q.,  Yuan, Y.: Asymptotic stability of shock profiles and rarefaction waves under periodic perturbations for 1-D convex scalar viscous conservation laws. \emph{Indiana Univ. Math. J.} \textbf{70}(6), 2295--2349 (2021)
	
	\bibitem{XYYsiam}  \bysame: Asymptotic stability of shock waves and rarefaction waves under periodic perturbations for 1-d convex scalar conservation laws, \emph{SIAM J. Math. Anal.} \textbf{51}(4), 2971--2994 (2019)
%	
%	 
	 \bibitem{Yuan2021}
	  Yuan, Q.: Planar viscous shocks with periodic perturbations for scalar multi-dimensional viscous conservation laws, \emph{SIAM J. Math. Anal.} \textbf{55}(3), 1499--1523 (2023)
	  	
%	
	\bibitem{YY2019}
	 Yuan, Q.,  Yuan, Y.: On Riemann solutions under different initial periodic perturbations at two infinities for 1D scalar convex conservation laws, \emph{J. Differential Equations} \textbf{268}(9), 5140--5155 (2019)
	
	\bibitem{YY2022}
	 \bysame: Periodic perturbations of a composite wave of two viscous shocks for 1D full compressible Navier-Stokes equations, \emph{SIAM J. Math. Anal.} \textbf{54}(3), 2876--2905  (2022)
	
%	\bibitem{Z}  Zumbrun K.: Multidimensional stability of planar viscous shock waves. \emph{Advances in the theory of shock waves}, 307--516, Progr. Nonlinear Differential Equations Appl. \textbf{47}, \emph{Birkhäuser Boston, Boston, MA} (2001) 
	
	\bibitem{Z}
	Zumbrun, K:	Stability of large-amplitude shock waves of compressible Navier-Stokes equations.
	With an appendix by Helge Kristian Jenssen and Gregory Lyng. \emph{Handbook of mathematical fluid dynamics. Vol. III, 311--533, North-Holland, Amsterdam} (2004)
	
%	\bibitem{ZH1998}  Zumbrun, K.,  Howard, P.: Pointwise semigroup methods and stability of viscous shock waves. \emph{Indiana Univ. Math. J.} \textbf{47}(3), 741--871 (1998)

	
\end{thebibliography}

%\vspace{1.5cm}

\end{document}